\documentclass[oneside,reqno,12pt]{amsart}
\usepackage{graphicx, amssymb, color}
\usepackage{amsmath}
\usepackage{graphicx}
\usepackage{enumerate}

\usepackage{marginnote}
\usepackage{todonotes}
\usepackage{mathrsfs}
\usepackage{hyperref}
\usepackage{pdfsync} 

\newtheorem{thm}{Theorem}[section]
\newtheorem{cor}[thm]{Corollary}
\newtheorem{lem}[thm]{Lemma}
\newtheorem{prop}[thm]{Proposition}
\theoremstyle{definition}
\newtheorem{defn}[thm]{Definition}

\newtheorem{ass}[thm]{Assumption}

\theoremstyle{remark}
\newtheorem{rem}[thm]{Remark}
\newtheorem{exa}[thm]{Example}

\newcommand{\norm}[1]{\big\|#1\big\|}

\newcommand{\Real}{\mathbb R}

\newcommand{\F}{\mathcal{F}}

\newcommand{\prob}{\mathbb{P}}

\newcommand{\expec}{\mathbb{E}}

\newcommand{\Bu}{\boldsymbol u}

\newcommand{\esssup}{\text{esssup}}

\newcommand{\Rt}{[0,1]\times \Real^J}

\newcommand{\tsum}{\textstyle \sum}
\newcommand{\tint}{\textstyle \int}
\newcommand{\tiint}{\textstyle \iint}
\newcommand{\lec}{\leq_C}
\newcommand{\len}{\leq_N}

\newcommand{\E}{\mathbb{E}}
\newcommand{\Id}{\mathrm{d}}
\newcommand{\R}{\mathbb{R}}
\newcommand{\N}{\mathbb{N}}

\def\Xint#1{\mathchoice
{\XXint\displaystyle\textstyle{#1}}%
{\XXint\textstyle\scriptstyle{#1}}%
{\XXint\scriptstyle\scriptscriptstyle{#1}}%
{\XXint\scriptscriptstyle\scriptscriptstyle{#1}}%
\!\int}
\def\XXint#1#2#3{{\setbox0=\hbox{$#1{#2#3}{\int}$ }
\vcenter{\hbox{$#2#3$ }}\kern-.6\wd0}}

\def\dashint{\Xint-}

\begin{document}
\date{\today}

\title[Radner equilibrium and quadratic BSDE systems]{Radner equilibrium and systems of quadratic BSDEs with discontinuous generators}\thanks{We are grateful for Johannes Muhle-Karbe for helpful comments on the paper.}

\author[]{Luis Escauriaza}
\address{Departamento de Matem\'{a}ticas, UPV/EHU, Barrio de Sarriena s/n, 48940 Leioa, Spain}
\email{luis.escauriaza@ehu.eus}
\thanks{L. Escauriaza is supported by Basque Government grant IT1247-19 and MICINN grant PGC2018-094522-B-I00.}

\author{Daniel C. Schwarz}
\address{Mathematics Department, University College London, 25 Gordon Street, London WC1H 0AY, UK}
\email{d.schwarz@ucl.ac.uk}

\author{Hao Xing}
\address{Department of Finance, Questrom School of Business, Boston University, 595 Commonwealth Ave, Boston MA 02215, USA}
\email{haoxing@bu.edu}

\keywords{Radner equilibrium, incomplete market, backward stochastic differential equation, discontinuous generator, backward uniqueness}

\begin{abstract}
Motivated by an equilibrium problem, we establish the existence of a solution for a family of Markovian backward stochastic differential equations with quadratic nonlinearity and discontinuity in $Z$.  Using unique continuation and backward uniqueness, we show that the set of discontinuity has measure zero. In a continuous-time stochastic model of an endowment economy, we  prove the existence of  an incomplete Radner equilibrium with  nondegenerate endogenous volatility.
\end{abstract}

\maketitle

\section{Introduction}

\noindent{\bf The equilibrium problem.} Equilibrium is a fundamental concept in economics. It determines asset prices in markets so that supply and demand are matched when agents trade optimally. A milestone in the development of the theory was Radner's \cite{Rad82} extension of the classical framework of Arrow-Debreu. This extension incorporates market incompleteness in equilibrium models. The present paper focuses on an important source of market incompleteness: the number of sources of randomness in the economic environment is larger than the number of  risky assets, so that agents cannot fully hedge the risk they face by trading in the market. In a continuous-time stochastic model of an endowment economy, we prove the existence of an incomplete Radner equilibrium  whose asset price volatility is determined endogenously. 

 Study on general equilibrium with incomplete markets (GEI) has a long tradition in the economics literature. { The incomplete market structure is valued in the economic literature as an implicit model for the consequences of bounded rationality and the opportunistic behaviour of agents (cf. \cite[p.30]{MM-book}).} 
 For discrete time GEI models with finite sample space, we refer reader to the survey \cite{G90} and the textbook \cite{MM-book}. In particular, Duffie and Shafer show in \cite{DS85, DS86}  that equilibrium exists for generic endowments. For continuous time models, Anderson and Raimondo emphasize in \cite{AR} that ``with dynamic incompleteness, essentially nothing is known". In recent years, some progress has been made by the mathematical finance community, cf. \cite{Zit12, Lar12, Zha12, Lar14, ChoLar14, Kardaras-Xing-Zitkovic, Lar16, Jarrow, Weston20}. Nevertheless, general existence results for equilibrium in endowment economy, similar to those considered by the aforementioned economists, still remain missing.

We consider a model of a financial market consisting of one riskless asset (bond) in zero net supply and one risky asset (stock) in unit supply. The stock pays a random dividend at the time horizon, normalized to be $1$.\footnote{This normalization simplifies the notation. All of our results hold for any finite horizon $T\geq 0$.} A finite population of CARA agents trade both assets in order to maximize their expected utility of terminal consumption.\footnote{{ The benefit of CARA utility is highlighted in the aforementioned literature on dynamic incomplete equilibria. The wealth-independence of the agent's portfolio choice enables the characterization of equilibria via systems of BSDEs or quasilinear PDEs.}} Agents consume the profit or loss that results from their dynamic trading and the random endowments they receive at the time horizon. One example of a random endowment is  an agent's labor income, which cannot be fully insured by dynamic trading due to moral hazard considerations. Therefore we allow more than one source of randomness to define the random endowments of the agents and the stock dividend. When CARA agents consume intertemporally, equilibrium problems of a similar type were considered in \cite{Lar12}, \cite{Lar14}, and \cite{Lar16}, where the existence of equilibrium is established under a linear structure for the stock dividend and the random endowment. Whether equilibrium exists under general conditions still remains an open problem, which will be addressed in the present paper. 

\medskip

\noindent{\bf Backward stochastic differential equations.} Since the seminal paper  \cite{Pardoux-Peng}, backward stochastic differential equations (BSDEs) have been a subject of extensive study. Given a time horizon (normalized to be $1$) and a filtered probability space $(\Omega, \mathcal{F}, (\mathcal{F}_t)_{t\in [0,1]}, \prob)$ satisfying the usual conditions, a BSDE is an equation of the form
\begin{equation}\label{BSDE-intro}
 \boldsymbol Y_t = \boldsymbol G + \int_t^1 \boldsymbol f (s, \boldsymbol Y_s, \boldsymbol Z_s)\, \Id s - \int_t^1 \boldsymbol Z_s \, \Id W_s, \quad t\in [0,1],
\end{equation}
where $W$ is a $d$-dimensional $\{\F_t\}_{t\in [0,1]}$-Brownian motion, $\boldsymbol G\in \mathcal{F}_1$ is an $N$-dimensional random vector, and $\boldsymbol f:$ $\Omega \times [0,1]\times\R^N\times\R^{N\times d}\rightarrow\R^N$ is called the generator of the BSDE.  A solution to \eqref{BSDE-intro} is a pair $(\boldsymbol Y, \boldsymbol Z)$ of an $N$-dimensional semimartingale $\boldsymbol Y$ and an $N\times d$-dimensional adapted process $\boldsymbol Z$ so that \eqref{BSDE-intro} is satisfied a.s. for all $t$.

We characterize the equilibrium problem as a system of BSDEs where $N$ the dimension of the system corresponds to the sum of the number of agents in the economy and the risky asset. The generator $\boldsymbol f$ depends on $\boldsymbol Z$ nonlinearly and exhibiting quadratic growth. The wellposedness of systems of BSDEs $(N>1)$ whose generators are allowed to grow quadratically is an important and long-standing open problem posed by  Peng in \cite{Pen99}. The wellposedness of such systems with a ``smallness" assumption on the $L^\infty$-norm of $\boldsymbol G$ is established in \cite{Tevzadze}. Without further structural assumptions on $\boldsymbol f$, solutions may not exist as is illustrated by the example in \cite{FreRei11}. Several structural assumptions on $\boldsymbol f$ have been identified: \cite{Tan03} studies linear-quadratic systems, \cite{CheNam15} proposes a special structure which, in a Markovian setting and using a change of probability measure, allows the problem to be transformed into one which can be solved, \cite{HuTan15} proposes a diagonally-quadratic structure under which existence and uniqueness are obtained without the Markovian assumption. Superquadratic cases in the Markovian setting are studied in \cite{KupLuoTan16}. Also in the Markovian setting, existence and uniqueness are established under a general Lyapunov condition and an a priori local boundedness assumption in \cite{hXing2018}. From an application point of view, the  global existence and uniqueness of Radner equilibria is studied in \cite{Kardaras-Xing-Zitkovic} and obtained in a Markovian setting. Systems of BSDEs with quadratic growth in $\boldsymbol Z$ are also applied to other types of equilibrium problems, see \cite{Kramkov-Pulido} for a price impact model and \cite{Weston20} for an equilibrium problem in an  annuity market. Related PDE approach is taken in \cite{Zit12, Zha12, ChoLar14}. 

The system of BSDEs considered in the present paper departs from the aforementioned literature because its generator $\boldsymbol f$ is \emph{discontinuous} in $\boldsymbol z$.  A standard technique to construct a solution to \eqref{BSDE-intro} is to consider a family of approximating BSDEs whose generators $\boldsymbol f_n$ are well-behaved and converge to $\boldsymbol f$ as $n\rightarrow \infty$. Let $(\boldsymbol Y_n, \boldsymbol Z_n)$ be the solution of the approximating BSDE with generator $\boldsymbol f_n$. Suppose that $(\boldsymbol Y_n, \boldsymbol Z_n)_{n}$ (or a subsequence thereof) converges to $(\boldsymbol Y, \boldsymbol Z)$. In order to verify that $(\boldsymbol Y, \boldsymbol Z)$ is indeed a solution to \eqref{BSDE-intro}, one needs to  prove that, for each $t\in[0,1]$, 
\begin{equation}\label{conv-intro}
 \tint_t^1 \boldsymbol f_n (s, \boldsymbol Y_{n,s}, \boldsymbol Z_{n,s}) \,\Id s \rightarrow  \tint_t^1 \boldsymbol f (s, \boldsymbol Y_s, \boldsymbol Z_s) \,\Id s, \quad \prob-a.s., \quad \text{as } n\rightarrow \infty.  
\end{equation}
To this end, the continuity of $\boldsymbol f$ is a standard assumption; for $N=1$ cases, see \cite{Lepeltier-SanMartin} for continuous generators with linear growth, \cite{mKobylanski2000} and \cite{Briand-Hu} for generators with quadratic growth; for $N>1$ cases, see \cite{hXing2018} and \cite{Richou19} for generators with quadratic growth.    

In the case of the equilibrium problem considered in this paper, $\boldsymbol f(\boldsymbol z)$ is discontinuous at  the point $|\boldsymbol z^0|=0$, where $\boldsymbol z^0$ is the first row of the $N\times d$ dimensional matrix $\boldsymbol z$.  Studying the discontinuity of the stochastic process $\boldsymbol f(\boldsymbol Z)$ is  therefore closely related to determining the zeros of the process $\boldsymbol Z$, i.e., the \emph{nodal set} of $\boldsymbol Z$. When $N>1$ and $d=1$, \cite{HM} studies a system of BSDEs with discontinuous generators which emerges from non-zero-sum Nash games of bang-bang type. The generators of the BSDEs are modified at the nodal set of $\boldsymbol Z$ to establish the existence of an equilibrium. A similar approach is used in the case when $d>1$ in \cite{HM2}. When $N=d=1$, \cite{MPR} uses the representation of $\boldsymbol Z$ in terms of Malliavin derivatives to study the existence of densities of BSDE solutions in a Markovian setting.  The nondegeneracy of $\boldsymbol Z$  ensures that the marginal law of $\boldsymbol Y$ is absolutely continuous with respect to the Lebesgue measure on $\mathbb{R}$. Using a PDE representation of $\boldsymbol Z$, \cite{CLQX} examines the nodal set and the monotonicity of $\boldsymbol Z$. When $d>1$ and $N\geq 1$, $\boldsymbol Z$ is vector- or even matrix-valued. Using representations of $\boldsymbol Z$ to investigate whether some components of  $\boldsymbol Z$ equals zero becomes less effective. No general results beyond the linear case are available to the best of our knowledge. 

In the equilibrium context, the discontinuity of $\boldsymbol f$ arises when the stock volatility degenerates. Nondegenerate volatility has been studied in the literature  in the context of dynamically complete Radner equilibria  and the problem of market completion with derivatives, cf. \cite{AR}, \cite{HMT}, \cite{HR}, \cite{Kramkov15}, and \cite{Schwarz}. Thanks to the market completeness, the systems of equations are linear. Our incomplete market gives rise to highly nonlinear systems, making the analysis of the stock volatility more challenging. 

We hope to stress that the discontinuity of the BSDE systems considered cannot be removed ex ante. Whether $\boldsymbol f(\cdot, \boldsymbol Y, \boldsymbol Z)$ is continuous along the solution $(\boldsymbol Y, \boldsymbol Z)$ depends on the solution, which is  determined ex post. We present sufficient conditions on problem primitives so that the solution $(\boldsymbol Y, \boldsymbol Z)$ avoids discontinuity of the function $\boldsymbol f$.

\medskip

\noindent{\bf Unique continuation  property and backward uniqueness.} In order to determine the nodal set of $\boldsymbol Z$, we borrow analytic tools: unique continuation and backward uniqueness. Let us first briefly discuss these tools.  Consider a scalar function $u: [0, 1]\times \mathbb{R^d} \rightarrow \mathbb{R}$ satisfying a parabolic differential equation 
\begin{equation}\label{BU-ex}
 P u = \boldsymbol W (\nabla u)^\top + V u, \quad \text{over } [0,1) \times \mathbb{R}^d,
\end{equation}
where $P:= \partial_t + \frac12 \nabla \cdot (\boldsymbol A(x,t)\nabla)$ is a (backward) parabolic operator with leading coefficients $\boldsymbol A: [0,1] \times \R^d \rightarrow \R^{d\times d}$ and functions $\boldsymbol W: [0,1] \times \R^d \rightarrow \R^d$ and $V: [0,1] \times \R^d \rightarrow \R$, called respectively the first order drift term and the zero order potential. Recalling that when $\frac 12 \boldsymbol A$ is the identity matrix, $\boldsymbol W\equiv 0$ and $V\equiv 0$, the locally bounded solutions $u$ to the backward heat equation \eqref{BU-ex} verify that  $u(t,\cdot)$ is an analytic function with respect to the space-variables, for all times $0\le t < 1$, the historial study of the \emph{unique continuation  property} consists essentially on trying to find minimal conditions on $\boldsymbol A$, $\boldsymbol W$ and $V$ such that the solutions $u$ to \eqref{BU-ex} preserve certain know properties of space analytic functions even when $\boldsymbol A$, $\boldsymbol W$ and $V$ are not analytic. For example,  the \emph{unique continuation} studies the following questions:
\begin{center}
 Does $u$ satisfying \eqref{BU-ex} and $u(\tau,\cdot)=0$ over $\{\tau\}\times B_\rho$, for some $\rho >0$ and $0\le \tau<1$,  imply that $u(\tau, \cdot) \equiv 0$ over $\R^d$? If one knows that $u(\tau,\cdot)$ has a zero of infinite order at $x=0$, for some $0\le \tau< 1$, does it follow that $u(\tau,\cdot)\equiv 0$?
\end{center}
On the other hand, the \emph{Backward uniqueness} investigates the following type of backward in time uniqueness property:
\begin{center}
 Does $u$ satisfying \eqref{BU-ex} and $u(\tau, \cdot) \equiv 0$ over $\R^d$, for some $0\le\tau < 1$, imply that $u \equiv 0$ over $[\tau,1]\times \R^d$?
\end{center}
 We refer to the review \cite{Vessella-review} for a textbook treatment on unique continuation properties and backward uniqueness for solutions to second order parabolic equations.

For our application to BSDEs, consider a Markovian setting where $\boldsymbol G = \boldsymbol g (X_1)$ for a function $\boldsymbol g: \R^d \rightarrow \R^N$ and a vector-valued forward process $X$. The BSDE \eqref{BSDE-intro} is expected to admit a Markovian solution $(\boldsymbol Y, \boldsymbol Z) = (\boldsymbol v, \boldsymbol u \boldsymbol \sigma)(\cdot, X)$ for functions $\boldsymbol v: [0,1] \times \R^d \rightarrow \R^N$, $\boldsymbol u : [0,1]\times \R^d \rightarrow \R^{N\times d}$, and the volatility $\boldsymbol \sigma$ of $X$. Suppose that $\boldsymbol u$ satisfies a system of equations of type \eqref{BU-ex} with the terminal condition $\boldsymbol u(1, \cdot) = \nabla \boldsymbol g$. The unique continuation and backward uniqueness properties for $\boldsymbol u$ can help us investigate the nodal set of $\boldsymbol u$. More specifically, if the nodal set of $\boldsymbol u$ has positive Lebesgue measure on $[0,1]\times \R^d$, then unique continuation and backward uniqueness would imply that $\boldsymbol u \equiv 0$ on $[0,1]\times \R^d$, which may contradict with nonzero terminal condition $\nabla \boldsymbol g$. Therefore the nodal set of $\boldsymbol u$ must have zero Lebesgue measure on $[0,1]\times \R^d$.

\medskip

\noindent{\bf Our contribution and main results. } Motivated by the equilibrium problem, we consider a family of Markovian BSDE systems whose generator has quadratic growth in $\boldsymbol Z$, satisfies the structural condition identified by Bensoussan and Frehse (cf. \cite{Bensoussan-Frehse}) and an a priory boundedness condition (cf. Assumption \ref{ass:BSDE} later). Moreover, the generator is discontinuous when  the first row of the matrix-valued $\boldsymbol z$ is zero, that is when $|\boldsymbol z^0|=0$.  For bounded terminal conditions satisfying  certain global integrability conditions, we prove in Theorem \ref{thm:BSDE-existence} the existence of a solution $(\boldsymbol Y, \boldsymbol Z)$ where $\boldsymbol Y$ and $\boldsymbol Z$ are both bounded and $\boldsymbol Z^0$, the first row of $\boldsymbol Z$, is nonzero almost everywhere. As an application, Theorem \ref{thm:equilibrium} establishes the existence of a Radner equilibrium in which the stock volatility is nondegenerate almost everywhere. { Once the equilibrium is characterized by a system of BSDEs or PDEs and the existence of its solutions is established, one can employ standard numeric methods for PDEs to study equilibrium quantities in models with economic interest. We explore this direction in Example \ref{exa:option}} We also identify  three economies in which we obtain explicit solutions for all equilibrium quantities including  stock expected return,  volatility, and optimal strategies of agents. Example \ref{exa:complete} considers an economy where the market is complete, that is when $d=1$, { Example \ref{exa:nonhed} presents a case where agents only trade in the market to exchange hedgeable risk,}
and Example \ref{exa:Gaussian} studies an incomplete economy in which the stock dividend and the endowments are Gaussian distributed.

The present paper contributes to the literature in several ways. First, we present a self-contained backward uniqueness result over $\mathbb{R}^d$ in Theorem \ref{thm:BU} for a vector-valued function $\boldsymbol u$ which satisfies a second order backward parabolic differential inequality with variable time-dependent leading coefficients, a bounded first order drift term, and an unbounded potential of zero order term. To our surprise, despite of the simplicity of the question raised in the statement of Theorem \ref{thm:BU} and of the large number of  related existing publications, 
we found that the result in Theorem \ref{thm:BU} 
has not been considered in the current literature on backward uniqueness for second order parabolic equations (cf. \cite{doi:10.1002/cpa.3160090407,ito1958,LionsMalgrange1960,protter1961,doi:10.1002/cpa.3160200106,kurata1994,lEscauriaza2004a,Escauriaza_2003c,lEscauriaza2004b,10.2307/4097306,tNguyen2010,delsanto2015,WuJieZhangLiqun,jWu2017,WuZhang1}) or in related publications on unique continuation properties of their local solutions (cf. \cite{Landis_1974,doi:10.1002/cpa.3160430105,sautscheurer,sogge,HanFangHua94,chen,osti_441145,escauriaza2000,EscauriazaVega2001,EscauriazaFernandez,Fernandez03,AlessandriniVessella,EscauriazaFernandezVessella,KochTataru}). 
{ 
 More specifically, the following list outlines papers which handle some aspects of modelling components that we need:
 \begin{itemize}
 \item Variable leading coefficients are considered in \cite{sautscheurer, doi:10.1002/cpa.3160430105, EscauriazaFernandez, Fernandez03, WuJieZhangLiqun,jWu2017,WuZhang1}. But \cite{doi:10.1002/cpa.3160430105} studies time-independent coefficients and other papers focus on bounded zero order potentials. 
 \item Unbounded zero order potentials satisfying some integrability assumptions are studied in \cite{sogge}. However the set $\{\boldsymbol{u} =0\}$ is assumed to be open, while we need to work with a measurable set { for our aforementioned argument by contradiction}. \cite{escauriaza2000,EscauriazaVega2001,KochTataru} also consider unbounded potentials, but focus only on unique continuation properties. 
 \end{itemize}
}

{ To study our incomplete equilibrium model, we need to work with a unbounded zero order potential, because the second order spatial derivative of agent's certainty equivalent is only globally $L^{d+2}$-integrable.} { In order to work with all modelling components, we overcome several technical difficulties. 
First, in order to work with a measurable set $\{\boldsymbol{u}=0\}$, we extend techniques in \cite{Regbaoui} for elliptic equations to parabolic equations, meanwhile adjust arguments in \cite{EscauriazaFernandezVessella} to handle unbounded zero order potential.} { Second,  the Carleman inequality we derive in Lemma \ref{T: teormea2}  contains an additional term, which is ignored in \cite{EscauriazaFernandez}. This additional term helps us handle the unbounded zero order potential.
Consequently, we present all ingredients in the proof of Theorem \ref{thm:BU}, whose statement is not covered by aforementioned literature.}

Our unique continuation and backward uniqueness results imply that $|\boldsymbol Z^0|\neq 0$ almost everywhere, hence the discontinuity of $\boldsymbol f$ is avoid and the convergence in \eqref{conv-intro} holds. Our result could be useful to study BSDEs with discontinuous generators and nodal sets of $\boldsymbol Z$, in particular in the case when $d>1$. 

Second, we obtain a general existence result for an incomplete Radner equilibrium in a continuous-time endowment economy.  To the best of our knowledge, this is the first time such result is obtained. { We focus on the setting with discrete dividend and random endowment. While the study of this setting is well understood in discrete-time (cf. \cite{MM-book}), results in continuous time are rare to date. Continuous dividend and random endowment setting is studied in \cite{Lar12}, \cite{Lar14}, and \cite{Lar16}, where a linear structure is explored to establish equilibrium quantities explicitly. As soon as one moves away from this linear structure, an abstract study of the equations characterising the equilibrium becomes necessary and the technical challenges we face in this paper appear. Going beyond the linear setting also generates new economic insight. Example \ref{exa:option} numerically shows that nonlinear discrete random endowment can generate excess equity premium comparing to its complete market analogue.}

{ In \cite{Weston20}, an equilibrium model where the agents only trade a stochastic annuity is studied and volatility of the annuity is also determined endogenously. In this case, agent's optimal holding in the annuity does not depend on the endogenous volatility, so the issue of degenerating volatility does not appear.}

{
Our results also relate to the literature studying endogenously complete dynamic equilibria. In the study of these problems the non-degeneracy of the stock price volatility is shown to follow from the time analyticity of the solution of the linear PDEs characterising the stock price.  We replace the analysis of the time analyticity of solutions of linear PDEs with backward uniqueness results for nonlinear PDEs.} We show that the time analyticity assumption on model coefficients can be replaced with weaker conditions in the case of a single stock (cf. Theorem \ref{thm:BU} for the precise statement of these conditions).

Third, more technically, we complement the H\"{o}lder estimate of solutions to systems of quadratic BSDE  in \cite{hXing2018} with a local Sobolev norm estimate on an unbounded domain. This result parallels the boundary Sobolev estimate in \cite[Proposition 5.1]{Bensoussan-Frehse} which applies on a bounded domain. In combination with  the classical Sobolev embedding theorem, this Sobolev norm estimate allows us to show that $\boldsymbol Z$ is bounded. Applying this norm estimate to an approximating sequence of a quadratic BSDE system with continuous generator, one could establish the uniform BMO-norm estimate of $\boldsymbol Z^n\cdot W$ needed to construct a solution by the stability argument in \cite{Richou19}.

\medskip

\noindent{\bf Structure of the paper.} The remainder of the paper is organized as follows. The equilibrium problem is presented in \textsection \ref{sec:model} and subsequently characterized via a system of BSDEs in \textsection \ref{sec:formulation}. The main equilibrium result is also presented in \textsection \ref{sec:formulation} and followed by four examples in \textsection \ref{sec:examples}. A class of quadratic BSDE systems with discontinuous generators is introduced and the main existence result is presented in \textsection \ref{sec:BSDE-main}. In \textsection \ref{sec:approximation}, a sequence of approximating BSDEs is constructed and properties and convergence of solutions are analyzed. A self-contained backward uniqueness result is presented in \textsection \ref{sec:bu}. Additional proofs are presented in \textsection \ref{sec:add-proofs}. { Finally, some potential future research questions are discussed in \textsection \ref{sec:future}.}

\medskip

\noindent{\bf Notation and conventions.} We mark row vector or matrix valued functions or processes by bold symbols, except the $\mathbb{R}^d$-valued spatial variable $x$ which is an independent variable. Superscripts indicate components in a vector or matrix valued object. For a $(I+1)\times d$ matrix $\boldsymbol z$,  we denote $\boldsymbol z^0$ and $\boldsymbol z^i$, for $i=1, \dots, I$,  rows of $\boldsymbol z$ from the first to the last. The superscript $\top$ of a matrix indicates its transpose. Subscripts are time index or index in a sequence. 

For a scalar function $v : [0,1] \times \mathbb{R}^d \rightarrow \mathbb{R}$, $\nabla v$ is the gradient as a $\mathbb{R}^d$-valued row vector. For a vector-valued function $\boldsymbol v : [0,1] \times \mathbb{R}^d \rightarrow \mathbb{R}^{I+1}$, $\nabla \boldsymbol v$ is understood as the $\mathbb{R}^{(I+1)\times d}$-valued Jacobian matrix. 

For $1\leq p \leq \infty$, $\gamma\in (0,1]$ and a domain $B\subseteq \mathbb{R}^d$ with its closure $\overline{B}$, Sobolev spaces $W^2_p(B), W^2_\infty(B), W^{1,2}_p((0,1)\times B), W^{1,2}_{p, loc}((0,1)\times \mathbb{R}^d)$ and H\"{o}lder spaces $C^{\frac{\gamma}{2}, \gamma}([0,1]\times \overline{B})$ and $C^{1+ \frac{\gamma}{2}, 2+\gamma}([0,1] \times \overline{B})$ are defined as in \cite{oLadyzhenskaya1968}. $L^\infty + L^{d+2}$ is the following class of  functions
\[
   \{f: [0,1]\times \mathbb{R}^d \rightarrow \mathbb{R} : f= g+h, g\in L^\infty((0,1)\times \mathbb{R}^d), h \in L^{d+2}((0,1) \times \mathbb{R}^d))\}.
\]
The vectorial version is defined analogously.
The Banach space $L_t^\infty L_x^2([0,T]\times \mathbb{R}^d)$ is the space of functions $f:[0,T]\times \mathbb{R}^d \rightarrow \mathbb{R}$ with the finite norm
\[
 \|f\|_{L^\infty_t L_x^2([0,T]\times \mathbb{R}^d)} := \esssup_{t\in [0,T]} \Big(\tint_{\mathbb{R}^d}|f(t,x)|^2 \Id  x\Big)^{\frac12}.
\]
 
For a filtered probability space  $(\Omega, \mathcal{F}, (\mathcal{F}_t)_{t\in[0,1]}, \prob)$, $\mathbb{E}_t[\cdot]$ denotes the conditional expectation $\mathbb{E}[\cdot | \mathcal{F}_t]$, $\mathcal{H}^0(\R^d)$ is the class of $\R^d$-valued progressively measurable processes, $\eta \cdot S$ denotes the stochastic integral $\int_0^{\cdot} \eta_t dS_t$, and $\mathcal{E}(\eta \cdot S)$ is the stochastic exponential $\exp\big(-\tfrac12 \langle \eta\cdot S\rangle + \eta \cdot S\big)$.

\section{The Model}\label{sec:model}
Time and uncertainty are described by a filtered probability space $(\Omega, \mathcal{F}, (\mathcal{F}_t)_{t\in[0,1]}, \prob)$ satisfying the usual conditions. The initial $\sigma$-algebra $\mathcal{F}_0$ is trivial and $\mathcal{F} = \mathcal{F}_1$.

There exists a single perishable and perfectly divisible consumption good in the economy which serves as num\'{e}raire: income, consumption and prices are expressed in units of this good. A total of $I$ \textit{agents}, whose lifespan is represented by the interval $[0,1]$, populate the economy. Agents are endowed with an \textit{endowment} $\boldsymbol E = (E^i)_{i=1,\ldots,I}$, where each $E^i$ is a random variable measurable with respect to $\mathcal{F}_1$. Wealth may be consumed at the end of the time horizon only and the agents' \textit{preference ordering} over consumption is represented by CARA utility functions
\[
U^i(x) = -e^{-\frac{x}{\delta^i}}, \qquad x\in \R,\ i=1,\ldots,I.
\]
Hence individual agents may differ in their degree of \textit{risk-tolerance} $\delta^i >0$.

The financial market consists of one \textit{riskless asset} (bond) in zero net supply and one \textit{risky asset} (stock) in unit net supply. The price of the riskless asset is assumed to be constant, equivalent to the assumption of the interest rate being zero.\footnote{Because agents in this economy only consume at the end of the time horizon, this assumption entails no loss of generality.} The stock pays a dividend $\xi\in\mathcal{F}_1$. At the initiation of the market, the unit of stock is distributed amongst the agents, thereafter it can be traded without any frictions. Throughout time the agents' \textit{positions} in the stock are represented by a stochastic process $\boldsymbol{\theta} = (\theta^i)_{i=1, \dots, I}$. At the end of the time horizon the price $S_1$ of the stock equals the dividend; prior, on the time interval $[0,1)$ its value $S_t$ is determined endogenously by the equilibrium conditions specified below. Notably, the filtration in our setting will be generated by multiple sources of randomness. Therefore, the risk transfers that agents can achieve by trading a single stock are limited and the market is \emph{incomplete}.

Agents form self-financing portfolios in order to maximise their expected utility of terminal consumption. Agent $i$'s optimization problem is
\[
\sup_{\theta^i} \expec\Big[U^i \Big(\tint_0^1 \theta^i_t\ \Id S_t + E^i\Big)\Big].
\]

The CARA nature of the agents' utility functions allows us to simplify notation by scaling all variables
\begin{equation}\label{scaling}
E^i \to \frac{E^i}{\delta^i}, \quad \xi \to \frac{\xi}{\tsum_k \delta^k}, \quad S \to \frac{S}{\tsum_k \delta^k}, \quad \theta^i \to \frac{\theta^i}{\alpha^i}, \quad \text{where } \alpha^i := \frac{\delta^i}{\tsum_k \delta^k}.
\end{equation}
Hence $\alpha^i\in [0,1]$ and $\tsum_i \alpha^i =1$. Hereafter the variables $(\boldsymbol E,\xi,S,\boldsymbol \theta)$ will always represent the scaled, dimensionless quantities. The scaling of consumption has the effect that all agents now rely on the same utility function $U=U(x)$ when computing their preference ordering over \textit{scaled} consumption:
\[
U(x) = -e^{-x}, \qquad x\in \mathbb{R}.
\]
Using the scaled variables agent $i$'s optimization problem is now of the form
\begin{equation}\label{agent-op}
\sup_{\theta^i} \expec\Big[U \Big(\tint_0^1 \theta^i_t\ \Id S_t + E^i\Big)\Big].
\end{equation}
Throughout, the consistency of their mutual investment decisions is ensured by the equilibrium conditions laid out in the below definition.

Let $\mathcal{Q} = \{\mathbb{Q}^i\}_{i=1,\ldots,I}$ denote the set of probability measures defined by
\[
\frac{\Id \mathbb{Q}^i}{\Id \mathbb{P}} = \frac{U'\Big( \int_0^1 \theta^i_t\ \Id S_t  + E^i\Big) }{\mathbb{E}\Big[U'\Big( \int_0^1 \theta^i_t\ \Id S_t  + E^i\Big)\Big]} = \frac{ e^{-\int_0^1 \theta^i_t\ \Id S_t  - E^i} }{\mathbb{E}\Big[e^{-\int_0^1 \theta^i_t\ \Id S_t  - E^i}\Big]}.
\]
If $\mathcal{Q}$ is well defined, we call it the set of \textit{pricing measures}.

\begin{defn}\label{def:radner_equilibrium}
A pair $(\boldsymbol \theta, S)$ consisting of a predictable process $\boldsymbol \theta = (\theta^i)_{i=1,\ldots,I}$ and a semimartingale $S$ is a \textit{Radner equilibrium}, if
\begin{enumerate}
\item[(i)] the collection $\mathcal{Q}$ is well defined, $S_1 = \xi$ and, for $i =1,\ldots,I$, the processes $S$ and $\theta^i \cdot S$ are $\mathbb{Q}^i$-martingales;
\item[(ii)] the stock market clears:
\begin{equation}\label{eq:clearing}
\sum_{i=1}^I \alpha^i \theta^i_t = 1, \qquad t\in [0,1].
\end{equation}
\end{enumerate}
\end{defn}

The clearing condition \eqref{eq:clearing} is formulated in terms of agents' scaled positions in the stock. It readily implies that in equilibrium agents' unscaled holdings in the stock sum to one. We do not explicitly state the clearing condition for the bond market. Condition \eqref{eq:clearing} and the self-financing nature of agents' portfolios immediately imply that the agents' positions in the bond sum to zero.

The following lemma, the proof of which is presented in \textsection \ref{sec:add-proofs}, recalls that in light of the equilibrium stock price $S$ the agents' choice of strategy $\boldsymbol \theta$ in the above definition is optimal. 
\begin{lem}\label{Lem:port}
Let $(\boldsymbol \theta, S)$ be a Radner equilibrium according to Definition \ref{def:radner_equilibrium}, then, for $i=1,\ldots,I$, it holds that $\E[|U(\theta^i \cdot S + E^i)|] < \infty$ and  
\[
\E\Big[ U\Big( \tint_0^1 \theta^i_t\ \Id S_t + E^i \Big) \Big] \geq \E \Big[ U\Big( \tint_0^1 \eta_t\ \Id S_t + E^i \Big) \Big],
\]
for all processes $\eta$ such that $\eta\cdot S$ is a supermartingale under the pricing measure $\mathbb{Q}^i$.
\end{lem}

\section{Radner equilibria as solutions to a system of BSDEs}\label{sec:formulation}
We assume throughout the rest of the paper that the filtration $(\mathcal{F}_t)_{t\in[0,1]}$ is generated by a $d$-dimensional Brownian motion $W = (W^j)_{j=1,\ldots,d}$. 


Given a trading strategy $\boldsymbol \theta$ and a stock price $S$, we denote by $\boldsymbol R = (R^i)_{i=1,\ldots,I}$ the agents' \textit{certainty equivalents} of their continuation utilities, given by
\begin{equation}\label{eq:cert_equiv}
\begin{aligned}
R^i_t := U^{-1} \Big(  \E_t \Big[ U\Big( \tint_t^1 \theta^i_u\ \Id S_u + E^i \Big) \Big] \Big).
\end{aligned}
\end{equation}

The following theorem characterises Radner equilibria in terms of solutions to a system of quadratic BSDEs. The vector-valued process $\boldsymbol \zeta$, $\boldsymbol \gamma^i$, $i=1, \dots, I$, introduced in this theorem are always considered as row vectors.
\begin{thm}\label{Thm:char}
A pair $(\boldsymbol \theta,S)$ is a Radner equilibrium if and only if there exist processes $\boldsymbol R\in \mathcal{H}^0(\mathbb{R}^I)$, $\boldsymbol \zeta\in\mathcal{H}^0(\mathbb{R}^{d})$, and $\boldsymbol \gamma\in\mathcal{H}^0(\mathbb{R}^{I\times d})$ such that $(S,\boldsymbol R,\boldsymbol\zeta, \boldsymbol\gamma)$ satisfies, for $i=1,\ldots,I$ and every $t\in[0,1]$,
\begin{equation}\label{BSDE-sys}
\begin{cases}
&S_t = \xi - \int_t^1 \big(\tsum_k\alpha^k\boldsymbol \gamma^k_u + \boldsymbol \zeta_u\big) \boldsymbol \zeta_u^\top\ \Id u - \int_t^1 \boldsymbol \zeta_u\ \Id W_u\\
&R^i_t = E^i + \frac{1}{2}\int_t^1\big(\big(\boldsymbol \zeta_u +\sum_k \alpha^k \boldsymbol \gamma^k_u - \boldsymbol \gamma^i_u\big) \frac{\boldsymbol \zeta_u^\top}{|\boldsymbol \zeta_u|}\big)^2 1_{\{ \boldsymbol \zeta_u\neq 0 \}} - \tfrac12 |\boldsymbol \gamma^i_u|^2\ \Id u - \int_t^1 \boldsymbol \gamma^i_u\ \Id W_u\
\end{cases}
\end{equation}
and such that $\boldsymbol \theta$ has the decomposition
\begin{equation}\label{eq:strat_equ}
\theta^i_t = 1 + \big( \tsum_k\alpha^k\boldsymbol \gamma^k_t - \boldsymbol \gamma^i_t \big) \frac{\boldsymbol \zeta_t^\top}{|\boldsymbol \zeta_t |^2}, \quad \text{if } \boldsymbol \zeta_t\neq 0,
\end{equation}
or is arbitrarily chosen to satisfy \eqref{eq:clearing}, if $\boldsymbol \zeta_t =0$,
and the stochastic exponentials $\mathcal{Z}^i:=\mathcal{E}(-(\boldsymbol \gamma^i + \theta^i\boldsymbol \zeta)\cdot W)$ and the processes $\mathcal{Z}^iS$ and $\mathcal{Z}^i(\theta^i \cdot S)$ are  $\mathbb{P}$-martingales for all $i=1, \dots, I$.
\end{thm}

 Denote $\boldsymbol Z$ to be a $\R^{(I+1)\times d}$-valued process whose first row is $\boldsymbol \zeta$ and other rows are specified by $\boldsymbol \gamma$. If we define the function $\boldsymbol f = (f^0, f^1, \dots, f^I)(\boldsymbol z) : \mathbb{R}^{(I+1)\times d} \to \mathbb{R}^{I+1}$, by 
\begin{equation}\label{eq:f}
\begin{split}
 f^0(\boldsymbol z) =& -  \big(\tsum_k\alpha^k\boldsymbol z^k + \boldsymbol z^0 \big) (\boldsymbol z^0)^\top,\\
 f^i(\boldsymbol z) = & \tfrac12 \Big((\boldsymbol z^0 + \tsum_k \alpha^k \boldsymbol z^k - \boldsymbol z^i) \frac{(\boldsymbol z^0)^\top}{|\boldsymbol z^0|}\Big)^2 1_{\{\boldsymbol z^0 \neq 0\}} - \frac12 |\boldsymbol z^i|^2, \quad i = 1, \dots, I,
\end{split}
\end{equation}
where the summation over $k$ runs from $1$ to $I$ and $\boldsymbol z^0$ denotes the first row of the matrix $\boldsymbol z$, then {$\boldsymbol f(\boldsymbol z)$ represents the generator of the BSDE system \eqref{BSDE-sys}. Observe that $\boldsymbol f$ has quadratic growth in $\boldsymbol z$ and that when $\sum_k \alpha^k \boldsymbol z^k - \boldsymbol z^i\neq 0$, the maps $\boldsymbol z^0 \mapsto f^i(\boldsymbol z)$ are \emph{discontinuous} at $|\boldsymbol z^0|=0$. To illustrate this point, let us consider the following example:
\begin{exa}
 Consider $d=I=2$. Take $\sum_k \alpha^k \boldsymbol z^k - \boldsymbol z^i = (1, 0)$. Then, for any $z>0$,
 \[
  \big(\tsum_k \alpha^k \boldsymbol z^k  - \boldsymbol z^i\big) \frac{(\boldsymbol z^0)^\top}{|\boldsymbol z^0|} = \left\{\begin{array}{ll} 1 & \boldsymbol z^0 = (z,0) \\ -1 & \boldsymbol z^0 = (-z,0)\\ 0& \boldsymbol z^0 = (0, z)\end{array}\right..
 \]
 Therefore the previous expression is discontinuous at $\boldsymbol z^0 =(0,0)$. 
\end{exa}
The discontinuity of $f^i$, $i=1, \dots, I$, at $|\boldsymbol z^0| =0$ introduces major difficulties to establish the existence of a solution  $(S, \boldsymbol R,\boldsymbol \zeta, \boldsymbol \gamma)$ to the BSDE system \eqref{BSDE-sys} in order to prove the existence of a Radner equilibrium. In \textsection \ref{sec:BSDE-main}, we will study a family of BSDEs, containing \eqref{BSDE-sys}, with discontinuous generators at $|\boldsymbol z^0|=0$ and construct a solution $(\boldsymbol Y, \boldsymbol Z)$ such that $|\boldsymbol Z^0 |\neq 0$ a.s.-$\Id t\times \Id\prob$.

We consider a Markovian setting in which randomness is driven by a $d$-dimensional process $X$ satisfying 
\begin{equation}\label{eq:X}
 dX_t = \boldsymbol b(t, X_t) \Id t + \boldsymbol \sigma (t, X_t) \ \Id W_t, \quad X_0 \text{ given},
\end{equation}
where $X_0$ and functions  $\boldsymbol b: [0,1] \times \mathbb{R}^d \rightarrow \mathbb{R}^d$, $\boldsymbol \sigma : [0,1]\times \mathbb{R}^d \rightarrow \mathbb{R}^{d\times d}$ are given. Stock dividend and endowment are specified by
\[
 \xi = g^0(X_1) \quad \text{and} \quad E^i = g^i(X_1), \quad i = 1, \dots, I.
\]

We impose the following assumptions on the coefficients $(\boldsymbol b, \boldsymbol \sigma, \boldsymbol g)$:

\begin{ass}\label{ass:X}
$\,$
\begin{enumerate}
\item[(i)]The function $\boldsymbol b$ is once continuously differentiable in space and the function $\boldsymbol \sigma$ is once continuously differentiable in both time and space. Both functions and their first order derivatives are globally bounded and 
\begin{equation*}
\|\nabla\boldsymbol b\|_{L^\infty((0,1)\times \R^d)}+\|\nabla\boldsymbol \sigma\|_{L^\infty((0,1)\times \R^d)}+\|\partial_t\boldsymbol \sigma\|_{L^\infty((0,1)\times \R^d)}\le L_{b,\sigma},
\end{equation*}
for some constant $L_{b,\sigma}$.
\item[(ii)]	
	The functions $\boldsymbol b$, $\boldsymbol \sigma$, and their first order derivatives in space are H\"{o}lder continuous in time and space: there exists $\alpha\in (0,1]$ such that for all $h\in \{b^j,\sigma^{ij},\partial_{x^k} b^j, \partial_{x^k} \sigma^{ij} : i,j,k =1,\ldots,d \}$,
	\[
	|h(t_1,x_1) - h(t_2,x_2) | \leq L_{b,\sigma} (|t_1 - t_2|^{\frac{\alpha}{2}} + |x_1 - x_2|^\alpha), \, \text{ for any } t_1, t_2\in[0,1], x_1, x_2\in \Real^d.
	\]

\item[(iii)]	There exists a constant $\lambda>0$ such that the matrix-valued function $\boldsymbol A := \boldsymbol \sigma \boldsymbol \sigma^\top$ satisfies
	\begin{equation}\label{uni-ell}
	\lambda |\xi|^2 \leq \xi^\top \boldsymbol A(t,x)  \xi \leq \lambda^{-1} |\xi|^2, \quad \text{ for any } (t,x)\in [0,1]\times \Real^d \text{ and }  \xi\in \Real^d.
	\end{equation}
	
	\item[(iv)] The function $\boldsymbol g$ is twice continuously differentiable and $\boldsymbol g\in W^2_{2} \cap W^2_{\infty} (\R^d)$.
	Moreover, there exists a point $x_0 \in \R^{d}$ such that $|\nabla g^0(x_0)|\neq 0$.
\end{enumerate}
\end{ass}
The previous assumption readily imply that the SDE in \eqref{eq:X} admits a unique strong solution $X$. 

We now present our main result on the existence of Radner equilibrium.
\begin{thm}\label{thm:equilibrium}
 Let Assumption \ref{ass:X} hold. Then there exists a Radner equilibrium $(\boldsymbol \theta, S)$ with nondegenerate stock volatility, i.e., $|\boldsymbol \zeta |\neq 0$ a.s.-$\Id t\times \Id \mathbb{P}$.
\end{thm}

{
When $|\boldsymbol{\zeta}|\neq 0$, we can define a $1$-dimensional Brownian motion $B$ via $B_t = \int_0^t \frac{\boldsymbol{\zeta}_u}{|\boldsymbol{\zeta}_u|} \Id W_u$. Then the volatility part of the (scaled) stock price in the first equation of \eqref{BSDE-sys} is exactly $|\boldsymbol{\zeta}| \Id B$. Therefore we call $|\boldsymbol{\zeta}|$ the total volatility for the (scaled) stock price. The total volatility for the (unscaled) stock price is $\sum_k \delta^k |\boldsymbol{\zeta}|$.
}

\begin{rem}
 Uniqueness of equilibrium remains an important open question, which we left for future investigation. As Remark \ref{rem:BSDE-uniqueness} later explains, nondegenerate stock volatility is not sufficient to ensure sufficient regularity of the BSDE generator $\boldsymbol f$ to establish the uniqueness for solutions of the BSDE system \eqref{BSDE-sys}.
\end{rem}

\section{Examples}\label{sec:examples}
Our first example is an economy with one source of randomness, i.e., $d=1$. The equilibrium market is endogenously complete. 

\begin{exa}\label{exa:complete}
Assume that the dividend $\xi$ and the endowments $\boldsymbol E$ satisfy
\[
\mathbb{E}\Big[e^{-G}\Big] < \infty, \quad \mathbb{E}\Big[e^{-G} \xi\Big] <\infty, \quad \mathbb{E}\Big[e^{-G} E^k\Big] <\infty, \, k=1, \dots I,
\]
where $G= \xi + \sum_k \alpha^k E^k$. Moreover, $e^{-G}\xi$ is Malliavin 
differentiable and its Malliavin derivative $D_{t} (e^{-G} \xi) \neq 0$ a.s. for any $t\leq 1$.

We conjecture that the equilibrium volatility $\zeta\neq 0$ almost everywhere. In this case, 
summing up all equations in \eqref{BSDE-sys} yields the following BSDE for $S + \tsum_k \alpha^k R^k$:
\begin{equation*}\label{eq:ex_1_S}
S_t + \tsum_k \alpha^k R^k_t = G - \tfrac{1}{2} \tint_t^1 (\zeta_s + \tsum_k \alpha^k \gamma^k_s)^2\ \Id s - \tint_t^1 (\zeta_s + \tsum_k \alpha^k \gamma^k_s)\ \Id W_s,
\end{equation*}
which can be solved by using an exponential transform (Cole-Hopf transform) to obtain
\begin{equation*}\label{ex1:BSDE_1}
S_t + \tsum_k \alpha^k R^k_t = -\ln \mathbb{E}_t\left[e^{-G}\right]
\end{equation*}
and $\zeta + \sum_k \alpha^k \gamma^k$ can be identified as $\beta$ from the martingale representation 
\[
d e^{-(S_t + \sum_k \alpha^k R^k_t)} = - \beta_t \, e^{-(S_t+\sum_k\alpha^kR^k_t)}\ \Id W_t, 
\quad e^{- (S_1 + \sum_k \alpha^k R^k_1)} = e^{-G}.
\]

We now introduce an equivalent measure $\tilde{\mathbb{P}}\sim\mathbb{P}$ via $\Id \tilde{\mathbb{P}}/\Id \mathbb{P}|_{\mathcal{F}_t} = \mathbb{E}_t[e^{-G}] / \mathbb{E}[e^{-G}]$. Then Girsanov's theorem yields that $\tilde{W}:=W + \int_0^\cdot \beta_s \ \Id s$ is a $\tilde{\mathbb{P}}$-Brownian motion. It then follows immediately from \eqref{BSDE-sys} that
\[
S_t = \tilde{\mathbb{E}}_t[\xi] \quad \text{and} \quad  S_t + \tsum_k\alpha^kR^k_t - R^i_t = \tilde{\mathbb{E}}_t[\xi + \tsum_k \alpha^kE^k - E^i],
\]
meanwhile $\zeta$ and $\zeta + \tsum_k\alpha^k\gamma^k - \gamma^i$ can be identified via above martingale representations under $\tilde{\mathbb{P}}$. Solving above equations, we can obtain $R^i$ and $\gamma^i$ for each $i$. On the other hand, it follows from the Clark-Ocone formula that 
\[
 \zeta_t = \frac{\mathbb{E}_t \big[D_t(e^{-G} \xi)\big]}{\mathbb{E}[e^{-G}]}.
\]
Therefore the process $\zeta$ is nonzero almost everywhere from our assumption. Finally, one can verify $\mathcal{Z}^i_t = \Id \tilde{\mathbb{P}}/\Id \mathbb{P}|_{\mathcal{F}_t}$, moreover $\mathcal{Z}^i S$ and $\mathcal{Z}^i (\theta^i \cdot S)$ are $\mathbb{P}$-martingales for all $i=1, \dots, I$. Therefore the solution of \eqref{BSDE-sys} constructed above identifies an equilibrium thanks to Theorem \ref{Thm:char}.
\end{exa}

{
In the second example, there are two sources of randomness and agents' random endowments are the sum of a hedgeable and a non-hedgeable components. In equilibrium agents trade the risky asset to exchange their hedgeable risk and shoulder their own non-hedgable risk. 

\begin{exa}\label{exa:nonhed}
Consider the case $d=2$ with $W= (W^1, W^2)$, which are independent 1-dimensional Brownian motions. Assume that the (scaled) dividend $\xi\in \mathcal{F}_1^{W^1}$ and the (scaled) random endowment $E^i$ can be decomposed as $E^k = E^{k,1} + E^{k,2}$ with $E^{k,1}\in \mathcal{F}^{W^1}_1$ and $E^{k,2} \in \mathcal{F}_1^{W^2}$ for each $k$. Here $\mathcal{F}^{W^i}$ is the filtration generated by the Brownian motion $W^i$. We assume that $\xi$ and $E^{k,1}$ satisfy the non-degeneracy, integrability and Malliavin differentiability assumptions in Example \ref{exa:complete}, moreover, 
\[
 \mathbb{E}\Big[e^{- E^{k,2}}\Big]<\infty, \quad k=1, \dots, I.
\]

In this case, the system of BSDEs \eqref{BSDE-sys} admits an explicit solution $(S, \boldsymbol{R}, \boldsymbol{\zeta}, \boldsymbol{\gamma})$ which satisfy 
\[
 \boldsymbol{\zeta} = (\zeta^1, 0), \quad R^i = R^{i,1} + R^{i,2}, \quad \boldsymbol{\gamma}^i = \big(\gamma^{i,1}, \gamma^{i,2}\big).
\]
Here $(S, R^{i,1}, \zeta^1, \gamma^{i,1})_{i=1,\dots, I}$ forms an complete market in the sub-filtration $\mathcal{F}^{W^1}$. $(R^{i,2}, \gamma^{i,2})$ solves the following BSDE
\[
 R^{i,2}_t =E^{i,2} - \tfrac12 \tint_t^1 |\gamma^{i,2}_u|^2 du - \tint_t^1 \gamma^{i,2}_u dW^2_u.
\]
The process $R^{i,2}$ is the (scaled) certainty equivalent for the agent $i$ to shoulder the unhedgeable random endowment $E^{i,2}$. 
\end{exa}
}

Our third example features Gaussian dividend and endowments. There exist a closed form Radner equilibrium with an incomplete market. One can consider this example as the terminal consumption analogue of \cite{Lar12}.
\begin{exa}\label{exa:Gaussian}
Consider (scaled) dividend and endowments of the form
\[
 \xi = \boldsymbol b^0 \, W_1 \quad \text{and} \quad E^i = \boldsymbol b^i \, W_1, \quad i=1, \dots, I,
\]
where $\boldsymbol b^0$ and $\boldsymbol b^i$, $i=1, \dots, I$, are all constant $d$-dimensional (row) vectors with $|\boldsymbol b^0|\neq 0$, and $W_1$ is the time 1 value of a $d$-dimensional Brownian motion $W$. In this case, the BSDE system \eqref{BSDE-sys} admits an explicit solution 
\[
 \boldsymbol \zeta \equiv \boldsymbol b^0, \quad \boldsymbol \gamma^i \equiv \boldsymbol b^i, \quad i=1, \dots, I,
\]
\begin{align*}
 S_t =& (t-1) \big(\tsum_k \alpha^k \boldsymbol b^k + \boldsymbol b^0\big) (\boldsymbol b^0)^\top + \boldsymbol b^0 \, W_t,\\
 R^i_t = & (t-1) \Big[-\tfrac12 \Big(\big(\boldsymbol b^0 + \tsum_k \alpha^k \boldsymbol b^k - \boldsymbol b^i\big)\tfrac{(\boldsymbol b^0)^\top}{|\boldsymbol b^0|}\Big)^2 + \tfrac12 |\boldsymbol b^i|^2\Big] + \boldsymbol b^i \, W_t, \quad i=1,\dots, I.
\end{align*}
Agent's (scaled) optimal investment strategy is 
\begin{equation}\label{exa:opt_the}
 \theta^i = \frac{(\boldsymbol b^0 + \tsum_k \alpha^k \boldsymbol b^k) (\boldsymbol b^0)^\top}{|\boldsymbol b^0|^2} - \frac{\boldsymbol b^i (\boldsymbol b^0)^\top}{|\boldsymbol b^0|^2}, \quad i=1, \dots, I.
\end{equation}
Observe that 
\[
(\boldsymbol b^0 + \tsum_k \alpha^k \boldsymbol b^k) (\boldsymbol b^0)^\top = \text{Cov} (\xi + \tsum_k \alpha^k E^k, \xi), \quad \boldsymbol b^i (\boldsymbol b^0)^\top = \text{Cov} (E^i, \xi), \quad |\boldsymbol b^0|^2 = \text{Var} (\xi).
\]
Reverse the scaling in \eqref{scaling}. Let $\widetilde{S} = \tsum_k \delta^k S, \widetilde{\xi} = \tsum_k \delta^k \xi, \widetilde{E}^i = \delta^i E^i$, and $\widetilde{\theta}^i = \alpha^i \theta^i$ be the (unscaled) stock price, dividend, endowment, and investment strategy, respectively.  Then the (unscaled) stock price has an expected return and variance
\begin{equation}\label{tmu}
\widetilde{ \mu} = \tfrac{1}{\sum_k \delta^k} \text{Cov} \big(\tsum_k \widetilde{E}^k + \widetilde{\xi}, \widetilde{\xi}\big), \quad \tilde{\boldsymbol \zeta} \tilde{\boldsymbol \zeta}^\top = \text{Var}(\widetilde{\xi}),
\end{equation}
respectively. Agent's (unscaled) optimal investment strategy is 
\begin{equation}\label{ttheta}
 \widetilde{\theta}^i = \delta^i \frac{\widetilde{\mu}}{\tilde{\boldsymbol \zeta} \tilde{\boldsymbol \zeta}^\top} - \frac{\text{Cov}(\widetilde{\xi}, \widetilde{E}^i)}{\text{Var}(\widetilde{\xi})}.
\end{equation}

Results in \eqref{tmu} and \eqref{ttheta} have several economic implications, which are similar to the economic results for equilibrium with intertemploral consumption in \cite{Lar12}:
\begin{enumerate}
\item[(i)] Agent's optimal investment strategy \eqref{ttheta} can be decomposed as a \emph{mean-variance efficient component} $\delta^i \, \widetilde{\mu}/ (\tilde{\boldsymbol \zeta} \tilde{\boldsymbol \zeta}^\top) $ and a \emph{hedging component} $\text{Cov}(\widetilde{\xi}, \widetilde{E}^i)/ \text{Var}(\widetilde{\xi})$. When the covariance between endowment and stock dividend is positive, investing in stock is risky, because its final payoff is likely to co-move with agent's endowment. Therefore, agent reduces her holding in stock. Meanwhile, when the covariance between endowment and stock dividend is negative, agent's hedging component in positive. Agent uses additional position in stock to hedge randomness in her endowment.

\item[(ii)] When the covariance between the aggregated endowment and the stock dividend increases, stock expected return $\widetilde{\mu}$ increases to compensate the reduced demand from agents.  When aggregated endowment is deterministic, expected return is $\text{Var}(\widetilde{\xi}) / \tsum_k \delta^k$, which is the variance of dividend normaized by the aggregated risk tolerance of all agents. 

\item[(iii)] { Consider a complete market benchmark with a representative agent whose aggregate (unscaled) endowment is $\widetilde{G}= \widetilde{\xi} + \sum_k \widetilde{E}^k$. The (unscaled) stock price in the complete market is given by 
\begin{equation}\label{com-S}
 S^{\text{com}}_t = \mathbb{E}^{\mathbb{Q}}_t [\xi], \quad \text{where } \frac{d\mathbb{Q}}{d\prob} \Big|_{\mathcal{F}_t} = \frac{\expec_t\big[e^{-\widetilde{G}/\sum_k \delta^k}\big]}{\expec\big[e^{-\widetilde{G}/\sum_k \delta^k}\big]}.
\end{equation}
The expected return and variance of $S^{\text{com}}$ is the same as in \eqref{tmu}. This is observed in \cite{Lar14} in equilibrium models among CARA agents with intertemporal consumption.
}

\end{enumerate}
\end{exa}

{
When the endowments are nonlinear in the state variables, our last example below shows numerically that the equilibrium quantities in incomplete markets can be different from their complete counterpart. 

\begin{exa}\label{exa:option}
 Consider an economy with two sources of randomness $(W^1, W^2)$, which are independent 1-dimensional Brownian motions. The consumption good produced at time $1$ is assumed to be $\sigma W^1_1 + \sigma W^2_1$. We interpret the first source of randomness $W^1$ as the production and the second source $W^2$ as the weather. There is only one risky asset, whose (unscaled) dividend is the consumption good produced at time $1$, i.e., $\widetilde{\xi}= \sigma W^1_1 + \sigma W^2_1$. Therefore the market is incomplete. Two agents with CARA utilities trade in this market.  Agent 1 holds $N$ units of put options on the weather risk with (unscaled) payoff $P(W^2_1) = -\min\{\sigma W^2_1,0\}$, so that $\widetilde{E}^1 = N P(W^2_1)$; Agent 2 is the seller of options so that $\widetilde{E}^2 = - \widetilde{E}^1$. 
 
Consider $X=(\sigma W^1, \sigma W^2)'$ as the state variable, we solve the system of PDEs associated to (8) to obtain the numeric results in Figures \ref{Figure1} and \ref{Figure2}. 

In both figures, the complete market benchmark is calculated via \eqref{com-S}.  When $x_2$ is far away from $0$, the put option is either in-the-money or out-of-the-money, put option payoffs in these regimes are linear. As $x_2$ moves further away from $0$, the incomplete equilibrium converges to the situation in Example \ref{exa:Gaussian}, whose equity premium and total volatility are the same as their complete market counterparts. This is confirmed by the left and right tails in the left and middle panels of Figures \ref{Figure1} and \ref{Figure2}. When the put option is at-the-money, i.e., $x_2$ close to $0$, the nonlinearity of option payoff impacts the equilibrium quantities. The equity premium in the incomplete market is larger than its complete market counterpart in both figures, but the change in the total volatility is negligible. When $x^2$ moves away from $0$, each agent's (unscaled) optimal holding $\widetilde{\theta}^i$, $i=1,2$, also converges to its counter-part in the linear payoff case in Example \ref{exa:Gaussian}. Using \eqref{exa:opt_the}, we obtain 
\[
 \widetilde{\theta}^1 \rightarrow \left\{\begin{array}{ll} \alpha^1 + \tfrac{N}{2}, & x^2 \rightarrow -\infty\\ \alpha^1, & x^2 \rightarrow \infty\end{array}\right. \quad \text{and} \quad \widetilde{\theta}^2 \rightarrow \left\{\begin{array}{ll} \alpha^2 - \tfrac{N}{2}, & x^2 \rightarrow -\infty\\ \alpha^2, & x^2 \rightarrow \infty\end{array}\right..
\]
Here $\alpha^i$ is the Pareto-optimal holding for the agent $i$ and $\pm \frac{N}{2}$ is the hedging component.
When $x_2$ is negative, the put option holder chooses a positive hedging component in the risky asset to hedge the put option payoff; meanwhile the put option seller takes a negative hedging component. As $x_2$ increases to be positive, the put option becomes out-of-the-money. The hedging components vanish for both agents and their optimal holdings are determined by the Pareto-optimal holding.

Figure \ref{Figure1} shows the equilibrium quantities when the number of put options is either $N=1$ or $N=2$. The larger $N$ is, the left and middle panels show that the larger deviations in the equity premium and the total volatility from their complete market counter-parts.

\begin{figure}[h]
\centering
\includegraphics[scale=0.45]{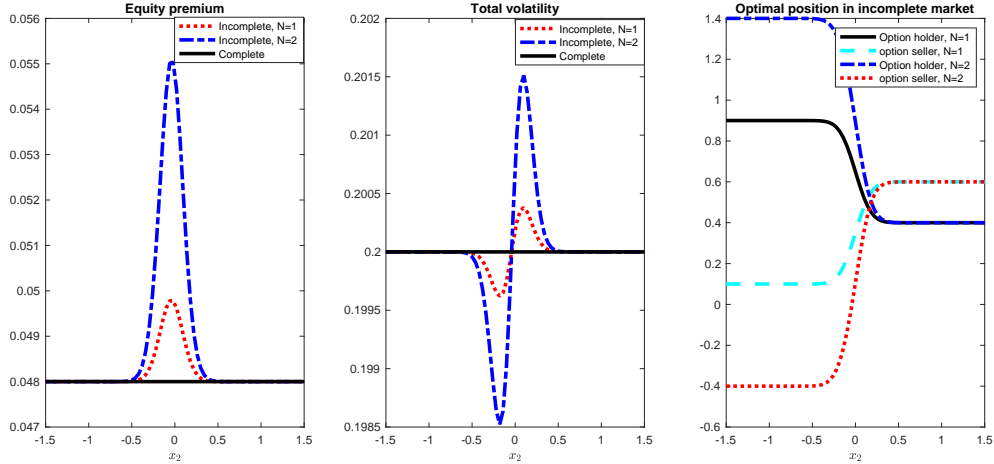}
\caption{{\textbf{Equilibrium quantities with different option leverage} This figure presents the equilibrium equity premium, stock total volatility $\sum_k \delta^k |\boldsymbol{\zeta}|$, and optimal positions for both agents at time $0$ in both incomplete and complete markets, when the number of put options is either $N=1$ or $N=2$. The parameters are $T=1$, $\delta^1 =\frac13$, $\delta^2=\frac12$, and $\sigma = \frac{20\%}{\sqrt{2}}$ so that the total volatility of $\widetilde{\xi}$ is $20\%$.}}	\label{Figure1}
\end{figure}

Figure \ref{Figure2} plots the equilibrium quantities when the risk tolerance for the option holder is $\delta^1= \frac13$ or $\frac12$ and the risk tolerance for the option seller is $\delta^2 = \frac12$. In the left panel when the put option holder is less risk tolerant, the risk tolerance of the representative agent is smaller. Hence the equity premium is higher in the complete market. Moreover, the deviations of the equity premium and total volatility in incomplete markets from their complete market counter-parts are also bigger when the option holder is less risk tolerant.

\begin{figure}[h]
\centering
\includegraphics[scale=0.45]{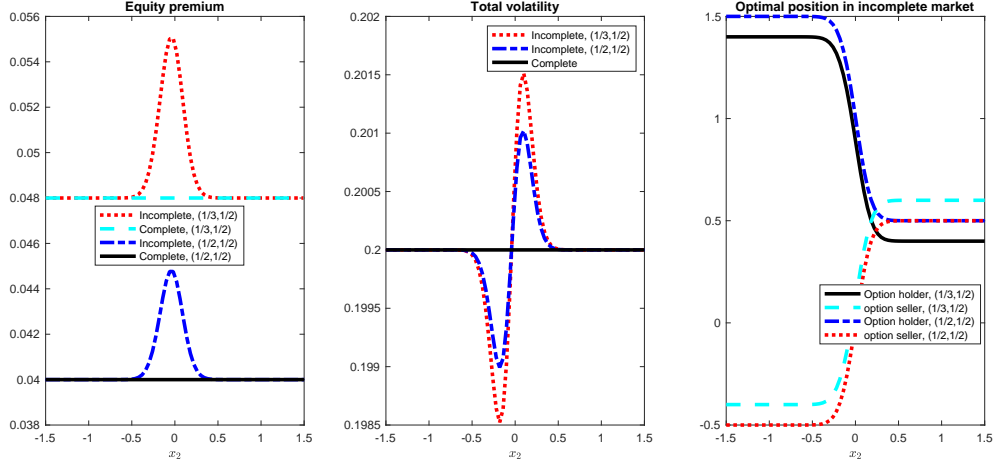}
\caption{{\textbf{Equilibrium quantities with different risk profile} This figure presents the equilibrium equity premium, stock total volatility, and optimal positions for both agents at time $0$ in both incomplete and complete markets when the risk tolerance for the option holder is $\delta^1 = \frac13$ or $\frac12$. The parameters are $T=1, N=2$,  $\delta^2=\frac12$, and $\sigma = \frac{20\%}{\sqrt{2}}$.}}	\label{Figure2}
\end{figure}

\end{exa}
}

\section{A quadratic BSDE system with discontinuous generators}\label{sec:BSDE-main}

Motivated by the equilibrium problem in the previous section, we consider a class of systems of Markovian BSDEs, which contains the system \eqref{BSDE-sys} as a special case.  Given the process $X$ satisfying \eqref{eq:X}, we seek a pair of processes $(\boldsymbol Y, \boldsymbol Z)$ which satisfies 
\begin{equation}\label{eq:P}
\boldsymbol Y_t = \boldsymbol g(X_1) + \int_t^1 \boldsymbol f(\boldsymbol Z_s)\ \Id s - \int_t^1 \boldsymbol Z_s\ \Id W_s
\end{equation}
and such that the process $\boldsymbol Y$ is continuous, $\int_t^1 \boldsymbol f(\boldsymbol Z_s)\ \Id s$ and $\int_t^1 |\boldsymbol Z_s|^2\ \Id s$ finite a.s., for all $t\in [0,1]$. We call such a pair of processes a \emph{solution} of the BSDE system \eqref{eq:P}.

Due to the Markovian nature of the equation we exploit the duality between systems of BSDEs and systems of semilinear PDEs in order to prove the existence of a solution $(\boldsymbol Y,\boldsymbol Z)$ to \eqref{eq:P}. Specifically, we will construct a sufficiently regular function $\boldsymbol v:[0,1]\times\R^d \to \R^{I+1}$, such that
\[
\boldsymbol Y_t = \boldsymbol v(t,X_t) \quad \text{and} \quad \boldsymbol Z_t = (\nabla \boldsymbol v \boldsymbol \sigma)(t,X_t),
\]
which solves the Cauchy problem
\begin{equation}\label{v-pde-sys}
\left\{
\begin{aligned}
&\begin{aligned}
\frac{\partial v^i}{\partial t} + \mathcal{L} v^i + f^i(\nabla  \boldsymbol v \boldsymbol \sigma) &= 0, &&(t,x)\in [0,1) \times \R^d,\\
v^i(1,\cdot) &= g^i, && x\in \R^d, \, i=0, \dots, I,
\end{aligned}
\end{aligned}
\right.
\end{equation}
where
\begin{equation}\label{def:L}
\mathcal{L} :=\frac12 \sum_{j,k=1}^d A^{jk}(t,x)\frac{\partial^2}{\partial x^j \partial x^k} + \sum_{j=1}^d b^j(t,x)\frac{\partial}{\partial x^j} \quad \text{and} \quad \boldsymbol A := \boldsymbol \sigma \boldsymbol \sigma^\top.
\end{equation}

Let $\boldsymbol z^0$ be the first row of $\boldsymbol z\in \R^{(I+1)\times d}$. Define
 $\Xi$ a
\[
\Xi = \{ \boldsymbol z\in \R^{(I+1)\times d} : |\boldsymbol z^0| = 0\}.
\]
We assume the following conditions on $\boldsymbol f$:

\begin{ass}\label{ass:BSDE}
$\,$
\begin{enumerate}
	\item[(i)] The function $f^0$  is zero on $\Xi$ and it is locally Lipschitz continuous on $\R^{(I+1)\times d}$, for $i= 1, \dots, I$, $f^i$ is locally Lipschitz continuous on $\R^{(I+1)\times d} \setminus \Xi$: for every compact set $K \subset \R^{(I+1)\times d}$ (if $i= 0$) or $K \subset \R^{(I+1)\times d}\setminus \Xi$ (if $i=1,\dots,I$) there is a constant $L_{f, K}$ such that
	\[
	 |f^i(\boldsymbol z_1) - f^i(\boldsymbol z_2)| \leq L_{f, K} |\boldsymbol z_1 - \boldsymbol z_2|, \quad \text{ for any } \boldsymbol z_1, \boldsymbol z_2 \in K.
	\] 
	\item[(ii)] For each $i=0, 1, \dots, I$, the function $f^i$ admits a decomposition of the form 
	\[
	 f^i(\boldsymbol z) = \boldsymbol z^i \ell^i(\boldsymbol z) + q^i(\boldsymbol z) + s^i(\boldsymbol z),
	\]
	such that for all $\boldsymbol z\in \R^{(I+1)\times d}$ the functions $\ell^i : \R^{(I+1)\times d} \rightarrow \R^d$, $q^i: \R^{(I+1)\times d} \rightarrow \R$ and $s^i: \R^{(I+1)\times d} \rightarrow \R$ satisfy
	\begin{equation}\label{BF}\tag{BF}
	\begin{split}
	 |\ell^i(\boldsymbol z)| &\leq M \, |\boldsymbol z|,\\
	 |q^i(\boldsymbol z)| & \leq M \, \tsum_{j=0}^{i} |\boldsymbol z^j|^2,\\
	 |s^i(\boldsymbol z)| & \leq \kappa (|\boldsymbol z|),
	\end{split}
	\end{equation}
	for some constant $M$ and a locally bounded function $\kappa: [0,\infty)\rightarrow [0,\infty)$ which satisfies $\lim_{r\rightarrow \infty} \kappa(r)/r^2 =0$.
	\item[(iii)] There exist a  sequence of functions $L^k: \R^{(I+1)\times d}\rightarrow \R^d$, $k=1, \dots, K$, with $K>I+1$, which satisfy $|L^k(\boldsymbol z)| \leq L_0 (1+ |\boldsymbol z|)$ for some constant $L_0$, and a sequence of nonzero vectors $\boldsymbol a_1, \dots, \boldsymbol a_K \in \R^{I+1}$  which positively span\footnote{A sequence of nonzero vectors $\boldsymbol a_1, \dots, \boldsymbol a_K \in \R^{I+1}$ positively span $\R^{I+1}$ if for each $\boldsymbol a\in \R^{I+1}$ there exist nonnegative constants $\lambda_1, \dots, \lambda_K$ such that $\lambda_1 \boldsymbol a_1 + \cdots + \lambda_K \boldsymbol a_K = \boldsymbol a$.} $\R^{I+1}$, such that 
	\begin{equation}\label{wAB}\tag{wAB}
	 \boldsymbol a_k^\top \boldsymbol f(\boldsymbol z) \leq \frac12 |\boldsymbol a_k^\top \boldsymbol z|^2 + \boldsymbol a_k^\top \boldsymbol z L^k(\boldsymbol z), 
	\end{equation} 
	for all $\boldsymbol z\in \R^{(I+1)\times d}$  and  $k = 1, \dots, K$.
	\item[(iv)] The first order derivative $\nabla  f^0$, the $(I+1)\times d$ matrix $(\partial_{z^{ij}} f^0)$, is locally Lipschitz continuous on $\R^{(I+1)\times d}$. For $i=0, 1, \dots, I$, let $(\nabla  f^0)^i$ be the $(i+1)$-th row of $\nabla f^0$. There exists a constant $M$ such that, for any $\boldsymbol z\in \R^{(I+1)\times d}$,
	\begin{equation}\label{nabf}
	\begin{split}
	|(\nabla  f^0)^0(\boldsymbol z)| &\leq M \,|\boldsymbol z|,\\
	|(\nabla  f^0)^i(\boldsymbol z)| &\leq M\, |\boldsymbol z^0|,\quad  i = 1,\ldots,I.
	\end{split}
	\end{equation}
\end{enumerate}	
\end{ass}

Assumption \ref{ass:BSDE} (ii) shows that its generator $\boldsymbol{f}$ has quadratic growth in $\boldsymbol{z}$. BSDEs of this type need structural conditions on $\boldsymbol{f}$ to ensure its wellposedness (see \cite{FreRei11}). The structural condition \eqref{BF} was discovered by \cite{Bensoussan-Frehse}. Condition  \eqref{wAB} was proposed in \cite{hXing2018} to provide a-priori $L^\infty$-bound of the solution $\boldsymbol v$ to \eqref{eq:P}. These two conditions combined provide H\"{o}lder estimates for $\boldsymbol v$; see \cite{Bensoussan-Frehse} and \cite{hXing2018}. 

In contrast to aforementioned literature, major difficulty raises in our current situation due to the discontinuity of $f^i$ at $\Xi$ for $i=1, \dots, I$. To construct a solution to \eqref{eq:P} or \eqref{v-pde-sys}, one typically approximates the discontinuous $f^i$ by a sequence of well-behaved continuous functions $(f^i_n)_{n\in \mathbb{N}}$. After establishing wellposedness for the approximating systems with the generator $\boldsymbol f_n$ and obtaining the solution $\boldsymbol v_n$, one aims to construct a limit $\boldsymbol v$ from the sequence $(\boldsymbol v_n)_{n\in \mathbb{N}}$. In order to show that $\boldsymbol v$ is indeed a solution to the system \eqref{v-pde-sys}, one needs to prove that the nonlinear term converges almost everywhere, that is to say that 
\begin{equation}\label{f-conv}
 f^i_n(x, \nabla \boldsymbol v_n \boldsymbol \sigma) \rightarrow f^i(\nabla \boldsymbol v\,  \boldsymbol \sigma) \quad \text{almost everywhere},
\end{equation}
for any $i=1, \dots, I$. Due to the discontinuity of $f^i$ at $\Xi$, in order to establish the convergence in \eqref{f-conv}, we will prove a backward uniqueness result in \textsection 7. It shows that, when $|\nabla  g^0(x_0)|\neq 0$ for some point $x_0$ (see Assumption \ref{ass:X} (iv)), then  
\begin{equation}\label{non-deg}
 |\nabla   v^0| \neq 0 \quad \text{almost everywhere on } [0,1]\times \R^d.
\end{equation}
To establish \eqref{non-deg}, we need to establish global integrability of $\boldsymbol v$ and $\nabla \boldsymbol v$. This requires the global integrability of $\boldsymbol g$ in Assumption \ref{ass:X} (iv) and that the growth bounds on the right-hand sides of \eqref{BF} and \eqref{nabf} do not have additive constants. 

We now present our main existence result for the system of BSDEs \eqref{eq:P}.

\begin{thm}\label{thm:BSDE-existence}
 Under Assumption \ref{ass:X} and \ref{ass:BSDE}, the multidimensional BSDE \eqref{eq:P} admits a Markovian solution $\boldsymbol Y_t = \boldsymbol v(t,X_t)$ and $\boldsymbol Z_t = (\nabla \boldsymbol v \boldsymbol \sigma)(t,X_t)$, $t\in [0,1]$, where $\boldsymbol v\in L^\infty \cap W^{1,2}_2 \cap  W^{1,2}_{d+2} ((0,1) \times \mathbb{R}^d)$ and $\nabla \boldsymbol v\in L^\infty ((0,1)\times \mathbb{R}^d)$. Moreover, the set 
 \[
 \{(t, x)\in [0,1]\times \R^d\,:\, |\nabla \boldsymbol v^0| = 0 \}
 \] 
 has Lebesgue measure zero.  
\end{thm}

\begin{rem}\label{rem:BSDE-uniqueness}
 Uniqueness result currently is out of reach and is left for future investigation. This is because the generator $\boldsymbol f$ lacks the local Lipschitz property 
 \begin{equation}\label{f-ll}
  \big| \boldsymbol f(\boldsymbol z) - \boldsymbol f (\tilde{\boldsymbol  z})\big| \leq C (1+|\boldsymbol z| + |\tilde{\boldsymbol  z}|) |\boldsymbol z - \tilde{\boldsymbol  z}|, \quad \text{ for any } \boldsymbol z, \tilde{\boldsymbol  z},
 \end{equation}
 for some constant $C$. This local Lipschitz property is needed to compare two solutions and is assumed for the uniqueness result in \cite[Theorem 2.14]{hXing2018}. In our equilibrium application \eqref{BSDE-sys}, a sufficient condition for \eqref{f-ll} is $|\boldsymbol \zeta|$ bounded uniformly away from zero, which is difficult to establish.
\end{rem}

{
\begin{rem}
 The statement of Theorem \ref{thm:BSDE-existence} remains valid when the set of discontinuity $\Xi$ is identified via several rows of $\boldsymbol{z}$, rather than the first one, i.e., there exists $N<I$ such that without loss of generality $\Xi = \{\boldsymbol{z} \in \Real^{(I+1)\times d} \,:\, |\boldsymbol{z}^n|=0, n=0,\dots, N\}$. Then the second condition in Assumption \ref{ass:X} (v) needs to be replaced by the existence of a point $x_0\in \Real^d$ and $n\in \{0,\dots, N\}$ such that $|\nabla g^n(x_0)|\neq 0$. Moreover, Assumption \ref{ass:BSDE} (i) and (iv) are replaced by
 \begin{enumerate}
 \item[(i')] The functions $f^0, \dots, f^n$ are zero on $\Xi$ and they are locally Lipschitz continuous on $\Real^{(I+1)\times d}$. For $i=N+1, \dots, I$, $f^i$ is locally Lipschitz continuous on $\Real^{(I+1)\times d} \setminus \Xi$.
 \item[(v')] For any $n = 0,\dots, N$, the first order derivative $\nabla f^n$, the $(I+1)\times d$ matrix $(\partial_{z^{ij}} f^n)$, is locally Lipschitz continuous on $\Real^{(I+1)\times d}$. For $i=0, 1,\dots, I$, let $(\nabla f^n)^i$ be the $(i+1)$-th row of $\nabla f^n$. There exists a constant $M$ such that, for any $\boldsymbol{z}\in \Real^{(I+1)\times d}$ and $n=0, \dots, N$, 
 \begin{align*}
  |(\nabla f^n)^i(\boldsymbol{z})| \leq M |\boldsymbol{z}|, & \quad i=0,1, \dots, N,\\
  |(\nabla f^n)^i(\boldsymbol{z})| \leq M \tsum_{n=0}^N |\boldsymbol{z}^n|, & \quad i=N+1, \dots, I.
 \end{align*}
 \end{enumerate}
 Indeed, let $\boldsymbol{u}$ be a $(N+1)\times d$ matrix with its $n$-th row $\boldsymbol{u}^n = \nabla v^n$.  Then assumptions (i') and (v') above together with other parts of Assumption \ref{ass:X} and \ref{ass:BSDE} ensure that each $\boldsymbol{u}^n$ satisfies assumptions in Lemma \ref{lem:Bu-sys} with $\boldsymbol{u}$ therein replaced by the matrix-valued $\boldsymbol{u}$. Then by working with each row of the matrix-valued $\boldsymbol{u}$, the proof of Theorem \ref{thm:BU}, which is applied to the vector-valued $\boldsymbol{u}$, also applies to the matrix-valued $\boldsymbol{u}$.
\end{rem}
}

\section{A candidate solution}\label{sec:approximation}

In this section, we will construct a candidate solution $\boldsymbol v$ for \eqref{v-pde-sys}. First, we construct a family of functions $(\boldsymbol f_n)_{n\in \mathbb{N}}$ approximating $\boldsymbol f$ and consider the associated family of systems of PDEs approximating \eqref{v-pde-sys}. Each function $\boldsymbol f_n$ is Lipschitz continuous and has certain important properties which we prove in Lemma \ref{lem:reg}. These properties allow us to apply results from \cite{hXing2018} to construct solutions $(\boldsymbol v_n)_{n\in \mathbb{N}}$ to the approximating PDEs, and to show that the sequence $(\boldsymbol v_n)_{n\in \mathbb{N}}$ is bounded and locally H\"{o}lder continuous, uniformly in $n$. Secondly, we extend a Sobolev estimate for systems of PDEs whose nonlinearity exhibits at most quadratic growth in the gradient term, which was first proved in \cite{Bensoussan-Frehse} for bounded domains, to our setting of unbounded domains to show that the sequence $(\boldsymbol v_n)_{n\in \mathbb{N}}$ is locally bounded in the $W^{1,2}_p$-norm. This allows us to prove in Corollary \ref{cor:vnm-reg} that the sequence $(\nabla \boldsymbol v_n)_{n\in \mathbb{N}}$ is globally bounded and that its $L^\infty$-norm is uniformly bounded in $n$. Thirdly, we construct $\boldsymbol v$ by taking a convergent subsequence of $(\boldsymbol v_n)_{n\in \mathbb{N}}$. However, in order to send $n\rightarrow \infty$ in approximating system to verify that $\boldsymbol v$ is indeed a solution to the system \eqref{v-pde-sys}, we need to prove \eqref{non-deg}. To prepare for this in the next section, we close this section with a global $W^{1,2}_p$-estimate for $\boldsymbol{v}$ on $(0,1)\times \R^d$ in Proposition \ref{prop:global-sobolev}.

\subsection{An approximating family}

\begin{lem}\label{lem:reg}
 Under Assumptions \ref{ass:X} and \ref{ass:BSDE} (i)-(iii), there exists a sequence of Lipschitz continuous functions $(\boldsymbol f_{n})_{n\in \mathbb{N}}$, which satisfies:
 \begin{enumerate}
 \item[(i)]
 Each $\boldsymbol f_{n}: \R^d \times \R^{(I+1)\times d} \rightarrow \R^{I+1}$ admits the decomposition 
 \begin{equation}\label{BF'}
	f^i_{n}(x,z) = z^i \ell_{n}^i(x,z) + q_{n}^i(x,z) + s_{n}^i(x,z), \quad i=0, 1, \ldots, I,
\end{equation}
 where $\ell^i_{n}$, $q^i_{n}$ and $s^i_{n}$ satisfy \eqref{BF} with the constant $M$ and the function $\kappa$ independent of $n$.
\item[(ii)] Each $f_{n}$ satisfies 
 \begin{equation}\label{wAB'}
 \boldsymbol a_k^\top \boldsymbol f_{n}(x,\boldsymbol z) \leq \frac12 |\boldsymbol a_k^\top \boldsymbol z|^2 + \boldsymbol a_k^\top \boldsymbol z L^k_n(x,\boldsymbol z),
 \end{equation}
 where $\boldsymbol a_1, \dots, \boldsymbol a_K$ is the same set of vectors as in \eqref{wAB} which positively spans $\mathbb{R}^{I+1}$ and each $L^k_n$ is a bounded function.
\item[(iii)]
 As $n\rightarrow \infty$,
\begin{align*}
f^0_{n}(x, \boldsymbol z) &\to f^0(\boldsymbol z), && (x,\boldsymbol z) \in \R^d \times \R^{(I+1)\times d},\\
f^i_{n}(x,\boldsymbol z) &\to f^i(\boldsymbol z), &&  (x,\boldsymbol z) \in\R^{d} \times (\R^{(I+1) \times d} \setminus \Xi),
\end{align*} 
with the convergence being uniform on compact subsets.
\item[(iv)] For any locally bounded function $\boldsymbol h:[0,1]\times \R^d \rightarrow \R^{(I+1)\times d}$, each $\boldsymbol f_n(\cdot, \boldsymbol h(\cdot, \cdot)) \in L^p((0,1)\times \R^d)$ for any $p\in [1,\infty]$.
\end{enumerate}
\end{lem}
\begin{proof}
 To prove assertions (i) to (iv) above we truncate the quadratic growth of $\boldsymbol f$, multiply it with a function vanishing at $\boldsymbol z^0=0$ and localize it using the newly introduced space variable $x$. Define a truncation function $\Pi_n(\boldsymbol z) := (|\boldsymbol z|\wedge n) (\boldsymbol z/|\boldsymbol z|)$. It is Lipschitz continuous and satisfies $|\Pi_n(\boldsymbol z)| = |\boldsymbol z|\wedge n$. Let $\varphi :[0,\infty)\rightarrow [0,\infty)$ be another Lipschitz continuous function satisfying $\varphi(0) =0$ and $\varphi(r) = 1$ when $r\geq 1$. Define $\varphi_n(r) := \varphi(n r)$.  Let $\eta$ be a smooth cut-off function on $\R^d$ such that $\eta(x) =1$ when $|x|\leq 1$ and $\eta(x) =0$ when $|x|\geq 2$. Define $\eta_n(x) := \eta(x/n)$ for $x\in \R^d$. The sequence of functions $(\boldsymbol f_n)_{n\in \mathbb{N}}$ is now defined in the following way: for $(x,\boldsymbol z)\in \R^d \times \R^{(I+1)\times d}$,
 \begin{equation}
f^i_n(x,\boldsymbol z) := f^i(\Pi_n(\boldsymbol z)) \varphi_n(|\boldsymbol z^0|) \eta_n(x), \quad i=0,1,\ldots,I.
\end{equation}
From the truncation function $\Pi_n$ and the local Lipschitz continuity of $f^0$ on $\R^{(I+1)\times d}$ and $f^i$ on $\R^{(I+1)\times d} \setminus \Xi$, for $i=1,\dots, I$, we obtain the Lipschitz continuity of each $f^i_n$ on these domains. In fact, the multiplication factor $\varphi_n$ allows us to deduce that, also for $i=1,\ldots,I$, the functions $f^i_n$ are Lipschitz continuous on the entire domain $\R^d \times \R^{(I+1)\times d}$.  To see this, let $(\bar x, \bar{\boldsymbol z})\in\R^d\times \R^{(I+1)\times d}$ with the first row of $\bar{\boldsymbol z}$ equal to zero. Observe that $f^i_n(\bar x, \bar{\boldsymbol z}) =0$ because $\varphi_n(0)=0$. Then, for any $(x,\boldsymbol z)\in \R^d \times \R^{(I+1)\times d}$,
\begin{align*}
| f^i_n(x,\boldsymbol z) - f^i_n(\bar x, \bar{\boldsymbol z}) | = |f^i_n(x,\boldsymbol z)| \leq C_n |\boldsymbol z^0| \leq C_n |(x,\boldsymbol z) - (\bar x, \bar{\boldsymbol z})|,
\end{align*}
where the constant $C_n$ depends on the maximum of $f^i(\Pi_n(\cdot))$ and the Lipschitz constant of $\varphi_n$. The second inequality follows because $|\boldsymbol z^0| \leq \big(|x-\bar x|^2 + \sum_{i=1}^I |\boldsymbol z^i - \bar{\boldsymbol z}^i|^2 + |\boldsymbol z^0|^2\big)^{1/2}$.
 
From the construction of $\boldsymbol f_n$, we see that $\boldsymbol f_n$ admits the decomposition \eqref{BF'} with  $\ell^i_n(x,\boldsymbol z) = \ell^i(\Pi_n(\boldsymbol z)) ((|\boldsymbol z| \wedge n)/|\boldsymbol z|) \varphi_n(|\boldsymbol z^0|) \eta_n(x)$, if $i=0,\ldots,I$.  The functions $q^i_n$ and $s^i_n$ are defined similarly. Because the values of $(|\boldsymbol z|\wedge n)/|\boldsymbol z|, \varphi_n$ and $\eta_n$ are less than one, the functions $\ell^i_n$, $q^i_n$, and $s^i_n$ satisfy \eqref{BF} with the same constant $M$ and the same function $\kappa$ uniformly in $n$. This confirms the assertions in (i).

We now prove (ii). Using the inequality \eqref{wAB}, we obtain for any $k$ and any $n$ that
\begin{align*}
 \boldsymbol a_k^\top \boldsymbol f_n(x,\boldsymbol z) &= \boldsymbol a_k^\top \boldsymbol f(\Pi_n(\boldsymbol z)) \varphi_n(|\boldsymbol z^0|) \eta_n(x) \\
 &\leq \frac12 |\boldsymbol a_k^\top \boldsymbol z|^2 \big(\tfrac{|\boldsymbol z|\wedge n}{|\boldsymbol z|} \big)^2  \varphi_n(|\boldsymbol z^0|) \eta_n(x) + \boldsymbol a_k^\top \boldsymbol z \tfrac{|\boldsymbol z|\wedge n}{|\boldsymbol z|} L^k(\Pi_n(\boldsymbol z)) \varphi(|\boldsymbol z^0|) \eta_n(x)\\
 & \leq \frac12 | \boldsymbol a_k^\top \boldsymbol z|^2 + \boldsymbol a_k^\top \boldsymbol z L^k_n(x,\boldsymbol z),
\end{align*}
where $L^k_n(x,\boldsymbol z) = ((|\boldsymbol z|\wedge n)/|\boldsymbol z|) L^k(\Pi_n(\boldsymbol z)) \varphi(|\boldsymbol z^0|) \eta_n(x)$ and it is bounded due to the at most linear growth of $L^k$ and the boundedness of $\Pi_n$, $\varphi$, and $\eta_n$.

The construction of $f^i_n$ immediately implies the convergence in (iii). Finally, it follows from the construction of $\boldsymbol f_n$ that $\boldsymbol f_n(\cdot, \boldsymbol h(\cdot, \cdot))$ has compact support and is also bounded on its support. Hence the claim in (iv) readily follows.

\end{proof}

With the approximating family $(\boldsymbol f_n)_{n\in \mathbb{N}}$ constructed in Lemma \ref{lem:reg}, we consider the following family of systems of BSDEs
\begin{equation}\label{P-app}
\left\{
\begin{aligned}
&X_t = X_0 + \int_0^t \boldsymbol b(s,X_s)\ \Id s + \int_0^t \boldsymbol \sigma(s,X_s)\ \Id W_s,\\
&\boldsymbol Y_{n,t} = \boldsymbol g(X_1) + \int_t^1 \boldsymbol f_n(X_s, \boldsymbol Z_{n,s})\ \Id s - \int_t^1 \boldsymbol Z_{n,s}\ \Id W_s.
\end{aligned}
\right.
\end{equation}
Having established the Lipschitz continuity of each $\boldsymbol f_{n}$ and the properties \eqref{wAB'} and \eqref{BF'} in Lemma \ref{lem:reg}, we may use the results in \cite{hXing2018} in order to prove the following corollary.

\begin{cor}\label{cor:vnm}
Under Assumptions \ref{ass:X} and \ref{ass:BSDE} (i)-(iv), the system of BSDEs \eqref{P-app} admits a Markovian solution $(\boldsymbol Y_n, \boldsymbol Z_n)$; there exists a function $\boldsymbol v_n: [0,1]\times\mathbb{R}^d \to \mathbb{R}^{I+1}$ such that: 
\begin{enumerate}
 \item[(i)] $\boldsymbol Y_{n,t} = \boldsymbol v_n(t, X_t)$ and $\boldsymbol v_{n}$ is bounded and locally H\"{o}lder continuous on $[0,1]\times \R^d$.
 \item[(ii)] For any $R>0$, the $L^\infty$-norm of $\boldsymbol v_{n}$ and the H\"{o}lder norm of $\boldsymbol v_{n}$ are independent of $n$ and $x_0$ on $[0,1]\times B_R(x_0)$. 
\end{enumerate}
\end{cor}

\begin{proof}
Result (i) follows from Theorem 2.5 in \cite{hXing2018}. Thanks to \eqref{wAB'} and the argument in \cite[p. 543]{hXing2018}, the $L^\infty$-norm of $\boldsymbol v_{n}$ is independent of $n$ on $[0,1]\times \R^d$. Moreover, Theorem 2.5 and Proposition 5.5 in \cite{hXing2018} imply that the H\"{o}lder norm of $\boldsymbol v_{n}$ on $(0,1)\times B_R(x_0)$ is also bounded uniformly in $n$ and $x_0$.
\end{proof}

\subsection{Local Sobolev estimates}
In this section we establish local Sobolev estimates for the family of functions $\boldsymbol v_{n}$.
We begin with the following general Sobolev estimate for systems of PDEs whose nonlinear term exhibits at most quadratic growth in the gradient term. The proof of this result follows extending the argument in \cite[Proposition 5.1]{Bensoussan-Frehse} to an unbounded domain. 

\begin{prop}\label{prop:sob_int} 
Given a point $x_0\in \R^d$, parameters $0<r<R<\infty$, $p\in(1,\infty)$, and functions $\tilde{\boldsymbol v} : [0,1] \times\R^d \rightarrow \R^{I+1}$, $\tilde{\boldsymbol g} : \R^d \rightarrow \R^{I+1}$, such that, for some $\alpha\in (0,1)$, $\tilde{\boldsymbol v}\in (W^{1,2}_p\cap C^{\alpha/2,\alpha})((0,1)\times B_R(x_0))$, $\tilde{\boldsymbol g}\in W^2_p(B_R(x_0))$; if $\tilde{\boldsymbol v}$ satisfies
\begin{equation}\label{tv-quad}
\begin{cases}
\left|\frac{\partial \tilde{\boldsymbol v}}{\partial t} + \mathcal{L} \tilde{\boldsymbol v}\right| \leq C(1 + |\nabla  \tilde{\boldsymbol v}|^2),\  &(t,x)\in [0,1) \times B_R(x_0),\\
\tilde{\boldsymbol v}(1,\cdot)= \tilde{\boldsymbol g},\ &x \in\ \Real^d,
\end{cases}
\end{equation}
then
\begin{equation}\label{eq:Sob-est}
|\tilde{\boldsymbol v}|_{W^{1,2}_p((0,1)\times B_r(x_0))} \leq C(\alpha, p, r, R, [\tilde{\boldsymbol v}]_{C^{\alpha/2, \alpha}((0,1)\times B_R(x_0))}, |\tilde{\boldsymbol v}|_{L^p((0,1)\times B_R(x_0))}, |\tilde{\boldsymbol g}|_{W^2_p(B_R(x_0))}),
\end{equation}
where $C$ is a constant depending only on the quantities in the brackets.
\end{prop}

Before presenting the proof of Proposition 4.3 we apply it to the Markovian solution of the approximating family of BSDEs \eqref{P-app}. Given that each $\boldsymbol f_{n}$ admits the decomposition \eqref{BF'} with each term satisfying the condition \eqref{BF} uniformly in $n$ and since $\boldsymbol \sigma$ is bounded, we obtain that
\[
 |\boldsymbol f_{n} (x, \nabla  \boldsymbol v_{n} \boldsymbol \sigma)| \leq C \, \big( 1+ |\nabla  \boldsymbol v_{n}|^2\big),
\]
for some constant $C$. Therefore the inequality \eqref{tv-quad} is satisfied. Applying Proposition \ref{prop:sob_int} to each function $\boldsymbol v_{n}$, we obtain the next corollary.

\begin{cor}\label{cor:vnm-reg}
Suppose Assumptions \ref{ass:X} and \ref{ass:BSDE} (i)-(iii) hold. Then
\begin{enumerate}
\item[(i)] $\boldsymbol v_{n} \in W^{1,2}_{p, loc}$ for any $p\in (1,\infty)$ and, for every $r>0$, its $W^{1,2}_p$-norm on $(0,1)\times B_r(x_0)$ is bounded, uniformly in $n$ and $x_0$.
\item[(ii)] The spatial derivative $\nabla  \boldsymbol v_{n}$ is locally H\"{o}lder continuous and globally bounded on $[0,1]\times \R^d$. Moreover, the H\"{o}lder and $L^\infty$-norms of $\nabla \boldsymbol v_{n}$ are bounded uniformly in $n$ on $[0,1]\times \mathbb{R}^d$.
\item[(iii)] For some $\gamma \in (0, 1)$, $\boldsymbol v_n \in C^{1+\frac{\gamma}{2}, 2+\gamma} ([0,1)\times B_r(x_0))$ for any $r$ and $x_0$, and it solves the Cauchy problem
\begin{equation}\label{vnm-sys}
\begin{cases}
\frac{\partial v_{n}^i}{\partial t} + \mathcal{L}  v_{n}^i + f^i_{n}(x, \nabla  \boldsymbol v_{n} \boldsymbol \sigma) = 0,\  &(t,x)\in [0,1) \times \R^d,\\
v_{n}^i(1,\cdot)= g,\  &x\in \R^d, \, i=0, \dots, I.
\end{cases}
\end{equation}
 \end{enumerate}
\end{cor}

\begin{proof}
As we have shown above, $\boldsymbol v_{n}$ satisfies  \eqref{tv-quad}. It is assumed in Assumption \ref{ass:X} (iv) that $\boldsymbol g$ is twice continuously differentiable, hence $\boldsymbol g \in W^2_{p, loc}$ for any $p\in (1,\infty)$.
Moreover, thanks to Corollary \ref{cor:vnm} (ii), the $L^\infty$ and H\"{o}lder norms of $\boldsymbol v_{n}$ in $(0,1)\times B_R(x_0)$ are independent of $n$ and $x_0$, for any $R>r$. Then (i) follows from Proposition \ref{prop:sob_int}.

To prove (ii) we choose $p>d+2$ in (i). Then it follows from the classical Sobolev embedding theorem (see \cite[Page 80, Lemma 3.3]{oLadyzhenskaya1968}) that
\[
 | \nabla  \boldsymbol v_{n}|_{C^{\alpha/2, \alpha}((0,1)\times B_r(x_0))} \leq C(|\boldsymbol v_n |_{W^{1,2}_p((0,1)\times B_r(x_0))}), 
\]
for some $\alpha \in (0, 1-(d+2)/p)$. Recall that the $W^{1,2}_p-$ estimate of $\boldsymbol v_n$ on $(0,1)\times B_r(x_0)$ is uniform in $n$ and $x_0$. Therefore, the claim in (ii) follows from the fact that, on the same domain, the H\"{o}lder and the $L^\infty$ norms of $\nabla  \boldsymbol v_{n}$ on $[0,1]\times B_r(x_0)$ are dominated by the $C^{\alpha/2,  \alpha}$-norm of $\nabla  \boldsymbol v_{n}$. Therefore they are also dominated by the constant $C$, which is independent of $n$ and $x_0$.

For each $i$ and $n$, given $\nabla \boldsymbol v_n$ and a spatial domain $B_r(x_0)$, consider the following linear boundary value problem
\begin{equation*}
\begin{cases}
\frac{\partial u}{\partial t} + \mathcal{L} u =- f^i_{n}(x, \nabla  \boldsymbol v_{n} \boldsymbol \sigma),\  &(t,x)\in [0,1) \times B_r(x_0),\\
u = v^i_n,\ & (t,x)\in [0,1] \times \partial B_r(x_0) \cup \{1\} \times B_r(x_0).
\end{cases}
\end{equation*}
Thanks to (ii), $\nabla \boldsymbol v_n$ is H\"{o}lder continuous on $[0,1]\times B_r(x_0)$. Hence $f^i_n(\cdot, \nabla \boldsymbol v_n\boldsymbol \sigma)$ is H\"{o}lder continuous on the same domain as well. It follows from \cite[Chapter 3, Theorem 9]{friedman} that the previous boundary value problem has a unique $C^{1+\frac{\alpha}{2}, 2+\alpha}$-solution, hence, also a $W^{1,2}_p$-solution. It then follows from the uniqueness of $W^{1,2}_p$-solutions for linear parabolic equations (see e.g. \cite[Chapter IV, Theorem 9.1]{oLadyzhenskaya1968}) that the unique solution for the previous boundary value problem is $v^i_n$. Therefore $v^i_n\in C^{1+\frac{\alpha}{2}, 2+\alpha}([0,1]\times B_r(x_0))$. 
\end{proof}

\begin{proof}[Proof of Proposition \ref{prop:sob_int}]
For any point $x_0\in \R^d$ and a function $\boldsymbol h:[0,1]\times \R^d\to \R^{I+1}$ we write,
\begin{align*}
|\boldsymbol h|^{(k)}_{p,s,x_0} &:= |\boldsymbol h|_{W^k_p(B_s(x_0))}, & |\boldsymbol h|_{p,s,x_0} &:= |\boldsymbol h|_{L^p((0,1)\times B_s(x_0))},\\
|\boldsymbol h|^{(r,k)}_{p,s,x_0} &:= |\boldsymbol h|_{W^{r,k}_p((0,1)\times B_s(x_0))}, & [\boldsymbol h]^\alpha_{s,x_0} &:= [\boldsymbol h]_{C^{\alpha/2,\alpha}((0,1)\times B_s(x_0))}.
\end{align*}
If $x_0=0$ we omit $x_0$ from the above definitions.

Without loss of generality we set $x_0=0$. We introduce a family of smooth cut-off functions $\tau_{R,\delta} : \R^d\to \R$, parametrised by $\delta\in [0,1]$, with the property that
\[
\tau_{R,\delta}(x) = 
\begin{cases}
1, & x\in B_{R-\delta},\\
0, & x\notin B_{R-\delta/2}
\end{cases}
\]
and
\[
|\nabla \tau_{R,\delta}| \leq \frac{C}{\delta} \quad \text{and} \quad |\nabla^2\tau_{R,\delta}| \leq \frac{C}{\delta^2},
\]
for some constant $C$. These properties imply that
\[
|\nabla \tau_{R,\delta}^2| \leq \frac{C}{\delta}\tau_{R,\delta} \leq \frac{C}{\delta} \quad \text{and} \quad |\nabla ^2 \tau_{R,\delta}^2| \leq \frac{C}{\delta^2}.
\]

We fix a value for $\delta$ and observe that the components $\tau_{R,\delta}^2\tilde v^i$, $i = 1,\dots,I+1$, of the vector-valued function $\tau_{R,\delta}^2\tilde{\boldsymbol v}$ verify the following terminal boundary value problems over the cylinder $[0,1) \times B_{R-\frac\delta2}$:
\begin{equation*}
\begin{cases}
\frac{\partial }{\partial t}(\tau_{R,\delta}^2\tilde v^i) + \mathcal{L}(\tau_{R,\delta}^2\tilde v^i) = \tau_{R,\delta}^2\left(\tfrac{\partial\tilde v^i}{\partial t} + \mathcal{L}\tilde v^i\right) + \tilde v^i\mathcal{L}\tau_{R,\delta}^2 + \boldsymbol A\nabla\tau_{R,\delta}^2\cdot\nabla\tilde v^i,\\
\tau_{R,\delta}^2\tilde v^i(t,x) = 0,\quad\quad\quad\quad (t,x)\in (0,1) \times \partial B_{R-\frac{\delta}{2}},\\
\tau_{R,\delta}^2\tilde v^i(1,\cdot) = \tau_{R,\delta}^2\, \tilde g^i,\quad\quad\,\, x\in B_{R-\frac{\delta}{2}}.
\end{cases}
\end{equation*}

The assumptions on $\tilde{\boldsymbol v}$ guarantee that the right-hand side of this system of equations belongs to $L^p((0,1)\times B_R)$. Moreover, taking into account the conditions on the coefficients of $\mathcal{L}$ in Assumption \ref{ass:X}, specifically the boundedness of $b$ and the boundedness and continuity of $\boldsymbol \sigma$, we may use the energy estimate from \cite[Theorem 9.1 in Chapter 4]{oLadyzhenskaya1968} to obtain
\begin{equation}\label{eq:estEnergy}
\begin{aligned}
|\tau_{R,\delta}^2\tilde{\boldsymbol v}|^{(1,2)}_{p,R-\delta/2} &\leq C \sum_{i=1}^{I+1}\left|\tau_{R,\delta}^2\left(\tfrac{\partial\tilde v^i}{\partial t} + \mathcal{L}\tilde v^i\right) + \tilde v^i\mathcal{L}\tau_{R,\delta}^2 + \boldsymbol A\nabla\tau_{R,\delta}^2\cdot\nabla\tilde v^i\right|_{p,R-\delta/2} + C|\tau_{R,\delta}^2\,\tilde{\boldsymbol g}|^{(2)}_{p,R-\delta/2}\\
&\leq C|\tau_{R,\delta}^2 |\nabla \tilde{\boldsymbol v}|^2|_{p,R-\delta/2} + \frac{C}{\delta^2} |\tilde{\boldsymbol v}|_{p,R-\delta/2} + \frac{C(p,R)}{\delta^2} + C|\tilde{\boldsymbol g}|^{(2)}_{p,R-\delta/2},
\end{aligned}
\end{equation}
where the second estimate follows from \eqref{tv-quad}, the properties of the cut-off function $\tau_{R,\delta}$ and Young's inequality applied to $\boldsymbol A\nabla\tau_{R,\delta}^2\cdot\nabla\tilde v^i$, $i =1,\dots, I+1$.

Define $\tilde{\boldsymbol v}_0(t):=\tilde{\boldsymbol v}(t,x_0)=\tilde{\boldsymbol v}(t,0)$ and, in order to facilitate the integration by parts in the below estimate, observe the decomposition
\begin{equation*}
|\nabla \tilde{\boldsymbol v}|^2 = \frac{\partial}{\partial x^k}\left( \frac{\partial \tilde{v}^i}{\partial x^k}(\tilde{v}^i - \tilde{v}_0^i)\right) - \Delta \tilde{v}^i(\tilde{v}^i - \tilde{v}_0^i).
\end{equation*}

Using integration by parts, we obtain
\begin{align*}
&\int_0^1 \int_{B_{R-\delta/2}} \tau_{R,\delta}^{2p} |\nabla \tilde{\boldsymbol v}|^{2p}\ \Id x\ \Id t\\
 =&\int_0^1 \int_{B_{R-\delta/2}}\tau_{R,\delta}^{2p} |\nabla \tilde{\boldsymbol v}|^{2(p-1)} \frac{\partial}{\partial x^k} \left(\frac{\partial \tilde{v}^j}{\partial x^k}(\tilde{v}^j - \tilde{v}_0^j) \right)\ \Id x\ \Id t\\
&-\int_0^1 \int_{B_{R-\delta/2}} \tau_{R,\delta}^{2p} |\nabla \tilde{\boldsymbol v}|^{2(p-1)} \Delta \tilde{v}^j(\tilde{v}^j - \tilde{v}_0^j) \ \Id x\ \Id t\\
=&- 2p \int_0^1 \int_{B_{R-\delta/2}} \tau_{R,\delta}^{2p-1} \frac{\partial \tau_{R,\delta}}{\partial x^k}|\nabla \tilde{\boldsymbol v}|^{2(p-1)} \frac{\partial \tilde{v}^j}{\partial x^k}(\tilde{v}^j - \tilde{v}_0^j)  \ \Id x\ \Id t\\
&- 2(p-1) \int_0^1 \int_{B_{R-\delta/2}} \tau_{R,\delta}^{2p} |\nabla \tilde{\boldsymbol v}|^{2(p-2)} \frac{\partial \tilde{v}^n}{\partial x^l}\frac{\partial^2 \tilde{v}^n}{\partial x^k \partial x^l} \frac{\partial \tilde{v}^j}{\partial x^k}(\tilde{v}^j - \tilde{v}_0^j) \ \Id x\ \Id t\\
&-\int_0^1 \int_{B_{R-\delta/2}} \tau_{R,\delta}^{2p} |\nabla \tilde{\boldsymbol v}|^{2(p-1)} \Delta \tilde{v}^j(\tilde{v}^j - \tilde{v}_0^j)  \ \Id x\ \Id t\\
\leq & C(p)\int_0^1 \int_{B_{R-\delta/2}} \tau_{R,\delta}^{2p} |\nabla \tilde{\boldsymbol v}|^{2(p-1)}|\nabla^2\tilde{\boldsymbol v}| |\tilde{v} - \tilde{\boldsymbol v}_0| \ \Id x\ \Id t\\
&+C(p) \int_0^1 \int_{B_{R-\delta/2}} \tau_{R,\delta}^{2p-1} |\nabla \tau_{R,\delta}| |\nabla \tilde{\boldsymbol v}|^{2p-1}| |\tilde{\boldsymbol v} - \tilde{\boldsymbol v}_0| \ \Id x\ \Id t.\\
\intertext{Using Young's inequality, the above is}
\leq & C(p)\int_0^1 \int_{B_{R-\delta/2}} \tau_{R,\delta}^{2p} |\nabla \tilde{\boldsymbol v}|^{2p} |\tilde{\boldsymbol v} - \tilde{\boldsymbol v}_0| \ \Id x\ \Id t\\
&+C(p)\int_0^1 \int_{B_{R-\delta/2}} \tau_{R,\delta}^{2p} |\nabla^2\tilde{\boldsymbol v}|^p |\tilde{\boldsymbol v} - \tilde{\boldsymbol v}_0| \ \Id x\ \Id t\\
&+C(p)\int_0^1 \int_{B_{R-\delta/2}} |\nabla \tau_{R,\delta}|^{2p} |\tilde{\boldsymbol v} - \tilde{\boldsymbol v}_0| \ \Id x\ \Id t.\\
\intertext{From the H\"{o}lder continuity of $\tilde{\boldsymbol v}$, for every $(t,x)\in (0,1)\times B_{R-\delta/2}$, $|\tilde{\boldsymbol v}(t,x) - \tilde{\boldsymbol v}_0(t)| \leq [\tilde{\boldsymbol v}]^\alpha_{R} R^\alpha$, the right-hand side above is bounded from above by}
\leq& C(p) [\tilde{\boldsymbol v}]^\alpha_{R} R^\alpha \int_0^1 \int_{B_{R-\delta/2}} \tau_{R,\delta}^{2p} |\nabla \tilde{\boldsymbol v}|^{2p} \ \Id x\ \Id t\\
&+C(p) [\tilde{\boldsymbol v}]^\alpha_{R} R^\alpha \int_0^1 \int_{B_{R-\delta/2}} \tau_{R,\delta}^{2p} |\nabla^2\tilde{\boldsymbol v}|^p \ \Id x\ \Id t\\
&+C(p) [\tilde{\boldsymbol v}]^\alpha_{R} R^\alpha \int_0^1 \int_{B_{R-\delta/2}} |\nabla \tau_{R,\delta}|^{2p} \ \Id x\ \Id t.
\end{align*}

Now choose $R_1$ such that $1-C(p) [\tilde{\boldsymbol v}]^\alpha_{R_1} R_1^\alpha > 0$. Replacing $R$ in the aforegoing computations with any $\tilde{R}\in(0,R_1)$ and possibly also replacing $\delta$ with a smaller $\tilde{\delta}\in (0,\delta)$, we obtain
\begin{equation}\label{eq:estGrad}
\int_0^1 \int_{B_{\tilde{R}-\tilde{\delta}/2}} \tau_{\tilde{R},\tilde{\delta}}^{2p} |\nabla \tilde{\boldsymbol v}|^{2p} \ \Id x\ \Id t \leq \frac{C(p) [\tilde{\boldsymbol v}]^\alpha_{\tilde{R}} \tilde{R}^\alpha}{1 - C(p) [\tilde{\boldsymbol v}]^\alpha_{\tilde{R}} \tilde{R}^\alpha} \int_0^1 \int_{B_{\tilde{R}-\tilde{\delta}/2}} |\nabla^2\tilde{\boldsymbol v}|^p \ \Id x\ \Id t + \frac{1}{\tilde{\delta}^{2p}}C(\alpha, p, \tilde{R},[\tilde{\boldsymbol v}]^\alpha_{\tilde{R}}).
\end{equation}
From \eqref{eq:estEnergy} and \eqref{eq:estGrad} we find that
\[
|\nabla^2\tilde{\boldsymbol v}|_{p,\tilde{R}-\tilde{\delta}}^p \leq \frac{C(p) [\tilde{\boldsymbol v}]^\alpha_{\tilde{R}} \tilde{R}^\alpha}{1 - C(p) [\tilde{\boldsymbol v}]^\alpha_{\tilde{R}} \tilde{R}^\alpha} |\nabla^2\tilde{\boldsymbol v}|_{p,\tilde{R}-\tilde{\delta}/2}^p + \frac{1}{\tilde{\delta}^{2p}}C(\alpha,p, \tilde{R},[\tilde{\boldsymbol v}]^\alpha_{\tilde{R}}, |\tilde{\boldsymbol v}|_{p,\tilde{R}})  + C|\tilde{\boldsymbol g}|^{(2)p}_{p,\tilde{R}}
\]

By possibly choosing an even smaller $R_1$ above, we may assume that
\[
\tilde{\delta}^{2p}\frac{C(p) [\tilde{\boldsymbol v}]^\alpha_{\tilde{R}} \tilde{R}^\alpha}{1 - C(p) [\tilde{\boldsymbol v}]^\alpha_{\tilde{R}} \tilde{R}^\alpha} \leq \frac12.
\]
Defining
\begin{align*}
F(\tilde{\delta}) &:= \tilde{\delta}^{2p} |\nabla^2\tilde{\boldsymbol v}|_{p,\tilde{R}-\tilde{\delta}}^p,\\
G(\tilde{\delta}) &:= \tilde{\delta}^{2p} C|\tilde{\boldsymbol g}|^{(2)p}_{p,\tilde{R}} + C(\alpha,p, \tilde{R}, [\tilde{\boldsymbol v}]^\alpha_{\tilde{R}}, |\tilde{\boldsymbol v}|_{p,\tilde{R}}),
\end{align*}
we obtain the recursive relationship
\[
F(\tilde{\delta}) \leq \frac12 F\Big(\frac{\tilde{\delta}}{2}\Big) + G(\tilde{\delta}).
\]
Since the function $f(\tilde{\delta})$ is bounded for all $\tilde{\delta}\in(0,\tilde{R})$ and $g$ is monotonically increasing it follows that
\[
F(\tilde{\delta}) \leq \sum_{n=0}^{\infty} \frac{1}{2^n}G\Big(\frac{\tilde{\delta}}{2^n}\Big) \leq 2G(\tilde{\delta}).
\]
Dividing this inequality by $\tilde{\delta}^{2p}$ we obtain the estimate
\[
|\nabla ^2\tilde{\boldsymbol v}|_{p,\tilde{R}-\tilde{\delta}}^p \leq C(\alpha, \delta, p, \tilde{R}, [\tilde{\boldsymbol v}]^\alpha_{\tilde{R}}, |\tilde{\boldsymbol v}|_{p,\tilde{R}}, |\tilde{\boldsymbol g}|^{(2)}_{p,\tilde{R}}).
\]

From a finite covering of the ball $B_{\tilde{R}-\tilde{\delta}/2}$ with smaller balls $B_{\tilde{R}-\tilde{\delta}}(x_1)$ together with \eqref{eq:estEnergy} and \eqref{eq:estGrad} it follows that
\[
|\tilde{\boldsymbol v}|^{(1,2)}_{p,\tilde{R}-\tilde{\delta}} \leq C(\alpha, \delta, p, \tilde{R}, [\tilde{\boldsymbol v}]^\alpha_{\tilde{R}}, |\tilde{\boldsymbol v}|_{p,\tilde{R}}, |\tilde{\boldsymbol g}|^{(2)}_{p,\tilde{R}}).
\]
A further covering argument now yields the result with $r=R-\delta$.
\end{proof}

\subsection{Limit of the approximating family and a global Sobolev estimate}\label{sec:limit_soln}
In this section we establish the existence of the limit of $(\boldsymbol v_n)_{n\in \mathbb{N}}$ as $n\rightarrow \infty$ and study some of its properties. The resulting function $\boldsymbol v : [0,1]\times \R^d \rightarrow \R^{I+1}$ serves as a candidate solution for the system of PDEs \eqref{v-pde-sys}. Corollary \ref{cor:vnm} implies that $(\boldsymbol v_{n})_{n\in \mathbb{N}}$ is uniformly bounded and equi-continuous on $[0,1]\times B_R(x_0)$. Therefore, the Arzel\'{a}-Ascoli theorem allows us to extract a subsequence of $(\boldsymbol v_{n})_{n\in \mathbb{N}}$, which converges uniformly. A diagonal procedure then produces a further subsequence which converges locally uniformly to a continuous function $\boldsymbol v$. It is well known that this convergence preserves the local H\"{o}lder continuity and local Sobolev integrability. In particular, Corollaries \ref{cor:vnm} and \ref{cor:vnm-reg} imply that the H\"{o}lder norm and the $W^{1,2}_p$-norm of $\boldsymbol v$ are finite on $(0,1)\times B_r(x_0)$, uniformly in $x_0$. Moreover, the $L^\infty$-norm of $\boldsymbol v$ and $\nabla  \boldsymbol v$ is also finite on $[0,1]\times \R^d$. Finally, to prove \eqref{non-deg} and to prepare for the next section, we will need the following global $W^{1,2}_p$-norm estimate of $\boldsymbol v$ on $(0,1)\times \R^d$.

\begin{prop}\label{prop:global-sobolev}
Suppose Assumptions \ref{ass:X} and \ref{ass:BSDE} (i)-(iii) hold. Then $\boldsymbol v_n\in W^{1,2}_2\cap W^{1,2}_{d+2} ((0,1)\times \R^d)$, and its $W^{1,2}_2\cap W^{1,2}_{d+2}$-norm is bounded uniformly in $n$. Therefore, also $\boldsymbol v\in W^{1,2}_2 \cap W^{1,2}_{d+2} ((0,1)\times \R^d)$.
\end{prop}

To prove Proposition \ref{prop:global-sobolev}, let us first prepare the following result.
\begin{lem}\label{lem:global-Lp}
Suppose Assumptions \ref{ass:X} and \ref{ass:BSDE} (i)-(iii) hold. Then $\boldsymbol v_n \in L^2\cap L^{d+2}((0,1)\times \R^d)$ and its $L^2\cap L^{d+2}$-norm is bounded uniformly in $n$.
\end{lem}

\begin{proof}
We will only prove the statement that $\boldsymbol v_n \in L^{d+2}$. The assertion that $\boldsymbol v_n \in L^2$ is proved similarly given that $\boldsymbol g\in W^2_2$ according to Assumption \ref{ass:X} (iv).
Let $a_1, \dots, a_k$ be the positively spanning set from condition \eqref{wAB}. Given $k\in \{1, \dots, K\}$, we consider the following BSDE:
\[
 d\bar{Y}^k_{n,t} = - \Big[\frac12 |\bar{ \boldsymbol Z}^{k}_{n,t}|^2 + \bar{\boldsymbol Z}^{k}_{n,t} L^k_n\big(X_t, \nabla \boldsymbol v_n(t, X_t)\big)\ \Id t  \Big] + \bar{\boldsymbol Z}^{k}_{n,t}\ \Id W_t, \quad \bar Y^k_{n,1} = \boldsymbol a_k^\top \boldsymbol g(X_1).
\]
Because $\boldsymbol g$ and $L^k_n$ are bounded, this BSDE admits a solution $(\bar Y^k_n, \bar{\boldsymbol Z}^k_n)$ such that  $\bar Y^k_{n,t} = \bar v^k_t (t, X_t)$ for some bounded function $\bar v^k_n$ and $\bar{\boldsymbol Z}^k_n \in \text{BMO}$; see \cite{mKobylanski2000}. Further, construction of $L^k_n$ in Lemma \ref{lem:reg} and boundedness of $\nabla \boldsymbol v_n$ uniformly in $n$ in Corollary \ref{cor:vnm-reg} (ii) imply that $L^k_n(\cdot, \nabla \boldsymbol v_n)$ is bounded on $[0,1]\times \R^d$ uniformly in $n$. Therefore, the generator of the previous BSDE satisfies the condition \eqref{BF}, uniformly in $n$, the same argument used to prove Corollary \ref{cor:vnm-reg} (ii) (now applied to a 1-dimensional BSDE) allows us to deduce that $\nabla \bar v^k_n\in L^\infty$ and that the $L^\infty$-norm is bounded uniformly in $n$. Therefore, also $\bar{\boldsymbol Z}^k_{n,t} = (\nabla \bar v^k_n \boldsymbol \sigma)(t, X_t)$ is bounded uniformly in $n$. 

We now define a measure $\bar{\mathbb{P}}$ under which
\[
\Id\bar W_t = \Id W_t - \big[\frac12 \bar{\boldsymbol Z}^k_{n,t} + L^k_n (X_t, \nabla \boldsymbol v_n(t, X_t))\big]\ \Id t
\]
defines a $\bar{\mathbb{P}}$-Brownian motion $\bar{W}$. Then the function $\bar v^k_n$ solves the linear Cauchy problem
\begin{equation}\label{vnm-sys-barP}
\begin{cases}
\frac{\partial \bar v^k_n}{\partial t} + \bar{\mathcal{L}} v^k_n =0,\  &(t,x)\in [0,1) \times \R^d,\\
v^k_n(1, \cdot) = \boldsymbol a^\top_k \boldsymbol g(\cdot),\  &x\in \R^d,
\end{cases}
\end{equation}
where $\bar{\mathcal{L}}$ is the infinitesimal generator of the stochastic process $X$ which solves the SDE
\[
 dX_t = \Big[\boldsymbol b(t, X_t) +\boldsymbol \sigma(t, X_t) \big( \tfrac12 \bar{\boldsymbol Z}^k_{n,t} + L^k_n (X_t, \nabla \boldsymbol v_n(t, X_t))\big)\Big]\ \Id t + \boldsymbol \sigma(t, X_t)\ \Id \bar{W}_t.
\]
Note that coefficients of $\bar{\mathcal{L}}$ are bounded uniformly in $n$.
Therefore, given $\boldsymbol g\in W^2_{d+2}((0,1)\times \R^d)$, the $W^{1,2}_{d+2}$-estimate for linear PDEs (see e.g. \cite[Chapter IV, Theorem 9.1]{oLadyzhenskaya1968}) implies that
\begin{equation}\label{barv-Lp}
 \bar v^k_n \in L^{d+2} ((0,1)\times \R^d),
\end{equation}
and the $L^{d+2}$-norm of $\bar v^k_n$ are bounded uniformly in $n$ as well. 

Meanwhile, thanks to \eqref{wAB'}, the comparison theorem for Lipschitz BSDEs (see, e.g., \cite[Theorem 2.2]{ElKaroui-Peng-Quenez}) implies that $\boldsymbol a^\top_k \boldsymbol v_n \leq \bar v^k_n$. Therefore, $\boldsymbol a^\top_k \boldsymbol v_n$ is bounded from above by a $L^{d+2}((0,1)\times \R^d)$ function. It remains to establish that when the sequence $(\boldsymbol a^\top_k \boldsymbol v_n)_{k=1, \dots, K}$ is bounded from above by a sequence of $L^{d+2}((0,1)\times \R^d)$ functions for a positive spanning set $\{\boldsymbol a_1, \dots, \boldsymbol a_K\}$ of $\R^{I+1}$, then $\boldsymbol v_n\in L^{d+2}((0,1)\times \R^d)$ itself. This fact is proved similarly as shown in the 
first paragraph of \cite[Page 542]{hXing2018}. Because the $L^{d+2}$-norm of $\bar v^k_n$ is bounded uniformly in $n$, so is the $L^{d+2}$-norm of $\boldsymbol v_n$.
\end{proof}


\begin{proof}[Proof of Proposition \ref{prop:global-sobolev}]
The constant $C$ below will differ from line to line throughout this proof.
 Recall that $\boldsymbol f_n$ admits the decomposition \eqref{BF'} with each term in the decomposition satisfying the condition \eqref{BF}. Therefore,
thanks to the global boundedness of $\nabla \boldsymbol v_n$ which we established in Corollary \ref{cor:vnm-reg} (ii), we deduce from \eqref{vnm-sys} that 
\begin{equation}\label{v_n-est}
 \left| \frac{\partial v_n^i}{\partial t} + \mathcal{L} v^i_n\right| \leq \big|f^i_n(x, \nabla \boldsymbol v_n \boldsymbol \sigma)\big| \leq C |\nabla \boldsymbol v_n|.
\end{equation} 
The constant $C$ depends on $\|\nabla \boldsymbol v_n\|_{L^\infty}$, which is bounded uniformly in $n$, and $\|\boldsymbol \sigma\|_{L^\infty}$. Using the fact that, according to Assumption \ref{ass:X} (iv), $\boldsymbol g\in W^2_{d+2}(\R^d)$ and the $W^{2,1}_{d+2}$-estimate for linear PDEs (see e.g. \cite[Chapter IV, Theorem 9.1]{oLadyzhenskaya1968}), we obtain 
\begin{equation}\label{vn-W21p}
\|v^i_n\|_{W^{1,2}_{d+2} ((0,1) \times \R^d)} \leq C \Big(\|\nabla \boldsymbol v_n\|_{L^{d+2}((0,1)\times \R^d)} + \|g^i\|_{W^2_{d+2}((0,1)\times \R^d)}\Big)
\end{equation}
and
\[
 \|v^i_n\|_{W^{1,2}_{d+2} ((0,1) \times \R^d)} \leq C \Big(\|f_n^i\|_{L^{d+2}((0,1)\times \R^d)} + \|g^i\|_{W^2_{d+2}((0,1)\times \R^d)}\Big).
\]
Thanks to Lemma \ref{lem:reg}(iv),  the right-hand side of the  second inequality above is bounded, hence $\|v^i_n\|_{W^{1,2}_{d+2} ((0,1) \times \R^d)}$  is bounded as well for all $i = 0,\dots, I$ and $n\ge 1$. Summing both sides of \eqref{vn-W21p} over $i$, we obtain 
\begin{equation}\label{vn-W21p2}
 \|\boldsymbol v_n\|_{W^{1,2}_{d+2} ((0,1) \times \R^d)} \leq C \Big(\|\nabla \boldsymbol v_n\|_{L^{d+2}((0,1)\times \R^d)} + \|\boldsymbol g\|_{W^2_{d+2}((0,1)\times \R^d)}\Big).
\end{equation}
From the classical Sobolev interpolation inequality (see e.g. \cite[Lemma 7.19]{Lieberman}) we know that 
\[
 \|\nabla \boldsymbol v_n\|_{L^{d+2}((0,1)\times \R^d)} \leq \epsilon \|\nabla ^2 \boldsymbol v_n\|_{L^{d+2}((0,1)\times \R^d)} + \frac{C(d)}{\epsilon} \|\boldsymbol v_n\|_{L^{d+2}((0,1)\times \R^d)},
\]
for some constant $C(d)$ depending on $d$. Choosing $\epsilon$ so that $C\epsilon \leq 1/2$, where $C$ is the constant in \eqref{vn-W21p2}, we combine the previous two estimates to conclude that
\begin{equation}\label{ass:VW}
  \|\boldsymbol v_n\|_{W^{1,2}_p ((0,1) \times \R^d)} \leq C \Big(\|\boldsymbol v_n\|_{L^{d+2}((0,1)\times \R^d)} + \|\boldsymbol g\|_{W^2_{d+2}((0,1)\times \R^d)}\Big).
\end{equation}
Therefore the statement in the proposition now follows from Lemma \ref{lem:global-Lp} and the fact that $\boldsymbol g\in W^2_{d+2}((0,1)\times \R^d)$. The assertion that $\boldsymbol v_n \in W^2_2$ is proved similarly with $d+2$ above replaced by $2$.
\end{proof}

\section{Backward uniqueness}\label{sec:bu}

We will show in this section $|\nabla v^0| \neq 0$  a.e. on $[0,1]\times \Real^d$. Denote $\Bu := \nabla v^0= (u^1,\dots,u^d)$. The following result presents the properties  that $\Bu$ satisfies.

\begin{lem}\label{lem:Bu-sys}
Suppose Assumptions \ref{ass:X} and \ref{ass:BSDE} hold. Then $\Bu \in W^{1,2}_2((0,1)\times \R^d)$ and there are $V$ and $W: (0,1)\times \R^d \rightarrow \R$ with
\begin{equation}\label{VW-int}
 \|V\|_{(L^\infty+L^{d+2})((0,1)\times \R^d)} + \|W\|_{L^\infty((0,1)\times \R^d)} \leq \lambda^{-1},
\end{equation}
for some $\lambda\in(0,1)$ such that the components of $\boldsymbol u$ verify inequalities
\begin{equation}\label{u-diff-bd}
|Pu^j| \leq W |\nabla \Bu| + V|\Bu|, \quad\text{over}\  [0,1)\times\Real^d\ \text{for}\ j=1, \dots, d,
\end{equation}
and with $P:= \partial_t + \tfrac12 \nabla \cdot \left(\boldsymbol A \nabla\ \right)$.
\end{lem}
To show that $|\Bu|\neq 0$ a.e., we need the following \emph{Backward Uniqueness} result.
\begin{thm}[Backward Uniqueness]\label{thm:BU}
Suppose that
 \begin{enumerate}
  \item[(i)] the vector valued function $\Bu: [0,1]\times \R^d \rightarrow \R^d$ satisfies $\Bu \in  W^{1,2}_2 ((0,1)\times \R^d)$ and \eqref{u-diff-bd} with \eqref{VW-int} hold.
  \item[(ii)] the matrix valued function $\boldsymbol A$ satisfies \eqref{uni-ell} and is globally Lipschitz with respect to both the time and the space variables over $[0,1]\times\Real^d$.
 \end{enumerate}
 Then, if the set
$$E =\{(t,x) \in [0,1) \times \R^d \,:\, |\Bu(t,x)| =0\}$$
has positive Lebesgue measure in $(0,1)\times \R^d$, $\Bu \equiv 0$ over $[0,1]\times \R^d$
\end{thm}

\begin{rem}\label{R: unadedum} When $V$ and $W$ are bounded functions over $[0,1]\times\R^d$, the same result holds provided that $\Bu$ is in $W^{1,2}_{2,\text{loc}} ((0,1)\times \R^d)$ and $|\boldsymbol u(x,t)|\le e^{N|x|^2}$ over $[0,1]\times\R^d$, for some $N\ge 1$,  and under some more constrain conditions on $\nabla \boldsymbol A$ over $[0,1]\times\R^d$. See \cite[Theorem 1.1]{tNguyen2010}, \cite[Theorem 1.2]{WuZhang1}  and the references there in.
\end{rem}
\begin{rem}
Under the hypothesis in Theorem  \ref{thm:BU}, the combination of the reasonings behind \cite[Theorem 2 (3) and (2.20)]{EscauriazaFernandezVessella}, \cite[Theorem 3]{Fernandez03} and the proof of Theorem \ref{thm:BU} imply that if $\boldsymbol u(1,\cdot)\not\equiv 0$ over $\R^d$, then 
$$\{x\in\R^d: \boldsymbol u(t,x)=0\}$$
has zero Lebesgue measure for all $0\le t <1$. Also, the combination of the reasonings behind \cite[Theorem 3 (2)]{EscauriazaFernandezVessella},  \cite[Theorem 1.1]{HanFangHua94} and the proof of Theorem \ref{thm:BU} imply that the Hausdorff dimension of 
$$\{(t,x)\in [0,1)\times\R^{d}: \boldsymbol u(t,x)=0\}$$
is less or equal than $d$, when $\boldsymbol u(1, \cdot)\not\equiv 0$ over $\R^d$.
\end{rem}

We will first prove Theorem \ref{thm:BU} and then come back to the proof of Lemma \ref{lem:Bu-sys} at the end of this section. 
The idea of the proof for Theorem \ref{thm:BU} is the following. First, by the Lebesgue differentiation theorem, there is some $(\tau,z)\in E$ such that
\begin{equation*}
\lim_{r\to 0^+}\frac{|E\cap Q_r(\tau,z)|}{|Q_r(\tau,z)|} =1,
\end{equation*} 
where $Q_r(\tau,z)$ denotes the backward parabolic cube $[\tau,\tau+r^2]\times B_r(z)$ and $Q_r = Q_r(0,0)$. Without loss of generality, we may assume that  $(\tau,z)=(0,0)$. Then, we show that $\boldsymbol u$ must have a zero of infinite order with respect to the $(t,x)$ variables at $(0,0)$. (See \cite{Regbaoui} for the elliptic analog.) Subsequently we use the Carleman inequality for parabolic operators with variable coefficients derived in \cite[Theorem 4]{EscauriazaFernandez}\footnote{Here one could also use the Carleman inequalities in \cite{KochTataru}.} to show that $\boldsymbol u(0,\cdot)\equiv 0$ on $\Real^d$. Then, by backward uniqueness and with a second Carleman inequality (see \cite[Theorem 3]{EscauriazaFernandez}), we derive $\boldsymbol u\equiv 0$ elsewhere.
\begin{lem}\label{lem:u-sobolev}
For sufficiently small $r$ depending on $\lambda$ and $d$
\begin{equation}\label{u-norm}
 \norm{\boldsymbol u}_{L^{\frac{2(d+2)}{d}}(Q_r)} \lec r^{-1}\,  \norm{\boldsymbol u}_{L^2 (Q_{2r})},
\end{equation}
where $C$ depends on $\lambda$ and $d$.
\end{lem}
\begin{proof} Let $\varphi \in C^\infty_0 (Q_{2r})$ with $0\leq \varphi \leq 1$ and $\varphi \equiv 1$ in $Q_r$, with $Q_{2r}\subset [0,1)\times \Real^d$. Multiply $P u^i$ by $\varphi^2u^i$ and add up in $i$. Then, from the product rule
\begin{equation}\label{u2-equiv}
\begin{split}
&\varphi^2\boldsymbol u\cdot P\boldsymbol u= \varphi^2u^i\left(\partial_t u^i + \tfrac12\,\nabla\left(\boldsymbol A\nabla u^i\right)\right)\\
&= \tfrac12\, \partial_t |\varphi\boldsymbol u|^2 - |\boldsymbol u|^2 \varphi \partial_t \varphi + \tfrac14\,\nabla\cdot\left(\varphi^2 \mathbf A\nabla |\boldsymbol u|^2\right) - \tfrac12\,\mathbf A\nabla(\varphi u^i)\cdot \nabla(\varphi u^i) +  \tfrac{|\boldsymbol u|^2}2 \mathbf A\nabla\varphi\cdot\nabla\varphi
\end{split}
\end{equation}
Multiply \eqref{u2-equiv} by $-2$ and integrate the result over $[\tau, 4r^2] \times B_{2r}$ for $0\leq \tau \leq 4 r^2$ while using the fact that $\varphi=0$ on the boundary of $B_{2r}$.  Then, we get 
\begin{equation*}\label{int_u2}
\begin{split}
-2\tint_{\tau}^{4r^2} \tint_{B_{2r}}\varphi^2\boldsymbol u\cdot P\boldsymbol u\ \Id t\ \Id x=&\tint_{B_{2r}} |\varphi(\tau)\boldsymbol u(\tau)|^2 \ \Id x + \tint_{\tau}^{4r^2} \tint_{B_{2r}} \mathbf A\nabla\left(\varphi u_i\right)\cdot\nabla\left(\varphi u_i\right)\ \Id t \ \Id x \\
&+ 2\tint_{\tau}^{4r^2} \tint_{B_{2r}} |\boldsymbol u|^2 \left[\varphi \partial_t \varphi - \frac 12\,\mathbf A\nabla \varphi\cdot\nabla\varphi\right] \ \Id t \ \Id x.
\end{split}
\end{equation*}
It then follows from \eqref{u-diff-bd} that 
\begin{equation*}
\begin{split}
 -2\tint_{\tau}^{4r^2} \tint_{B_{2r}}\varphi^2\boldsymbol u\cdot P\boldsymbol u\ \Id t\ \Id x \lec & \tint_{\tau}^{4r^2} \tint_{B_{2r}}  \varphi^2 |\Bu | \big( W |\nabla \Bu | + V |\Bu |\big)\ \Id t \ \Id x\\
 \lec & \tint_{\tau}^{4r^2} \tint_{B_{2r}}  V (\varphi |\Bu|)^2+W |\varphi \Bu | |\nabla (\varphi \Bu)| +W |\varphi \nabla \varphi| |\Bu |^2\ \Id t \ \Id x.
\end{split}
\end{equation*}
The above  inequality, \eqref{uni-ell}, \eqref{VW-int}, the natural bounds satisfied by $\varphi$, $r^2|\partial_t\varphi|+r|\nabla\varphi|\lec 1$, and H\"older's inequality with $p=\frac{d+2}{2}$ imply that for $0 < r\le 1$
 \begin{multline}\label{E: int_u2}
\|\varphi \boldsymbol u\|_{L^\infty_t L^2_x([0,4r^2]\times \Real^d)}^2+ \|\nabla\left(\varphi \boldsymbol u\right)\|_{L^2([0,4r^2]\times\Real^d)}^2\lec r^{-2}\tint_{Q_{2r}} |\boldsymbol u|^2\ \Id t \ \Id x \\
+ \|V_2\|_{L^{\frac{d+2}2}(Q_{2r})}\|\varphi\boldsymbol u\|_{L^{\frac{2(d+2)}{d}}([0,4r^2]\times\Real^d)}^2+r\|\varphi \boldsymbol u\|_{L^\infty_tL^2_x([0,4r^2]\times \Real^d)}\|\nabla\left(\varphi \boldsymbol u\right)\|_{L^2([0,4r^2]\times\Real^d)},
 \end{multline}
where  $V=V_1+V_2$, with
 \begin{equation}\label{E: descomposi}
 \|V_1\|_{L^\infty((0,1)\times\Real^d)}+\|V_2\|_{L^{d+2}((0,1)\times\Real^d)}\le 2\lambda^{-1}.
 \end{equation}
The interpolation inequality in Lemma \ref{lem:ineq} (i), Jensen's inequality and \eqref{E: descomposi} give
\begin{equation}\label{phi-u-est2}
\norm{\varphi \boldsymbol u}_{L^{\frac{2(d+2)}{d}}([0,4r^2]\times\Real^d)}\leq  \norm{\varphi \boldsymbol u}_{L^\infty_tL_x^{2}([0,4r^2]\times\Real^d)} + \norm{\nabla(\varphi \boldsymbol u)}_{L^2([0,4r^2]\times\Real^d)}
\end{equation}
and
\begin{equation}\label{E: vayades}
 \|V_2\|_{L^{\frac{d+2}2}(Q_{2r})}\lec 2r\lambda^{-1}.
\end{equation}
Then, \eqref{E: int_u2} and \eqref{E: vayades} yield
 \begin{multline*}
\|\varphi \boldsymbol u\|_{L^\infty_t L^2_x([0,4r^2]\times \Real^d)}^2+ \|\nabla\left(\varphi \boldsymbol u\right)\|_{L^2([0,4r^2]\times\Real^d)}^2\lec r^{-2}\tint_{Q_{2r}} |\boldsymbol u|^2\ \Id t \ \Id x \\
 + r\left[\|\varphi \boldsymbol u\|_{L^\infty_t L^2_x([0,4r^2]\times \Real^d)}^2+\|\nabla\left(\varphi \boldsymbol u\right)\|_{L^2([0,4r^2]\times\Real^d)}^2\right].
 \end{multline*}
Now, if $r$ is small we can hide the second term on the right-hand side above on the left-hand side, which implies by \eqref{phi-u-est2}
\begin{equation*}
\norm{\varphi \boldsymbol u}_{L^{\frac{2(d+2)}{d}}([0,4r^2]\times\Real^d)}\lec \|\varphi \boldsymbol u\|_{L^\infty_t L^2_x([0,4r^2]\times \Real^d)}+ \|\nabla\left(\varphi \boldsymbol u\right)\|_{L^2([0,4r^2]\times\Real^d)}\lec r^{-1}\|\boldsymbol u\|_{L^2(Q_{2r})},
\end{equation*} 
for $0<r\le r(\lambda, d)$ and Lemma \ref{lem:u-sobolev} follows. 
\end{proof}

\begin{lem} 
If the set $E$ has positive Lebesgue measure on $[0,1)\times \Real^{d}$ and $(\tau,z)$ is a Lebesgue point of $E$, then $\boldsymbol u$ has a zero of infinite order at $(\tau,z)$; i.e., there is $R>0$ such that for all $m\ge 1$ there is $C_m$ with 
\begin{equation}\label{u-vanish}
 |\boldsymbol u(t,x)| \leq C_m (t-\tau + |x-z|^2)^{\frac{m}{2}},\quad (t,x)\in Q_R(\tau,z).
\end{equation}
\end{lem}

\begin{proof} Without loss of generality and after a translation we may assume that $(\tau,z)= (0,0)$. Then,  starting with the right-hand side of \eqref{u-norm}, 
\begin{equation}\label{E:12}
\begin{split}
&\norm{\boldsymbol u}_{L^2 (Q_{2r})} = \norm{\boldsymbol u}_{L^2 (Q_{2r}\setminus E)}\le |Q_{2r}\setminus E|^{\frac1{d+2}}\norm{\boldsymbol u}_{L^{\frac{2(d+2)}{d}}(Q_{2r}\setminus E)}\\
&\lec r\left(\frac{|Q_{2r}\setminus E|}{|Q_{2r}|}\right)^{\frac1{d+2}}\norm{\boldsymbol u}_{L^{\frac{2(d+2)}{d}}(Q_{2r})}
\end{split}
\end{equation} where the first inequality follows from Jensen's inequality. It then follows from \eqref{u-norm} and \eqref{E:12} that for $0<r\le r(\lambda, d)$ 
 \begin{equation*}\label{u-norm3}
 \norm{\boldsymbol u}_{L^{\frac{2(d+2)}{d}}(Q_r)} \lec \left(\frac{|Q_{2r}\setminus E|}{|Q_{2r}|}\right)^{\frac1{d+2}}\norm{u}_{L^{\frac{2(d+2)}{d}}(Q_{2r})},
 \end{equation*}
 where the constant $C$ is independent of $r$. Now, because $(0,0)$ is a Lebesgue point of $E$
 \begin{equation*}
 \lim_{r\rightarrow 0^+} \frac{|Q_{2r}\setminus E|}{|Q_{2r}|} = 0
 \end{equation*}
and for all $\epsilon>0$, there is some $r_\epsilon>0$ such that
 \begin{equation*}
 \norm{\boldsymbol u}_{L^{\frac{2(d+2)}{d}}(Q_r)}  \leq \epsilon \,\norm{\boldsymbol u}_{L^{\frac{2(d+2)}{d}}(Q_{2r})},\quad 0<r\leq r_\epsilon.
 \end{equation*}
 The iteration of the previous inequality implies that 
 \begin{equation*}
 \lim_{r\to 0} r^{-m} \norm{\boldsymbol u}_{L^{\frac{2(d+2)}{d}}(Q_r)}=0, \quad m\geq 1.
 \end{equation*}
 Finally, standard estimates for sub-solutions to parabolic equations \cite[Theorem 6.17]{Lieberman}),  - which are well known to extend for vector solutions to parabolic systems with a diagonal principal part - \eqref{u-diff-bd}, \eqref{VW-int} and \eqref{uni-ell} imply that with constants depending
on $\lambda$ and $d$
\begin{equation*}\label{E:101}
\max_{Q_r}|\boldsymbol u|\lec \dashint_{Q_{2r}} |\boldsymbol u|\, ds dy,\quad 0\le r\le\frac12.
\end{equation*}
The last two facts show that \eqref{u-vanish} holds for $R$ small and $(\tau, z) = (0,0)$.
\end{proof}
\begin{lem}[Strong uniqueness]\label{T: teormea2} Assume that $\boldsymbol u$ has a zero of infinite order with respect to the $(t,x)$ variables at $(0,0)$. Then, $\boldsymbol u(0,x)\equiv 0$ over $\Real^d$.
\end{lem}
Lemma \ref{T: teormea2} follows from the following Carleman inequality \cite[Theorem 4]{EscauriazaFernandez}. See also \cite{escauriaza2000}, \cite{EscauriazaVega2001}, \cite[p. 148]{lEscauriaza2004a}, \cite[\S 3]{lEscauriaza2004b} or \cite[Prop. 6.1]{Escauriaza_2003c} for similar Carleman inequalities for cases where the leading part of $P$ is the backward heat operator. { Comparing to the Carleman inequality in \cite[Theorem 4]{EscauriazaFernandez}, a term in \eqref{E:01} below, which is dropped in \cite{EscauriazaFernandez}, is utilized in the lemma below to control the unbounded part of the zero order potential.}

\begin{lem}\label{L:Carlemanclasica} Assume that $\tfrac 12\mathbf A(0,0)$ is the identity matrix. Then, 
there are $N$ and $0<\delta\le 1$ depending only on $\lambda$ and $d$ such that with $\gamma =2\alpha/\delta^2$, for each $\alpha\ge N$ there is an increasing $C^\infty$ function $\sigma: (0,1]\longrightarrow [0,+\infty)$ verifying 
\begin{equation}\label{barbaridad}
t/N\le \sigma (t)\le t,\quad 0\le t\le 1/2\gamma,
\end{equation} 
and such that the inequality
\begin{multline}\label{E:5}
\alpha^{\frac{d}{2\left(d+2\right)}}\|\sigma^{\frac12-\alpha}G^{\frac12}v\|_{L^{\frac{2\left(d+2\right)}{d}}([0,1]\times\Real^d)} +\alpha\|\sigma^{-\alpha}G^{\frac12}v\|_{L^2([0,1]\times\Real^d)}+\sqrt{\alpha}\|\sigma^{\frac12-\alpha}G^{\frac12}\nabla v\|_{L^2([0,1]\times\Real^d)}\\
\le N \|\sigma^{\frac{1}{2}-\alpha}G^{\frac12} Pv\|_{L^2([0,1]\times\Real^d)}
+ e^{N\alpha}\gamma^{\alpha +N}\left[\|v\|_{L^2([0,1]\times\Real^d)}+\|\nabla v\|_{L^2([0,1]\times\Real^d)}\right],
\end{multline}
holds for any $v\in C_0\cap W^{1,2}_2((0,\frac 1{2\gamma})\times\Real^d)$. Here $G(t,x) = t^{-\frac{d}{2}} e^{-|x|^2/{4t}}$.
\end{lem}

\begin{proof}[Proof of Lemma \ref{L:Carlemanclasica}] As in \cite[Theorem 4 and Lemma 4]{EscauriazaFernandez}, setting $G(t,x) = t^{-\frac d2}e^{-|x|^2/4t}$ and when $\tfrac12 \mathbf A(0,0)$ is the identity matrix, the Assumptions \ref{ass:X} (i)-(iii) and \eqref{uni-ell} hold,  there are $N= N(\lambda,d)\ge 1$ and $0<\delta_0=\delta(\lambda,d)<1$ such that if $\alpha\ge 2$, $0<\delta\le\delta_0$, $\gamma =\alpha/\delta^2$,
\begin{equation}\label{Defsigma}
\sigma(t)=\beta(\gamma t)/\gamma\quad\text{and}\quad \beta(t)=t\,\text{exp}\left[-\int_0^{
t}\left(1-\text{exp}\left(-\int_0^s\tau^{-\frac12}\left(\log{\left(\tfrac 1\tau\right)}\right)^{\frac 32}\ d\tau\right)\right)\tfrac
{ds}s\right],
\end{equation}
 the Carleman inequality (where integration is carried out over $(0,+\infty)\times\Real^d$),
\begin{multline}\label{E:formulacuriosa}
\alpha\gamma\tiint\sigma^{-\gamma}|v|^2G\ \Id t\ \Id x+\gamma\tiint\sigma^{1-\gamma}|\nabla v|^2G\ \Id t\ \Id x\\ \le N \tiint\sigma^{1-\gamma}|Pv|^2G\ \Id t\ \Id x
+ e^{N\gamma}\gamma^{\gamma+N}\tiint |v|^2+|\nabla v|^2\ \Id t\ \Id x,
\end{multline}
holds for any $v\in C_0^\infty((0,\tfrac{1}{2\gamma})\times\Real^d)$. The
main point about \eqref{Defsigma} and \eqref{E:formulacuriosa} is that
\begin{equation}\label{E: acotacion sigma}
\theta\le \dot\sigma(t)\le 1\quad \text{and}\quad  \theta t\le \sigma (t)\le t,\quad 0<t\le 1/(2\gamma), 
\end{equation}
holds over the support of $v$ \cite[Lemmas 4]{EscauriazaFernandez}, for some $0<\theta<1$ which depends only on the choice of $\beta$ and is independent of $\gamma\ge1$; i.e. of $\alpha\ge 2$ and $0<\delta\le\delta_0$.

In what follows we fix the value of $\delta$ and take it equal to $\delta_0$. On the other hand, the reader can verify that the authors of \cite{EscauriazaFernandez} could  have also added the integral 
\begin{equation}\label{E:01}
\tiint\sigma^{1-\alpha}\left[\partial_tv-\tfrac 12\mathbf A(0,x)\nabla \log G\cdot\nabla
v+\tfrac 12Fv - \tfrac{\alpha\dot\sigma }{2\sigma}v\right]^2G\ \Id t\ \Id x,\ \text{where}\ F=\tfrac{2|x|^2-\mathbf A(0,x)x\cdot x}{8t^2},
\end{equation}
(left aside along the proof of \cite[Theorem 4]{EscauriazaFernandez}) to the left-hand side of \eqref{E:formulacuriosa}, while integration by parts over $(0,\tau)\times\Real^d$, $0<\tau\le 1$, shows that the following identity holds
\begin{equation}\label{E:calculo}
\begin{split}
&2\tint_0^\tau\tint_{\Real^d}\sigma^{1-\alpha}\left[\partial_tv-\tfrac 12\mathbf A(0,x)\nabla \log G\cdot\nabla
v+\tfrac 12Fv - \tfrac{\alpha\dot\sigma }{2\sigma}v\right]v\,G\ \Id t\ \Id x\\
&=\tint_{\Real^d}\sigma(\tau)^{1-\alpha}|v|^2(x,\tau)G(x,\tau)\ \Id x-\tint_0^\tau\tint_{\Real^d}\dot\sigma\,\sigma^{-\alpha}|v|^2G\ \Id t\ \Id x
\\&\quad +\tint_0^\tau\tint_{\Real^d}\sigma^{1-\alpha}\left[FG-\partial_tG+\tfrac 12\nabla\cdot\left(\mathbf A(0,x)\nabla G\right)\right]|v|^2\ \Id t\ \Id x.
\end{split}
\end{equation}
The Lipschitz continuity of $\mathbf A(0,\cdot)$ and the fact that $\tfrac 12\mathbf A(0,0)$ is the identity matrix imply that
\begin{equation}\label{E:acota}
|FG-\partial_tG+\tfrac12\nabla\cdot\left(\mathbf A(0,x)\nabla G\right)|\len \tfrac{|x|}{t}\, G
\end{equation}
while from Young's inequality and the support properties of $v$
\begin{multline}\label{E:calculo2}
\left|2\tint_0^\tau\tint_{\Real^d}\sigma^{1-\alpha}\left[\partial_tv-\tfrac 12\mathbf A(0,x)\nabla \log G\cdot\nabla
v+\tfrac 12Fv - \tfrac{\alpha\dot\sigma }{2\sigma}v\right]v\,G\ \Id t\ \Id x\right|\\
\le \tiint\sigma^{1-\alpha}\left[\partial_tv-\tfrac 12\mathbf A(0,x)\nabla \log G\cdot\nabla
v+\tfrac 12Fv - \tfrac{\alpha\dot\sigma }{2\sigma}v\right]^2G\ \Id t\ \Id x\\
+\gamma^{-1}\sup_{0<\tau<1}\tint_{\Real^d}\sigma^{1-\alpha}(\tau)|v|^2(\tau,x)G(\tau,x)\ \Id x\, . 
\end{multline}
It then follows from \eqref{E:calculo}, \eqref{E:calculo2}, \eqref{E:acota} and \eqref{E: acotacion sigma} that for $\alpha\ge N$, with a possibly larger new $N$
\begin{equation}\label{E:1}
\begin{split}
&\sup_{0<\tau<1}\tint_{\Real^d}\sigma(\tau)^{1-\alpha}|v|^2(\tau,x)G(\tau,x)\ \Id x\\&
\len \tiint\sigma^{1-\alpha}\left[\partial_tv-\mathbf A(0,x)\nabla \log G\cdot\nabla
v+\tfrac 12Fv - \tfrac{\alpha\dot\sigma }{2\sigma}v\right]^2G\ \Id t\ \Id x\\&\quad +\tiint\sigma^{-\alpha}\left(1+|x|\right)|v|^2G\ \Id t\ \Id x.
\end{split}
\end{equation}
Also, the inequality
\begin{equation}\label{E:2}
\tint\tfrac{|x|^2}{8t}|h|^2G(t,x)\ \Id x\le 2t\tint|\nabla h|^2G(t,x)\ \Id x+\tfrac d2\tint |h|^2G(t,x)\ \Id x
\end{equation}
holds for all $h$ in $C_0^\infty(\Real^d)$ and $t>0$ \cite[Lemma 3]{EscauriazaFernandezVessella}. Multiply then \eqref{E:2} applied to $v(t,\cdot)$ by $\sigma^{-\alpha}$ and integrate the corresponding inequality over $[0,1]$ to get
\begin{equation}\label{E:3}
\tiint\sigma^{-\alpha}\tfrac{|x|^2}{t}|v|^2G\ \Id t\ \Id x\len \tiint\sigma^{1-\alpha}|\nabla v|^2G\ \Id t\ \Id x+\tiint\sigma^{-\alpha}|v|^2G\ \Id t\ \Id x.
\end{equation}
By \eqref{E: acotacion sigma}, the H\"older inequality, Young's  inequality and the support properties of $v$ imply that 
\begin{equation}\label{E:otra de}
\tiint\sigma^{-\alpha}|x||v|^2G\ \Id t\ \Id x\len \tiint\sigma^{-\alpha}\tfrac{|x|^2}{t}|v|^2G\ \Id t\ \Id x+\gamma^{-1}\sup_{0<\tau<1}\tint\sigma^{1-\alpha}(\tau)|v(\tau,x)|^2G(\tau,x)\ \Id x.
\end{equation}
Then, \eqref{E:1}, \eqref{E:otra de} and \eqref{E:3} yield  
\begin{equation}\label{E:121}
\begin{split}
&\sup_{0<\tau<1}\tint_{\Real^d}\sigma(\tau)^{1-\alpha}|v|^2(\tau,x)G(\tau,x)\ \Id x\\&
\len \tiint\sigma^{1-\alpha}\left[\partial_tv-\mathbf A(0,x)\nabla \log G\cdot\nabla
v+\tfrac 12Fv - \tfrac{\alpha\dot\sigma }{2\sigma}v\right]^2G\ \Id t\ \Id x\\&\quad +\tiint\sigma^{-\alpha}|v|^2G\ \Id t\ \Id x+\tiint\sigma^{1-\alpha}|\nabla v|^2G\ \Id t\ \Id x,
\end{split}
\end{equation}
Also, the triangle inequality and \eqref{E: acotacion sigma} imply that
\begin{equation}\label{E: casocasi}
\tiint\sigma^{1-\alpha}|\nabla\left(G^{1/2}v\right)|^2\ \Id t\ \Id x\len\tiint\sigma^{1-\alpha}|\nabla v|^2G\ \Id t\ \Id x+\tiint\sigma^{-\alpha}\tfrac{|x|^2}{t}|v|^2G\,.
\end{equation}

It now follows  from \eqref{E:formulacuriosa}, \eqref{E:121}, the fact stated in \eqref{E:01}, \eqref{E: casocasi} and \eqref{E:3} that
\begin{multline}\label{E:4}
\|\sigma^{\frac{1-\alpha}2}G^{\frac 12}v\|_{L^\infty_t L^2_x([0,1]\times\Real^d)}+\sqrt{\alpha}\|\nabla(\sigma^{\frac{1-\alpha}2}G^{\frac12}v)\|_{L^2([0,1]\times\Real^d)} +\alpha\|\sigma^{-\frac\alpha 2}G^{\frac12}v\|_{L^2([0,1]\times\Real^d)}\\+\sqrt{\alpha}\|\sigma^{\frac{1-\alpha}2}G^{\frac12}\nabla v\|_{L^2([0,1]\times\Real^d)}
\le N \|\sigma^{\frac{1-\alpha}2}G^{\frac12} Pv\|_{L^2([0,1]\times\Real^d)}
\\+ e^{N\alpha}\gamma^{\frac\alpha 2 + N}\left[\|v\|_{L^2([0,1]\times\Real^d)}+\|\nabla v\|_{L^2([0,1]\times\Real^d)}\right].
\end{multline}
Then, by Lemma \ref{lem:ineq} (i) with $[a,b]=[0,1]$ and the control we have on the first and the second terms on the left-hand side of \eqref{E:4}, we get
\begin{multline}\label{E:33}
\alpha^{\frac{d}{2\left(d+2\right)}}\|\sigma^{\frac{1-\alpha}2}G^{\frac12}v\|_{L^{\frac{2\left(d+2\right)}{d}}([0,1]\times\Real^d)} +\alpha\|\sigma^{-\frac\alpha 2}G^{\frac12}v\|_{L^2([0,1]\times\Real^d)}+\sqrt{\alpha}\|\sigma^{\frac{1-\alpha}2}G^{\frac12}\nabla v\|_{L^2([0,1]\times\Real^d)}\\
\le N \|\sigma^{\frac{1-\alpha}2}G^{\frac12} Pv\|_{L^2([0,1]\times\Real^d)}
+ e^{N\alpha}\gamma^{\frac\alpha 2+N}\left[\|v\|_{L^2([0,1]\times\Real^d)}+\|\nabla v\|_{L^2([0,1]\times\Real^d)}\right],
\end{multline}
when $\alpha\ge N(\lambda,d)$ and $v\in C_0^\infty(0,\frac1{2\gamma})\times\Real^d$. 

Then, mollification and the Dominated Convergence Theorem show that \eqref{E:33} holds for functions $v\in C_0\cap W^{1,2}_2((0,\frac 1{2\gamma})\times\Real^d)$. Finally, Lemma \ref{L:Carlemanclasica} follows from \eqref{E:33} after one replaces $\alpha$ by $2\alpha$ and redefines $\gamma$ as $2\alpha/\delta^2_0$ in \eqref{E:33} .  
\end{proof}
\begin{proof}[Proof of Lemma \ref{T: teormea2}]
After the constant change of variables $x=\boldsymbol R\, y$, $\boldsymbol R =\frac1{\sqrt2}\boldsymbol A^{1/2}(0,0)$, which satisfies
\begin{equation}\label{E:conservadistancias}
N^{-1}|y|\le |x|\le N|y|,\quad x\in\Real^d,\ \text{with}\ N=N(\lambda),
\end{equation}
we may after abusing  of the notation, assume that in the original coordinates $(t,x)$, the matrix $\tfrac 12\mathbf A(0,0)$ is the identity. 

In what follows $N$ and $\gamma$ are the constants defined in Lemma \ref{L:Carlemanclasica} so that it holds for any  $\alpha\ge N$. Let now $\boldsymbol u$ satisfy the conditions in Lemma \ref{T: teormea2}, $\boldsymbol v_\epsilon=\phi_\epsilon(t)\,\theta(x)\boldsymbol u$, where $\theta\in
C_0^\infty(\Real^d)$ and $\phi_\epsilon\in C_0^\infty(\Real)$ verify $\theta=1$ for $|x|\le 1$, $\theta=0$ for $|x|\ge 2$, 
$\phi_\epsilon=1$ when $\epsilon\le t\le 1/(6\gamma)$ and $\phi_\epsilon=0$ when $t\le\epsilon/2$ or $t\ge 1/(4\gamma)$. Apply now  the inequality \eqref{E:5} to each component  $\boldsymbol v_\epsilon^i$ of $\boldsymbol v_\epsilon$. Then, after adding up in $i$, we get
\begin{multline}\label{E:25}
\alpha^{\frac{d}{2\left(d+2\right)}}\|\sigma^{\frac 12-\alpha}G^{\frac12}\boldsymbol v_\epsilon\|_{L^{\frac{2\left(d+2\right)}{d}}([0,1]\times\Real^d)} +\alpha\|\sigma^{-\alpha}G^{\frac12}\boldsymbol v_\epsilon\|_{L^2([0,1]\times\Real^d)}\\
+\sqrt{\alpha}\|\sigma^{\frac12-\alpha}G^{\frac12}\nabla \boldsymbol v_\epsilon\|_{L^2([0,1]\times\Real^d)}
\le N \|\sigma^{\frac12 -\alpha}G^{\frac12} P\boldsymbol v_\epsilon\|_{L^2([0,1]\times\Real^d)}\\
+ e^{N\alpha}\gamma^{\alpha+N}\left[\|\boldsymbol v_\epsilon\|_{L^2([0,1]\times\Real^d)}+\|\nabla \boldsymbol v_\epsilon\|_{L^2([0,1]\times\Real^d)}\right].
\end{multline}
After writing $V$ as $V_1+V_2$ with $V_1$ and $V_2$ as in \eqref{E: descomposi}, from H\"older's inequality
\begin{multline*}
\|\sigma^{\frac12 -\alpha}G^{\frac12} V\boldsymbol v_\epsilon\|_{L^2([0,1]\times\Real^d)}+ \|\sigma^{\frac12 -\alpha}G^{\frac12} W\nabla\boldsymbol v_\epsilon\|_{L^2([0,1]\times\Real^d)}\\\le
\|V_1\|_{L^{\infty}([0,1]\times\Real^d)}\|\sigma^{-\alpha}G^{\frac 12}\boldsymbol v_\epsilon\|_{L^{2}([0,1]\times\Real^d)}+\|V_2\|_{L^{d+2}([0,1]\times\Real^d)}\|\sigma^{\frac 12-\alpha}G^{\frac 12}\boldsymbol v_\epsilon\|_{L^{\frac{2\left(d+2\right)}{d}}([0,1]\times\Real^d)}\\+\|W\|_{L^{\infty}([0,1]\times\Real^d)}\|\sigma^{\frac 12-\alpha}G^{\frac 12}\nabla\boldsymbol v_\epsilon\|_{L^{2}([0,1]\times\Real^d)},
\end{multline*}
it is possible to hide on the left-hand side of \eqref{E:25} the term $W|\nabla\boldsymbol v_\epsilon|+V\mathbf |v_\epsilon |$ arising on the right-hand side of the inequality
\begin{multline*}
|P\boldsymbol v_\epsilon|\len W|\nabla\boldsymbol v_\epsilon|+V\mathbf |v_\epsilon|+|\nabla\boldsymbol u|1_{[0,\frac1{4\gamma}]\times (B_2\setminus B_1)}\\+|\boldsymbol u|\left(\alpha\, 1_{[0,\frac1{4\gamma}]\times (B_2\setminus B_1)\cup [\frac1{6\gamma},\frac1{4\gamma}]\times B_2}+\epsilon^{-1}1_{[\frac\epsilon2,\epsilon]\times B_2}\right),
\end{multline*}
after one requires $\alpha$ to be sufficiently large. Also, from \eqref{barbaridad} there is $N\ge 1$ independent of $\alpha\ge 1$ such that
\begin{equation*}
 \sigma^{-\alpha}G^{\frac12}\le e^{N\alpha}\alpha^\alpha,\quad\text{when}\  (x,t)\in [0,\tfrac1{4\gamma}]\times (B_2\setminus B_1)\cup [\tfrac1{6\gamma},\tfrac1{4\gamma}]\times B_2.
 \end{equation*} 
 Altogether, we get that for $\alpha\ge N$, for some $N$ depending only on $\lambda$ and $d$
\begin{multline*}\label{E:6}
\|t^{-\alpha}e^{-|x|^2/8t}\boldsymbol u\|_{L^2([\epsilon,\frac1{6\gamma}]\times B_1)}\le  e^{N\alpha}\alpha^{\alpha}\left[\|\boldsymbol u\|_{L^2([0,1]\times B_2)}+\|\nabla \boldsymbol u\|_{L^2([0,1]\times B_2)}\right]\\+\epsilon^{-1}e^{N\alpha}\alpha^{\alpha}\|t^{-\alpha}e^{-|x|^2/8t}\boldsymbol u\|_{L^2([\frac\epsilon 2,\epsilon]\times B_2)}.
\end{multline*}
Next, the fact that $\boldsymbol u$ has a zero of infinite order at $(0,0)$ with respect to the $(t,x)$ variable, implies that the last term above tends to zero when $\epsilon\to 0^+$, and we get
\begin{equation}\label{E:7}
\|t^{-\alpha}e^{-|x|^2/8t}\boldsymbol u\|_{L^2([0,1]\times B_1)}\le  Ne^{N\alpha}\alpha^{\alpha}\left[\|\boldsymbol u\|_{L^2([0,1]\times B_2)}+\|\nabla \boldsymbol u\|_{L^2([0,1]\times B_2)}\right],\ \text{for all}\ \alpha\ge 0.
\end{equation}
Also, from
Stirling's formula
$\alpha^\alpha\le Ne^{N\alpha}\alpha!$, for all $\alpha\in \N$ \cite{Ahlfors}. Then, after multiplying \eqref{E:7} by $1/\left(\alpha !e^{2N\alpha}2^{\alpha}\right)$ and adding up over $\alpha\ge 0$, we derive that for some large new $N$ as above
\begin{equation*}\label{E:8}
\|e^{1/Nt-|x|^2/8t}\boldsymbol u\|_{L^2([0,1]\times B_1)}\le  N\left[\|\boldsymbol u\|_{L^2([0,1]\times B_2)}+\|\nabla \boldsymbol u\|_{L^2([0,1]\times B_2)}\right].
\end{equation*}
In particular,
\begin{equation}\label{E:9}
\|e^{1/2Nt}\boldsymbol u\|_{L^2([0,1]\times B_{2/\sqrt{N}})}\le  N\left[\|\boldsymbol u\|_{L^2([0,1]\times B_2)}+\|\nabla \boldsymbol u\|_{L^2([0,1]\times B_2)}\right].
\end{equation}
Finally, the standard estimates for sub-solutions to parabolic inequalities (see  \cite[Theorem 6.17]{Lieberman}) imply that with constants depending
on $\lambda$ and $d$
\begin{equation}\label{E:10}
|\boldsymbol u(x,t)|\len \tfrac1{t^{\frac d2 +1}}\tint_t^{2t}\tint_{B_{\sqrt{t}}(x)} |\boldsymbol u|\, ds dy,\ \text{when}\ x\in\Real^d\ \text{and}\ 0\le t\le 1/4
\end{equation}
and from \eqref{E:9} and \eqref{E:10}
\begin{equation*}
|\boldsymbol u(x,t)|\le Ne^{-1/Nt}\left[\|\boldsymbol u\|_{L^2([0,1]\times B_2)}+\|\nabla \boldsymbol u\|_{L^2([0,1]\times B_2)}\right],\ \text{when}\  |x|\le 1/\sqrt{N}\ \text{and}\ 0<t<\frac 14. 
\end{equation*}
Thus, $\boldsymbol u$ vanishes to infinite order with respect to the $(t,x)$ variables at all points $(0,y)$, with $|y|<1/\sqrt{N}$. 

Because of \eqref{E:conservadistancias}, one can now repeat the same reasoning but now with center at the point $(0,y)$ with $|y| < 1/\sqrt{N}$ and find that $\boldsymbol u$ vanishes to infinite order with respect to the $(t,x)$ variables at all points $(0,y)$, with $|y|<2/\sqrt{N}$. Eventually, one derives that $\boldsymbol u(0,\cdot)\equiv 0$ over $B_{m/\sqrt{N}}$ for all $m\ge 1$, which confirms Lemma \ref{T: teormea2}.  
\end{proof}

{ After establishing the unique continuation property for $\boldsymbol{u}(0, \cdot)$ on $\Real^d$, we are going to prove the backward uniqueness for $\boldsymbol{u}$ on $[0,1]\times \Real^d$. To this end, the Carleman inequality in the following lemma is important. In particular, the time derivative term on the left-hand side of \eqref{carleman1} is needed to handle the unbounded zero order potential.}

\begin{lem}\label{lem:carleman-v}
 There is $C=C(\lambda,d)$ such that the inequality
 \begin{multline}\label{carleman1}
\norm{e^{M(t+\delta)} \partial_t f}_{L^2((0,T)\times\Real^d)}  + \sqrt{M}  \norm{e^{M(t+\delta)} \nabla f}_{L^2((0,T)\times\Real^d)}\\+ \sqrt{\alpha}\, M\,\norm{e^{M(t+\delta)} f}_{L^2((0,T)\times\Real^d)}\lec \norm{e^{M(t+\delta)} (t+\delta)^{-\alpha} P v}_{L^2((0,T)\times\Real^d)} 
 \end{multline}
holds for any $\alpha >0$, $M\geq \lambda^{-2}$, $0<\delta\le \frac{1}{4M}$ and $v\in W^{1,2}_2((0,T)\times \Real^d)$ with $v(0,\cdot)\equiv 0$ and $v(T,\cdot)\equiv 0$  over $\Real^d$, when $f = (t+\delta)^{-\alpha} v$ and $T =\tfrac{1}{4M} -2\delta$. 
\end{lem}
\begin{proof}
 Let $\gamma:[0,+\infty)\rightarrow \Real_+$ and $\sigma: [0,+\infty)\rightarrow \Real$ be two smooth functions to be chosen. For $v\in C^{\infty}_0((-\delta,+\infty)\times \Real^d)$ set $f=e^{\sigma(t)} v$. Then 
 \[
  e^{\sigma(t)} P v = e^{\sigma(t)} P e^{-\sigma(t)} f = \partial_t f - \dot{\sigma}(t) f + \nabla\cdot \big(\mathbf A(t,x) \nabla f\big) = \partial_t f - \mathcal{S} f,
 \]
 where $\mathcal{S} = \dot{\sigma}(t) - \nabla \cdot\left(\mathbf A(t,x) \nabla\ \right)$ is for each fixed time $t$ a symmetric operator; i.e., $\int_{\Real^d} (\mathcal{S} \phi) \psi \ \Id x = \int_{\Real^d} \phi (\mathcal{S} \psi) \ \Id x$, for any $\phi, \psi\in C^\infty_0(\Real^d)$.
 
 Then,
 \begin{align*}
&\norm{\sqrt{\gamma}\, e^{\sigma} P v}^2_{L^2((-\delta,+\infty)\times\Real^d)} = \norm{\sqrt{\gamma} e^{\sigma} P e^{-\sigma} f}^2_{L^2((-\delta,+\infty)\times\Real^d)}  = \norm{\sqrt{\gamma} (\partial_t f - \mathcal{S} f)}^2_{L^2((-\delta,+\infty)\times\Real^d)} \\
  &= \norm{\sqrt{\gamma} \partial_t f}^2_{L^2((-\delta,+\infty)\times\Real^d)} + \norm{\sqrt{\gamma} \mathcal{S} f}^2_{L^2((-\delta,+\infty)\times\Real^d)} - 2 \tiint_{(-\delta,+\infty)\times\Real^d} \gamma\partial_t f \mathcal{S} f \ \Id t \ \Id x.
 \end{align*}
 For the third term above, using integration by parts and because $\mathcal{S}$ is a symmetric operator, 
 \begin{align*}
&- 2 \tiint_{(-\delta,+\infty)\times\Real^d} \gamma \partial_t f \mathcal{S} f \ \Id t \ \Id x = -\tiint_{(-\delta,+\infty)\times\Real^d} \gamma\big[ \partial_t f \mathcal{S} f+ \partial_t f \mathcal{S}f\big]\ \Id t \ \Id x\\
& = \tiint_{(-\delta,+\infty)\times\Real^d} \big[-\partial_t\left(\gamma f \mathcal{S}f\right) +\dot{\gamma}f \mathcal{S} f + \gamma f \mathcal{S}_t f +\gamma f\mathcal S\partial_tf- \gamma f\mathcal{S} \partial_t f\big] \ \Id t \ \Id x\\
&=  \tiint_{(-\delta,+\infty)\times \Real^d} \big(\dot{\gamma} \mathcal{S}  + \gamma\mathcal{S}_t \big) f \ \Id t\ \Id x,
 \end{align*}
 where $\mathcal{S}_t = \ddot{\sigma}(t) - \nabla \cdot (\partial_t \mathbf A(t,x) \nabla)$. From the definition of $\mathcal{S}$ and $\mathcal{S}_t$, we have 
 \[
  \dot{\gamma} \mathcal{S} + \gamma \mathcal{S}_t = \tfrac{d}{dt}\big(\dot{\sigma} \gamma\big) - \nabla \cdot \Big(\big(\gamma\partial_t \mathbf A + \dot{\gamma} \mathbf A\big) \nabla\  \Big).
 \]
 Combining the previous three identities, we get 
 \begin{equation}\label{u-est1}
 \begin{split}
  \norm{\sqrt{\gamma}\, e^{\sigma} P v}^2_{L^2((-\delta,+\infty)\times\Real^d)}  = &\tiint_{(-\delta,+\infty)\times\Real^d} \gamma |\partial_t f|^2 \ \Id t \ \Id x + \tiint_{(-\delta,+\infty)\times\Real^d} \gamma |\mathcal{S} f|^2 \ \Id t \ \Id x \\
  & + \tiint_{(-\delta,+\infty)\times\Real^d} \tfrac{d}{dt} (\dot{\sigma} \gamma) |f|^2 + \big(\gamma\partial_t \mathbf A + \dot{\gamma} \mathbf A\big)\nabla f \cdot \nabla f \ \Id t \ \Id x.
 \end{split}
 \end{equation}
 
 Take now $\sigma(t) = -\alpha \log(t+\delta)$ and $\gamma(t) = e^{2M(t+\delta)}$. Then, because $\lambda^2 M \geq 1$, we have  
 \begin{equation}\label{E:61}
 \tfrac{d}{dt}(\dot{\sigma} \gamma) = \alpha (t+\delta)^{-2}e^{2M(t+\delta)} [1-2M(t+\delta)]\geq \tfrac\alpha2 (t+\delta)^{-2} e^{2M(t+\delta)},\ \text{over}\ (-\delta, T+\delta)
 \end{equation}
 and
 \begin{equation}\label{E:62}
 \big(\gamma\partial_t \mathbf A + \dot{\gamma} \mathbf A \big) \xi\cdot \xi \geq M\lambda\, e^{2M(t+\delta)} |\xi |^2,\ \text{for}\ \xi\in \Real^d,\ \text{and}\ (t,x)\in\Real^{1+d}.
 \end{equation}
 Then, \eqref{E:61} and \eqref{E:62} confirm that the inequality
  \begin{multline}\label{carleman29}
\norm{e^{M(t+\delta)} \partial_t f}_{L^2((-\delta,T+\delta)\times\Real^d)}  + \sqrt{M}  \norm{e^{M(t+\delta)} \nabla f}_{L^2((-\delta,T+\delta)\times\Real^d)}\\+ \sqrt{\alpha}\,M\norm{ e^{M(t+\delta)} f}_{L^2((-\delta,T+\delta)\times\Real^d)}\le  N\norm{e^{M(t+\delta)} (t+\delta)^{-\alpha} P v}_{L^2((-\delta,T+\delta)\times\Real^d)} 
 \end{multline}
 holds for any $v\in C^{\infty}_0((-\delta,T+\delta)\times \Real^d)$ after dropping the integral 
 \begin{equation*}
 \tiint_{(-\delta, +\infty)\times\Real^d} \gamma|\mathcal{S} f|^2 \ \Id t \ \Id x
 \end{equation*}
 from the right-hand side of \eqref{u-est1}. 
 
 Next, when $v\in W^{1,2}_2((0,T)\times \Real^d)$, the identity
 \begin{equation*}
 v(t_2,x)-v(t_1,x) = \int_{t_1}^{t_2}\partial_tv(t,x)\ \Id t
 \end{equation*} 
 and the Minkowski inequality show that
 \begin{equation}\label{eq:continuity}
 \|v(t_2,\cdot)-v(t_1,\cdot)\|_{L^2(\Real^d)} \le \sqrt{|t_2-t_1|}\,\|\partial_tv\|_{L^2((0,T)\times\Real^d))},\quad\text{for}\  t_1, t_2\in [0,T],
 \end{equation} 
 and $v\in C([0,T], L^2(\Real^d)$. Now, if $v(0,\cdot)$ and $v(T,\cdot)$ are a.e. zero over $\Real^d$, the extension of $v$ to $\Real^{1+d}$ as $v=0$ for $t\le 0$ and $t\ge T$ (abusing the notation we continue to denote  it $v$) belongs to $W^{1,2}_2(\Real^{1+d})$. 
 
 Let now $v_{\epsilon, R}=\,\theta_R\,(v\ast\phi_\epsilon)$ be a compactly supported in space mollification of $v$, where $\theta_R(x) =\theta(x/R)$ for some $\theta\in C_0^\infty(B_2)$, with $\theta =1$ over $B_1$ and $\phi\in C_0^\infty( (-1,1)\times B_1)$ is a standard mollifier with integral $1$ over $\Real^{1+d}$. Then, for $\epsilon$ small and $R>0$, $v_{\epsilon, R}$ lies in $C^{\infty}_0((-\delta/2,T+\delta)\times \Real^d)$ and $v_{\epsilon,R}$ converges in a dominated way as $R\to +\infty$ and $\epsilon\to 0^+$ to $v$ in $W^{1,2}_2(\Real^{1+d})$. Also $f_{\epsilon,R} =  (t+\delta)^{-\alpha} v_{\epsilon,R}$ converges in a dominated way to $f = (t+\delta)^{-\alpha}v$ in $W^{1,2}_2((-\delta/2,T+\delta)\times\Real^d)$. 
 
 Finally, for $\epsilon >0$ small, we can plug in $v_{\epsilon,R}$ in the inequality \eqref{carleman29}, where $t+\delta\ge \delta/2$. Then, letting first $R\to+\infty$ and after $\epsilon\to 0^+$, the dominated convergence theorem confirms that Lemma \ref{lem:carleman-v} holds.\end{proof}

\begin{lem}[Backward Uniqueness]\label{T: teormea1} If $\boldsymbol u(0, \cdot) \equiv 0$ on $\Real^d$, then $\boldsymbol u\equiv 0$ on $[0,1]\times\Real^d$.
\end{lem}

\begin{proof}
 For $\alpha\ge 1$, $M>0$ and $0<\delta\le \frac{1}{40M}$ small, let $T =\tfrac{1}{4M}-2\delta\ge \tfrac1{5M}$ be as in Lemma \ref{lem:carleman-v}. Consider a smooth truncation function $\phi:[0,1]\rightarrow \Real$ such that $0\leq \phi \leq 1$, $\phi(t) =1$ when $0\leq t\leq \tfrac{1}{8M}$ and $\phi(t) =0$ when $t\geq \tfrac{1}{6M}$. Set $\boldsymbol v(t,x) = \phi(t)\boldsymbol u(t,x)$. We first apply Lemma \ref{lem:carleman-v} to each component $v_i$, $i =1,\dots,I$, of $\boldsymbol v$, which lies in $W^{1,2}_2([0,T]\times \Real^d)$, $v(0,\cdot)\equiv 0$ and $v(T,\cdot)\equiv 0$  over $\Real^d$. To this end, we have
 \begin{equation}\label{Pv}
  Pv_i = \phi\, P u_i + \dot{\phi}\,u_i=: \text{I} + \text{II}.
 \end{equation}
 For II, by the construction of $\phi$, $|\dot{\phi}|\lec M$ and $\dot{\phi}\neq 0$ only when $t$ lies in the interval $[\tfrac1{8M},\tfrac{1}{6M}]$. Therefore, for $0< t+\delta \leq T+\delta \leq \tfrac{1}{4M}$, we have
 \begin{equation}\label{Pv-estII}
  \big|e^{M(t+\delta)}(t+\delta)^{-\alpha} \dot{\phi}\, \theta_R u_i\big| \lec (8M)^{1+\alpha} |\boldsymbol u|.
 \end{equation}
 For I and due to \eqref{u-diff-bd}
 \begin{equation}\label{Pv-estI}
  e^{M(t+\delta)}(t+\delta)^{-\alpha} \phi\, \big|P u_i\big| \lec   W \big|(t+\delta)^{-\alpha} \phi \,\nabla\boldsymbol u| +  V \Big|(t+\delta)^{-\alpha} \phi\, \boldsymbol u\Big|\lec W |\nabla\boldsymbol f| + V|\boldsymbol f|
 \end{equation}
 where $\boldsymbol f = (t+\delta)^{-\alpha} \phi\,  \boldsymbol u$. Applying Lemma \ref{lem:carleman-v} to each $v_i$, $i=1,\dots, d$ and adding up over $i$ on both sides of \eqref{carleman1}, we obtain from \eqref{Pv}-\eqref{Pv-estI} and \eqref{E: descomposi} that the following holds
 \begin{equation}\label{long-est}
 \begin{split}
  & \norm{\partial_t \boldsymbol f}_{L^2((0,T)\times\Real^d)} + \sqrt{\alpha}\,M\norm{\boldsymbol f}_{L^2((0,T)\times\Real^d)} + \sqrt{M} \norm{\nabla \boldsymbol f}_{L^2((0,T)\times\Real^d)}\\
  &\lec (8M)^{1+\alpha} \norm{\boldsymbol u}_{L^2((0,1)\times\Real^d)} + \norm{\nabla \boldsymbol f}_{L^2((0,T)\times\Real^d)}+  \norm{\boldsymbol f}_{L^2((0,T)\times\Real^d)}
+\norm{V_2\boldsymbol f}_{L^2((0,T)\times\Real^d)}.
 \end{split}
 \end{equation}

For the last term on the right-hand side above, H\"{o}lder's inequality and \eqref{E: descomposi} yield
\begin{equation*}
 \norm{V_2\boldsymbol f}_{L^2((0,T)\times\Real^d)} \le \|V_2\|_{L^{d+2}((0,T)\times\Real^d)}\norm{\boldsymbol f}_{L^{\frac{2(d+2)}{d}}((0,T)\times\Real^d)}\lec \norm{\boldsymbol f}_{L^{\frac{2(d+2)}{d}}((0,T)\times\Real^d)} .
\end{equation*}
Applying Lemma \ref{lem:ineq} (i) and (ii) to the right-hand side of the previous inequality, we get 
\begin{equation}\label{Vf-holder}
\begin{split}
  &\norm{\boldsymbol f}_{L^{\frac{2(d+2)}{d}}((0,T)\times\Real^d)} \lec  \norm{\nabla \boldsymbol f}^{\frac{d}{d+2}}_{L^2((0,T)\times\Real^d)} \norm{\partial_t \boldsymbol f}^{\frac{2}{d+2}}_{L^2((0,T)\times\Real^d)}\\
  &\lec  \epsilon^{-\frac{d+2}{d}}  \norm{\nabla \boldsymbol f}_{L^2((0,T)\times\Real^d)} +  \epsilon^{\frac{d+2}{2}} \norm{\partial_t \boldsymbol f}_{L^2((0,T)\times\Real^d)},
 \end{split}
\end{equation}
for any $\epsilon>0$. Then,  the last three terms on the right-hand side of \eqref{long-est} are bounded by 
\begin{equation}\label{WVf-ub}
 \le C \left[\norm{\boldsymbol f}_{L^2((0,T)\times\Real^d))} +  \left(1 +\epsilon^{-\frac{d+2}{d}}\right) \norm{\nabla \boldsymbol f}_{L^2(\Rt)}+ \epsilon^{\frac{d+2}{d}} \norm{\partial_t \boldsymbol f}_{L^2((0,T)\times\Real^d))}\right],
 \end{equation}
 with $C=C(\lambda, d)$.
Choose then $\epsilon$ sufficiently small so that $C \epsilon^{\tfrac{d+2}{d}} \leq 1$ Then, choose $M$ sufficiently large such that 
\[
\sqrt{M} \geq 2C \left( 1+ \epsilon^{-\frac{d+2}{d}}\right).
\]
Then \eqref{WVf-ub} shows that the last three terms on the left-hand side of \eqref{long-est} 
are dominated by the left-hand side of \eqref{long-est}, if we choose and fix a sufficiently  large value of $M$.

As a result, we have from \eqref{long-est} that for that value of $M$, $\alpha \ge 1$, $0<\delta\le\tfrac1{40M}$ and with $T= \tfrac1{4M}-2\delta\ge \frac1{5M}$, we have
 \begin{equation}\label{long-est2}
\sqrt{\alpha}\norm{(t+\delta)^{-1-\alpha}\phi\,\boldsymbol u}_{L^2((0,\tfrac1{5M})\times\Real^d)} \leq   C(8M)^{(1+\alpha)} \norm{\boldsymbol u}_{L^2(\Rt)}.
  \end{equation}
 Letting then $\delta $ tend to zero and recalling that $\phi\equiv 1$, when $t\le \tfrac 1{16M}$, where $t^{-1-\alpha}\ge (16M)^{1+\alpha}$, we get from \eqref{long-est2}
  \begin{equation*}
   \norm{\boldsymbol u}_{L^2([0,\tfrac{1}{16M}]\times \Real^d)} \leq C\, 2^{-\alpha}\norm{ \boldsymbol u}_{L^2([0,1]\times \Real^d)},\quad\text{when}\ \alpha\ge 1.
  \end{equation*}
  Sending $\alpha\rightarrow \infty$ we conclude that $\norm{ \boldsymbol u}_{L^2([0,\tfrac{1}{16M}]\times \Real^d)}=0$. Now iterating the same reasoning over the time interval $[\tfrac{1}{16M},1)$, as many times as it is necessary, we derive $\boldsymbol u\equiv 0$ over $[0,1)\times\Real^d$. To reach the terminal time $t=1$, applying the estimate \eqref{eq:continuity} to $\boldsymbol u$, we obtain $\boldsymbol u \in C([0,1], L^2(\R^d))$. Therefore $\boldsymbol u(1, \cdot) \equiv 0$ and we confirm the statement of Lemma \ref{T: teormea1}.
 \end{proof}
 
 \begin{lem}\label{lem:ineq} There is a constant $C$ depending on $d\ge 1$ such that the following inequalities hold for any interval  $[a,b]$ in $\Real$.
\begin{enumerate}
\item[(i)] $\norm{f}_{L^{\frac{2(d+2)}{d}}([a,b]\times \Real^d)}\lec \norm{\nabla f}^{\frac{d}{d+2}}_{L^2([a,b]\times \Real^d)} \|f\|^{\frac{2}{d+2}}_{L^\infty_t L^2_x([a,b]\times \Real^d)}$. 
\item[(ii)] $\norm{f}_{L_t^\infty L_x^2([a,b]\times \Real^d)} \leq \sqrt{b-a}\,\norm{\partial_t f}_{L^2([a,b]\times \Real^d)}$, when $f(a,\cdot)\equiv 0$ over $\Real^d$. 
\end{enumerate}
\end{lem}
\begin{proof} For $\alpha\in (0,1)$, let $\Lambda_\alpha$ be the fractional differential operator and $\mathcal{I}_\alpha$ the fractional integral operator. Then $f= \mathcal{I}_{\alpha} \Lambda_\alpha f$, $\widehat{\Lambda_\alpha f} = |\xi|^\alpha \hat{f}$, and $\widehat{\mathcal{I}_{\alpha} g} = |\xi|^{-\alpha} \hat{g}$, where $\hat{\cdot}$ is the Fourier transform. Young's inequality implies that for any $\epsilon>0$ and $\xi$ in $\Real^d$, $|\xi|^{2\alpha} \leq \epsilon^{2\alpha} + \epsilon^{2\alpha -2}|\xi|^2$. Then 
\[
\norm{\Lambda_\alpha f(t, \cdot)}_{L^2(\Real^d)} \leq \epsilon^\alpha \norm{f(t,\cdot)}_{L^2(\Real^d)} + \epsilon^{\alpha-1} \norm{\nabla f(t, \cdot)}_{L^2(\Real^d)}.
\]
Minimizing with respect to $\epsilon >0$ the above right-hand side, we obtain 
\begin{equation}\label{frac-diff}
\norm{\Lambda_\alpha f(t, \cdot)}_{L^2(\Real^d)} \leq \norm{\nabla f(t, \cdot)}^\alpha_{L^2(\Real^d)} \norm{f(t,\cdot)}^{1-\alpha}_{L^2(\Real^d)}.
\end{equation}
On the other hand, when $\frac12 -\frac1r = \frac{\alpha}{d}$ and $0 < \alpha < d$, $\mathcal{I}_{\alpha}$ maps $L^2(\Real^2)$ into $L^r(\Real^2)$ \cite[p. 119]{Stein70}, i.e., there exists a constant $C$ depending on $r$ such that 
\begin{equation}\label{I-embed}
\norm{\mathcal{I}_{\alpha} g(t, \cdot)}_{L^r(\Real^d)} \lec \norm{g(t,\cdot)}_{L^2(\Real^d)}.
\end{equation}
Combining \eqref{frac-diff},\eqref{I-embed}, together with $f= \mathcal{I}_{\alpha} \Lambda_\alpha f$ and choosing $r=\frac{2(d+2)}{d}$, $\alpha =\frac d{d+2}$, we obtain 
\[
\norm{f(t,\cdot)}_{L^{\frac{2(d+2)}d}(\Real^d)} \lec \norm{\nabla f(t,\cdot)}^{\frac d{d+2}}_{L^2(\Real^d)} \norm{f(t, \cdot)}^{\frac2{d+2}}_{L^2(\Real^d)}.
\]
Taking the $\frac{2(d+2)}{d}$ power of both sides and integrating with respect to time over $[a,b]$ we confirm the statement in (i).
	
By the Fundamental Theorem of Calculus $f(t,x) = \int_a^t \partial_t f(s, x) ds$, when $f(a,x)\equiv 0$ and
\begin{equation*}
|f(t,x)| \leq \big(\tint_a^t |\partial_t f(s,x)|^2 ds\big)^{\frac12} \sqrt{t-a},
\end{equation*}
which implies 
\begin{equation*}
\norm{f}_{L_t^\infty L_x^2([a,b]\times\Real^d)} \leq \sqrt{b-a}\norm{\partial_t f}_{L^2([a,b]\times\Real^d)}.
\end{equation*}
\end{proof}

\begin{proof}[Proof of Lemma \ref{lem:Bu-sys}]
 It follows from the discussion at the beginning of Section \ref{sec:limit_soln} that $\Bu\in L^\infty ((0,1)\times \Real^d)$ and from Proposition \ref{prop:global-sobolev} that $\Bu, \nabla \Bu \in  L^2 ((0,1) \times \R^d)$. 
 
 Taking the derivative with respect to $x^n$ on both sides of the equation satisfied by $v^0$ in \eqref{v-pde-sys}, we obtain
 \begin{equation}\label{u-sys}
 \begin{split}
  &\partial_t u^n + \tfrac12 \tsum_{j,k} \partial_{x^j} \big(A^{jk} \, \partial_{x^k} u^n\big) \\
  & = \tfrac12 \tsum_{j,k} \partial_{x^j} A^{jk} \, \partial_{x^k} u^n - \tfrac12 \tsum_{j,k} \partial_{x^n} A^{jk} \, \partial_{x^j} u^k - \tsum_j b^j \partial_{x^j} u^n - \tsum_j \partial_{x^n} b^j \, u^j\\ 
  & \quad - {\tsum_{j,k}} \frac{\partial f^0}{\partial z^{0j}} \,\sigma^{kj} \,\partial_{x^k} u^n - {\tsum_{j,k}} \frac{\partial f^0}{\partial z^{0j}} \, \partial_{x^n}\sigma^{kj} \, u^k \\
  & \quad - {\tsum_{i\neq 0, j,k}} \frac{\partial f^0}{\partial z^{ij}} \, \sigma^{kj} \partial^2_{x^n x^k} v^i - {\tsum_{i\neq 0, j,k}} \frac{\partial f^0}{\partial z^{ij}} \, \partial_{x^n}\sigma^{kj} \, \partial_{x^k} v^i.
 \end{split}
 \end{equation}
 Recall from Assumption \ref{ass:X} (i), $\boldsymbol b, \boldsymbol \sigma, \nabla \boldsymbol b, \nabla \boldsymbol \sigma \in L^\infty$.  Moreover, it follows from Assumption \ref{ass:BSDE} (iv) that 
 \[
  \Big|\frac{\partial f^0}{\partial z^{0j}} (\nabla \boldsymbol v \, \boldsymbol \sigma) \Big| \lec |\nabla \boldsymbol v| \quad \text{and} \quad \Big|\frac{\partial f^0}{\partial z^{ij}} (\nabla \boldsymbol v \, \boldsymbol \sigma)\Big| \lec |\Bu|, \quad i=1, \dots, I,\ j = 1, \dots, d,
 \]
 for a constant $C$ depending on $M$ in Assumption \ref{ass:BSDE} (iv) and $\|\boldsymbol \sigma\|_{L^\infty}$. Thanks to Corollary \ref{cor:vnm-reg} (ii) and Proposition \ref{prop:global-sobolev}, $\nabla \boldsymbol v \in L^\infty \cap L^{d+2}\cap L^2$ and $\nabla^2 \boldsymbol v \in L^{d+2}\cap L^2$, we confirm \eqref{VW-int} and \eqref{u-diff-bd} from \eqref{u-sys} and we further obtain $V\in (L^\infty +L^{d+2})((0,1)\times \R^d)$, { where the $L^{d+2}$ component comes from the first term in the last line of \eqref{u-sys}}. Now due to the fact that $W, \Bu \in L^\infty$ and $\Bu, \nabla \Bu,  \nabla^2 \boldsymbol v \in L^2$, the right-hand side of \eqref{u-sys} belongs to $L^2$. Then the Sobolev norm estimate for linear parabolic equations (see e.g. \cite[Chapter IV, Theorem 9.1]{oLadyzhenskaya1968})) implies that $\Bu\in W^{1,2}_2((0,1)\times \R^d)$. 
\end{proof}

\section{Additional proofs}\label{sec:add-proofs}
\subsection{Proof of Lemma \ref{Lem:port}}

We fix $i\in \{1,\ldots,I\}$. The finiteness of $\E[|U(\theta^i \cdot S + E^i)|]$ follows from the assumption in Definition \ref{def:radner_equilibrium} that the set $\mathcal{Q}$ is well defined. To prove the optimality, define the conjugate function of $U$ by
\[
V(y):= \sup_{x\in\R} \left\{ U(x) - xy \right\} = y(\ln y -1), \qquad y>0,
\]
and observe that
\[
V(U'(x)) = U(x) - xU'(x) \geq U(c) - cU'(x), \qquad c\in \R.
\]
In particular, letting $a^i:=\E[U'(\theta^i \cdot S + E^i)]$, we find that
\begin{multline*}
U\Big( \tint_0^1 \theta^i_t\ \Id S_t + E^i \Big) - a^i \frac{\Id \mathbb{Q}^i}{\Id \mathbb{P}} \Big( \int_0^1 \theta^i_t\ \Id S_t + E^i \Big)\geq  U\Big( \tint_0^1 \eta_t\ \Id S_t + E^i \Big) - a^i \frac{\Id \mathbb{Q}^i}{\Id \mathbb{P}} \Big( \tint_0^1 \eta_t\ \Id S_t + E^i \Big),
\end{multline*}
for all processes $\eta$. Taking the expectation under $\mathbb{P}$ on both sides of the inequality yields the result due to the fact that $\eta \cdot S$ is a supermartingale under $\mathbb{Q}^i$ . \qed

\subsection{Proof of Theorem \ref{Thm:char}}
Let $(\boldsymbol \theta,S)$ be a Radner equilibrium, $\mathcal{Q} = \{ \mathbb{Q}^i \}_{i=1,\ldots,I}$ the set of associated pricing measures and $\boldsymbol R$ the certainty equivalent process defined by \eqref{eq:cert_equiv}. For $i=1,\ldots,I$, we define the martingales
\begin{align*}
K^i_t &:= \E_t \Big[ \xi U'\Big( \tint_0^1 \theta^i_u\ \Id S_u + E^i  \Big) \Big] = \E_t \Big[ \xi e^{-\int_0^1 \theta^i_u\ \Id S_u - E^i} \Big],\\
L^i_t &:= \E_t \Big[U'\Big( \tint_0^1 \theta^i_u\ \Id S_u + E^i  \Big) \Big] = \E_t \Big[e^{-\int_0^1 \theta^i_u\ \Id S_u - E^i} \Big].
\end{align*}
Since $S$ is a martingale under every element of $\mathcal{Q}$ and from the representation $\Id \mathbb{Q}^i / \Id \mathbb{P} = L^i_1/L^i_0$, we obtain that
\begin{align*}
S_t &= \E_t^{\mathbb{Q}^i}[\xi] = \frac{K^i_t}{L^i_t},\\
R^i_t&= -\ln L^i_t - \int_0^t \theta^i_u\ \Id \left( \frac{K^i_u}{L^i_u} \right),
\end{align*}
which we may write as
\begin{equation}\label{eq:form_fwd}
\begin{aligned}
S_t &= S_0 + \int_0^t \Id \left( \frac{K^i_u}{L^i_u} \right), \\
R^i_t&= R^i_0 - \int_0^t \left( \Id \ln L^i_u + \theta^i_u\ \Id \left( \frac{K^i_u}{L^i_u} \right) \right).
\end{aligned}
\end{equation}

Because the filtration is generated by the Brownian motion $W$, therefore, any local martingale can be represented as a stochastic integral with respect to $W$. This and the fact that $L>0$ allows us to deduce that there exist processes $\boldsymbol \beta,\boldsymbol \eta \in \mathcal{H}^0(\mathbb{R}^{I\times d})$ such that the martingales $K$ and $L$ have the representation
\begin{align*}
K^i_t &= K^i_0 + \int_0^t L^i_u\boldsymbol \beta^i_u\ \Id W_u,\\
L^i_t &= L^i_0 + \int_0^t L^i_u\boldsymbol \eta^i_u\ \Id W_u.
\end{align*}
A simple application of It\^{o}'s formula now yields
\begin{align*}
\Id \left(\frac{K^i_t}{L^i_t}\right) &= -(\boldsymbol \beta^i_t - S_t \boldsymbol \eta^i_t) \boldsymbol (\eta^i_t)^\top\ \Id t + (\boldsymbol \beta^i_t - S_t \boldsymbol \eta^i_t)\ \Id W_t,\\
\Id \ln L^i_t &= -\frac{1}{2} |\boldsymbol \eta^i_t|^2\ \Id t+ \boldsymbol \eta^i_t\ \Id W_t
\end{align*}
and therefore the SDE
\[
S_t = S_0 - \int_0^t (\boldsymbol \beta^i_u - S_u\boldsymbol \eta^i_u) \boldsymbol (\eta^i_u)^\top\ \Id u + \int_0^t (\boldsymbol \beta^i_u - S_u\boldsymbol \eta^i_u)\ \Id W_u.
\]
Because $S$ solves each of the $I$ previous SDEs, there exists a row vector $\boldsymbol \zeta\in\mathcal{H}^0(\mathbb{R}^{d})$ and $\mu\in\mathcal{H}^0(\mathbb{R})$ such that, for every $i\in\{1,\ldots,I \}$,
\begin{align*}
\boldsymbol \zeta_t &= \boldsymbol \beta^i_t - S_t \boldsymbol \eta^i_t,\\
\mu_t &= - \boldsymbol \zeta_t(\boldsymbol \gamma^i_t + \theta^i_t\boldsymbol \zeta_t)^\top,
\end{align*}
where $-\boldsymbol\gamma^i := \boldsymbol \eta^i + \theta^i\boldsymbol \zeta$. Taking into account the terminal conditions $S_1 = \xi$ and $R^i_1 = E^i$, we may write \eqref{eq:form_fwd} in backward form: 
\begin{align*}
S_t &= \xi + \int_t^1 \mu_u\ \Id u - \int_t^1 \boldsymbol \zeta_u\ \Id W_u,\\
R^i_t &= E^i - \frac{1}{2}\int_t^1 |\boldsymbol \gamma^i_u|^2 - |\theta^i_u\boldsymbol \zeta_u|^2\ \Id u - \int_t^1 \boldsymbol \gamma^i_u\ \Id W_u.
\end{align*}
From the market clearing condition $\sum_i \alpha^i \theta^i = 1$ we deduce that
\begin{align*}
\mu_t &= - (\tsum_k\alpha^k \boldsymbol \gamma^k_t + \boldsymbol \zeta_t)\boldsymbol \zeta_t^\top,\\
\theta^i_t &= 1 + (\sum_k \alpha^k\boldsymbol \gamma^k_t - \boldsymbol \gamma^i_t ) \frac{\boldsymbol \zeta_t^\top}{|\boldsymbol \zeta_t |^2}, \quad \text{if } \boldsymbol \zeta_t\neq 0.
\end{align*}
A simple substitution now yields the BSDE formulation in the statement of the theorem.

It remains to consider the stochastic exponentials $\mathcal{Z}^i = \mathcal{E}(-(\boldsymbol \gamma^i + \theta^i\boldsymbol \zeta) \cdot W)$. Observe that
\begin{align*}
\mathcal{Z}^i_t & = e^{ -\frac{1}{2}\int_0^t |\boldsymbol \gamma^i_u + \theta^i_u\boldsymbol \zeta_u |^2\ \Id u - \int_0^t \boldsymbol \gamma^i_u + \theta^i_u\boldsymbol \zeta_u\ \Id W_u}\\
&=e^{-\frac{1}{2}\int_0^t |\boldsymbol \eta^i_u |^2\ \Id u + \int_0^t \boldsymbol \eta^i_u\ \Id W_u }\\
&=\frac{L^i_t}{L^i_0} = \E_t\left[ \frac{\Id \mathbb{Q}^i}{\Id \mathbb{P}} \right].
\end{align*}
Hence each $\mathcal{Z}^i$ is a $\mathbb{P}$-martingale. Since $S$ and $\theta^i\cdot S$ are martingales under $\mathbb{Q}^i$, it now follows that $\mathcal{Z}^iS$ and $\mathcal{Z}^i(\theta^i \cdot S)$ are martingales under $\mathbb{P}$.

Suppose $(S,\boldsymbol R,\boldsymbol \zeta,\boldsymbol \gamma)$ is a solution of the BSDE stated in the theorem. Let $\boldsymbol \theta$ be defined as in \eqref{eq:strat_equ} if $\boldsymbol \zeta \neq 0$, and take arbitrary values satisfying \eqref{eq:clearing} if $\boldsymbol \zeta =0$, and let $\mathcal{Z}^i = \mathcal{E}(-(\boldsymbol \gamma^i + \theta^i\boldsymbol \zeta)\cdot W)$ be the density process of a probability measure $\mathbb{Q}^i$. Then
\begin{align*}
\frac{\Id \mathbb{Q}^i}{\Id \mathbb{P}} = Z^i_1 &= e^{ -\frac{1}{2}\int_0^1 |\boldsymbol \gamma^i_t + \theta^i_t\boldsymbol \zeta_t |^2\ \Id t - \int_0^1 \boldsymbol \gamma^i_t + \theta^i_t\boldsymbol \zeta_t\ \Id W_t  }\\
&= e^{ -R^i_1 + R^i_0 - \int_0^1 \theta^i_t\ \Id S_t }\\
&= \frac{U'\left( \int_0^1 \theta^i_t\ \Id S_t  + E^i\right) }{\mathbb{E}\left[U'\left( \int_0^1 \theta^i_t\ \Id S_t  + E^i\right)\right]},
\end{align*}
where we deduced $e^{-R^i_0} = \E[e^{-R^i_1 - \theta^i \cdot S}]$ from the fact that $\E[\mathcal{Z}^i_1] = 1$. Hence $\boldsymbol{\mathcal{Z}}=(\mathcal{Z}^i)_{i=1,\ldots,I}$ defines the elements of the set of pricing measures $\mathcal{Q}$. We further observe that $S_1 = \xi$ and that, for every $i\in\{1,\ldots,I \}$, since $\mathcal{Z}^iS$ and $\mathcal{Z}^i(\theta^i\cdot S)$ are martingales under $\mathbb{P}$ and the density process $\mathcal{Z}^i$ is a martingale, it follows that $S$ and $\theta^i\cdot S$ are $\mathbb{Q}^i$-martingales. Given the expression \eqref{eq:strat_equ} for each $\theta^i$, one easily verifies that also the clearing condition \eqref{eq:clearing} holds. Lastly, undoing the computations in the `only if' part of the proof shows that $\boldsymbol R$ is indeed the certainty equivalent defined by \eqref{eq:cert_equiv}. \qed

\subsection{Proof of Theorem \ref{thm:equilibrium}}

We will check  that an invertible linear transformation of $\boldsymbol f$ in \eqref{eq:f}  satisfies Assumption \ref{ass:BSDE}.  Once this is done,  Theorem \ref{thm:BSDE-existence} implies the existence of a Markovian solution $\widetilde{\boldsymbol Y}_t = \widetilde{\boldsymbol v}(t, X_t)$ and $\widetilde{\boldsymbol Z}_t = (\nabla \widetilde{\boldsymbol v}  \boldsymbol \sigma)(t, X_t)$, $t\in [0,1]$ with $\widetilde{\boldsymbol Y}$ and $\widetilde{\boldsymbol Z}$ both bounded for the BSDE with the transformed generator $\widetilde{\boldsymbol f}$. Invert the linear transformation, same properties hold for $(\boldsymbol Y, \boldsymbol Z)$. In notation of \eqref{BSDE-sys}, $\boldsymbol \zeta = \boldsymbol Z^0$ and $\boldsymbol \gamma^i = \boldsymbol Z^i$, $i=1, \dots, I$. Recall $\theta^i$ from \eqref{eq:strat_equ},  we obtain
\begin{equation}\label{the-zeta}
\theta^i \boldsymbol \zeta = \boldsymbol \zeta + \tsum_k \alpha^k \boldsymbol \gamma^k - \boldsymbol \gamma^i.
\end{equation} 
Boundedness of $\boldsymbol \zeta$ and $\boldsymbol \gamma^i$ imply  that  $\theta^i \boldsymbol \zeta$ is bounded no matter whether $\boldsymbol \zeta =0$ or not. Therefore the stochastic exponential $\mathcal{Z}^i = \mathcal{E} (- (\boldsymbol \gamma^i + \theta^i \boldsymbol \zeta) \cdot W)$ defines a probability measure $\mathbb{Q}^i$ via $d\mathbb{Q}^i/ d\mathbb{P} |_{\mathcal{F}_1} = \mathcal{Z}^i_1$. It then follows from the first equation of \eqref{BSDE-sys} and \eqref{the-zeta} that 
\[
 \Id S_t = \boldsymbol \zeta_t \ \Id W^{i}_t, 
\]
where $dW^{i}_t = dW_t + (\sum_k \alpha^k \boldsymbol \gamma^k_t + \boldsymbol \zeta_t) \Id t$ defines a $\mathbb{Q}^i$-Brownian motion $W^i$. Thanks to the boundedness of $\boldsymbol \zeta$, $S$ is a $\mathbb{Q}^i$-martingale. Moreover,
\[
 \theta^i \ \Id S_t = \theta^i \boldsymbol \zeta_t \ \Id W^{i}_t.
\]
Therefore boundedness of $\theta^i \boldsymbol \zeta$ implies the $\mathbb{Q}^i$-martingale property of $\theta^i \cdot S$. Now, the `only if' statement of Theorem \ref{Thm:char} implies that $(\boldsymbol \theta, S)$ is a Radner equilibrium. 

Come back to the linear transformation of $\boldsymbol f$. Introduce $\widetilde{\boldsymbol f}$ whose components are   
\[
 \widetilde{f}^0(\widetilde{\boldsymbol z}) = f^0(\boldsymbol z), \quad \widetilde{f}^i(\widetilde{\boldsymbol z})  = (f^i - f^I) (\boldsymbol z), \ i=1, \dots, I-1, \quad \widetilde{f}^I(\widetilde{\boldsymbol z}) = \widetilde{f}^I(\boldsymbol z),
\]
\[
 \widetilde{\boldsymbol z}^0 = \boldsymbol z^0, \quad  \widetilde{\boldsymbol z}^i = \boldsymbol z^i - \boldsymbol z^I, \ i=1,\dots, I-1, \quad \widetilde{\boldsymbol z}^I = \boldsymbol z^I.
\]
Assumption \ref{ass:BSDE} (i) is clearly satisfied. To check (ii), observe that 
\begin{align*}
 \widetilde{f}^0(\widetilde{\boldsymbol z}) = & \widetilde{\boldsymbol z}^0 \ell^0(\widetilde{\boldsymbol z}),\\
  \widetilde{f}^i(\widetilde{\boldsymbol z}) = & f^i(\boldsymbol z) - f^I(\boldsymbol z) = \widetilde{\boldsymbol z}^i \ell^i(\widetilde{\boldsymbol z}) - \tfrac12 \widetilde{\boldsymbol z}^i (\widetilde{\boldsymbol z}^i + 2 \widetilde{\boldsymbol z}^I), \quad i=1, \dots, I-1,
\end{align*}
where 
\begin{align*}
\ell^0(\widetilde{\boldsymbol z}) = & - \big(\tsum_k \alpha^k (\widetilde{\boldsymbol z}^k + \widetilde{\boldsymbol z}^I) + \widetilde{\boldsymbol z}^0\big)^\top,\\
\ell^i(\widetilde{\boldsymbol z}) = & -\tfrac12 \tfrac{(\widetilde{\boldsymbol z}^0)^\top}{|\widetilde{\boldsymbol z}^0|} \big(2 \widetilde{\boldsymbol z}^0 +  2 \tsum_k \alpha^k (\widetilde{\boldsymbol z}^k +  \widetilde{\boldsymbol z}^I) - \widetilde{\boldsymbol z}^i - 2 \widetilde{\boldsymbol z}^I\big) \tfrac{(\widetilde{\boldsymbol z}^0)^\top}{|\widetilde{\boldsymbol z}^0|} 1_{\{|\widetilde{\boldsymbol z}^0 \neq 0|\}}, \quad i=1, \dots, I-1,
\end{align*}
are both at most linear growth function of $\widetilde{\boldsymbol z}$. Meanwhile $\widetilde{f}^I(\widetilde{\boldsymbol z})$ is at most quadratic growth function of $\widetilde{\boldsymbol z}$. Therefore \eqref{BF} is satisfied by $\widetilde{\boldsymbol f}$. To verify (iii), note that 
\begin{align*}
 - f^0(\boldsymbol z) = &-\boldsymbol z^0  \ell^0(\boldsymbol z),\\
 -f^i(\boldsymbol z) \leq & \tfrac12 |\boldsymbol z^i|^2, \quad i=1, \dots, I, 
\end{align*}
Moverover 
\begin{align*}
 &f^0(\boldsymbol z) + \tsum_i \alpha^i f^i(\boldsymbol z) \\
 &=-  \tfrac12 \frac{(f^0(\boldsymbol z))^2}{|\boldsymbol z^0|^2} 1_{\{\boldsymbol z^0 \neq 0\}} +  \frac{(f^0(\boldsymbol z))^2}{|\boldsymbol z^0|^2} 1_{\{\boldsymbol z^0 \neq 0\}}   + \frac{f^0(\boldsymbol z)}{|\boldsymbol z^0|^2} \big(\tsum_i \alpha^i \boldsymbol z^i \big) (\boldsymbol z^0)^\top 1_{\{\boldsymbol z^0 \neq 0\}}+ f^0(\boldsymbol z) \\
 & \quad + \tfrac12 \tsum_i \alpha^i \Big(\frac{\boldsymbol z^i (\boldsymbol z^0)^\top}{|\boldsymbol z^0|}\Big)^2 1_{\{\boldsymbol z^0 \neq 0\}} - \tfrac12 \tsum_i \alpha^i |\boldsymbol z^i|^2\\
 & \leq \text{I} + \text{II},
\end{align*}
where, due to $f^0(\boldsymbol z) =0$ when $\boldsymbol z^0 =0$, 
\begin{align*}
 \text{I} &= \frac{(f^0(\boldsymbol z))^2}{|\boldsymbol z^0|^2} 1_{\{\boldsymbol z^0 \neq 0\}}   + \frac{f^0(\boldsymbol z)}{|\boldsymbol z^0|^2} \big(\tsum_i \alpha^i \boldsymbol z^i \big) (\boldsymbol z^0)^\top 1_{\{\boldsymbol z^0 \neq 0\}}+ f^0(\boldsymbol z) = - f^0(\boldsymbol z) 1_{\{\boldsymbol z^0 \neq 0\}} + f^0(\boldsymbol z) =0,\\
 \text{II} &=  \tfrac12 \tsum_i \alpha^i \Big(\frac{\boldsymbol z^i (\boldsymbol z^0)^\top}{|\boldsymbol z^0|}\Big)^2 1_{\{\boldsymbol z^0 \neq 0\}} - \tfrac12 \tsum_i \alpha^i |\boldsymbol z^i|^2 \leq 0.
\end{align*}
The two estimates above combined implies $f^0(\boldsymbol z) + \sum_i \alpha^i f^i(\boldsymbol z)\leq 0$. Therefore the previous estimates imply that $\boldsymbol f$ satisfies \eqref{wAB} with a positively spanning vectors in $\R^{I+1}$:
\[
 \boldsymbol a_i = - e_i, \,i=1, \dots, I+1, \quad \text{and} \quad \boldsymbol a_{I+2} = (1, \alpha^1, \dots, \alpha^I),
\]
where $(e_i)$ are the standard basis vectors of $\R^I$. Recall that the positively spanning property is remains after invertible linear transformation (see. e.g. \cite[Remark 2.13]{hXing2018}). Therefore $\widetilde{\boldsymbol f}$ satisfies \eqref{wAB} as well. Finally, from the specific form of $f^0(\boldsymbol z)$, we can verify Assumption \ref{ass:BSDE} (iv) as well. In conclusion, $\widetilde{\boldsymbol f}$ satisfies all conditions in Assumption \ref{ass:BSDE}. \qed

\subsection{Proof of Theorem \ref{thm:BSDE-existence}}
 Corollary \ref{cor:vnm-reg} (ii) shows that the H\"{o}lder and $L^\infty$-norms of $\nabla \boldsymbol v_n$ are bounded uniformly in $n$. Then by Arel\'{a}-Ascoli theorem, $\nabla \boldsymbol v_n$ converges to $\nabla \boldsymbol v$ local uniformly. On the other hand, Lemma \ref{lem:Bu-sys} and Theorem \ref{thm:BU} imply that,  if there is a measurable set in $[0,1)\times \mathbb{R}^d$ with positive Lebesgue measure such that $|\nabla v^0|=0$ there, then $\nabla v^0 \equiv 0$ on $[0,1]\times \mathbb{R}^d$. However, this is contradicts with $|\nabla g^0(x_0)|\neq 0$ in Assumption \ref{ass:X} (iv). Therefore, we confirm \eqref{non-deg}. As a result, the convergence \eqref{f-conv} follows from the local uniform convergence of $\nabla \boldsymbol v_n$ to $\nabla \boldsymbol v$ and the local uniform convergence of $\boldsymbol f_n$ to $\boldsymbol f$ in Lemma \ref{lem:reg} (iii). We have seen from Corollary \ref{cor:vnm-reg} (iii) that each $\boldsymbol v_n$ satisfies \eqref{vnm-sys}. Sending $n\rightarrow \infty$, we obtain that $\boldsymbol v$ solves \eqref{v-pde-sys} in weak sense, hence almost everywhere. It follows Corollary \ref{cor:vnm-reg} (ii) and Proposition \ref{prop:global-sobolev} that $\boldsymbol v\in L^\infty \cap W^{1,2}_2 \cap W^{1,2}_{d+2} ((0,1)\times \mathbb{R}^d)$ and $\nabla \boldsymbol v \in L^\infty ((0,1)\times \mathbb{R}^d)$. Finally, $(\boldsymbol Y_t, \boldsymbol Z_t)= (\boldsymbol v, \nabla \boldsymbol v \boldsymbol \sigma)(t,X_t)$ solves the BSDE \eqref{eq:P} thanks to Krylov's It\^{o} formula in \cite[Chapter 2, Section 10, Theorem 1]{nKrylov1980}. \qed
 
 {
 \section{Future research} \label{sec:future}
 This paper invites investigations in several challenging future research topics: First, due to the terminal consumption nature of our setting, the interest rate is not endogenously determined. When agents also consume intertemporally and receive flow of random endowments, the interest rate needs to be determined endogenously. The case of linear stock dividend and random endowment is considered in  \cite{Lar12}, \cite{Lar14}, where the equilibrium interest rate is deterministic. Going beyond the linear setting requires studying stochastic interest rate, which is challenging in a setting with CARA utility agents. Second, when there are multiple risky assets, their volatility $\boldsymbol \zeta$ becomes a matrix, rather than a vector in the single risky asset case. The mean-variance component in agents' optimal position requires $\boldsymbol \zeta \boldsymbol \zeta'$ to be invertible. This is the same research question as in the endogenously complete dynamic equilibria. However, rather than a system of linear PDEs, incomplete equilibrium demands analyzing a system of highly nonlinear PDE, rendering the non-degeneracy property of $\boldsymbol \zeta \boldsymbol \zeta'$ highly non-trivial to understand. Lastly, we hope our self-contained backward uniqueness result Theorem \ref{thm:BU} could provide researcher a tool to tackle other BSDEs or control problems where the degeneracy of $Z$ or the control variable naturally appear.
}
\bibliographystyle{alpha}
\bibliography{bib-existence}       


\end{document}